\documentclass[11pt,reqno]{amsart}
\usepackage[colorlinks,citecolor=red,pagebackref,hypertexnames=false]{hyperref}

\makeatletter
\def\@tocline#1#2#3#4#5#6#7{\relax
  \ifnum #1>\c@tocdepth 
  \else
    \par \addpenalty\@secpenalty\addvspace{#2}%
    \begingroup \hyphenpenalty\@M
    \@ifempty{#4}{%
      \@tempdima\csname r@tocindent\number#1\endcsname\relax
    }{%
      \@tempdima#4\relax
    }%
    \parindent\z@ \leftskip#3\relax \advance\leftskip\@tempdima\relax
    \rightskip\@pnumwidth plus4em \parfillskip-\@pnumwidth
    #5\leavevmode\hskip-\@tempdima
      \ifcase #1
       \or\or \hskip 1em \or \hskip 2em \else \hskip 3em \fi%
      #6\nobreak\relax
    \dotfill\hbox to\@pnumwidth{\@tocpagenum{#7}}\par
    \nobreak
    \endgroup
  \fi}
\makeatother

\usepackage{tikz,amsthm,amsmath,amstext,amssymb,amscd,epsfig,euscript, mathrsfs, dsfont,pspicture,multicol,graphpap,graphics,graphicx,times,enumerate,subfig,sidecap,wrapfig,color,pict2e}

\usepackage{ hyperref}
\usepackage{enumerate}

 \numberwithin{equation}{section}

\def\bB{{\mathbb{B}}}
\def\bC{{\mathbb{C}}}

\def\bR{{\mathbb{R}}}
\def\bS{{\mathbb{S}}}

\def\bZ{{\mathbb{Z}}}

\def\bN{{\mathbb{N}}}

\def\Stop{{\rm Stop}}
\def\Child{{\rm Child}}

\def\cB{{\mathscr{B}}}
\def\cC{{\mathscr{C}}}
\def\cD{{\mathscr{D}}}

\def\cF{{\mathscr{F}}}
\def\cG{{\mathscr{G}}}
\def\cH{{\mathscr{H}}}
\def\cI{{\mathscr{I}}}

\def\Layer{{\rm Layer}}

\def\cN{{\mathscr{N}}}

\def\cS{{\mathscr{S}}}
\def\Up{{\rm Up}}

\def\cX{{\mathscr{X}}}

\def\one{\mathds{1}}

\def\largetheta{\Theta_}

\def\ve{\epsilon} 

\renewcommand{\d}{{\partial}}
\def\lec{\lesssim}
\def\gec{\gtrsim}

\def\spn{\mathop\mathrm{span}} 	
\DeclareMathOperator{\diam}{diam}
\def\dist{\mathop\mathrm{dist}} 						
\def\supp{\mathop\mathrm{supp}}					

\newcommand{\ps}[1]{\left( #1 \right)}

\newcommand{\ck}[1]{\left\{#1 \right\}}
\newcommand{\av}[1]{\left| #1 \right|}

\newcommand{\cnj}[1]{\overline{#1}}

\def\Xint#1{\mathchoice
{\XXint\displaystyle\textstyle{#1}}%
{\XXint\textstyle\scriptstyle{#1}}%
{\XXint\scriptstyle\scriptscriptstyle{#1}}%
{\XXint\scriptscriptstyle\scriptscriptstyle{#1}}%
\!\int}
\def\XXint#1#2#3{{\setbox0=\hbox{$#1{#2#3}{\int}$ }
\vcenter{\hbox{$#2#3$ }}\kern-.58\wd0}}

\def\avint{\Xint-}
\def\grad{\nabla}

\theoremstyle{plain}
\newtheorem{theorem}{Theorem}

\newtheorem{corollary}[theorem]{Corollary}

\newtheorem{lemma}[theorem]{Lemma}

\newtheorem{proposition}[theorem]{Proposition}

\newtheorem{maintheorem}{Theorem}
\newtheorem{maincor}[maintheorem]{Corollary}

\theoremstyle{definition}

\newtheorem{example}[theorem]{Example}
\newtheorem{definition}[theorem]{Definition}
\newtheorem{remark}[theorem]{Remark}

\numberwithin{equation}{section}
\numberwithin{theorem}{section}

\newcommand\eqn[1]{\eqref{e:#1}}

\newcommand\Theorem[1]{Theorem \ref{t:#1}}
\newcommand\Lemma[1]{Lemma \ref{l:#1}}

  \DeclareFontFamily{U}{mathb}{\hyphenchar\font45} 
\DeclareFontShape{U}{mathb}{m}{n}{
      <5> <6> <7> <8> <9> <10> gen * mathb
      <10.95> mathb10 <12> <14.4> <17.28> <20.74> <24.88> mathb12
      }{}
\DeclareSymbolFont{mathb}{U}{mathb}{m}{n}

\DeclareMathSymbol{\toitself}      {3}{mathb}{"FD}  

\newcommand{\RS}[1]{{  }}
\newcommand{\JA}[1]{{  }}

\newcommand{\postRef}[1]{#1}
  \newcommand{\preRef}[1]{}

\begin{document}

\title[A Traveling Salesman Theorem]{An Analyst's Traveling Salesman Theorem for
sets of dimension larger than one}
\author{Jonas Azzam}
\address{School of Mathematics, University of Edinburgh, JCMB, Kings Buildings,
Mayfield Road, Edinburgh,
EH9 3JZ, Scotland.}
\email{j.azzam ``at" ed.ac.uk}
\author{Raanan Schul}
\address{Department of Mathematics, Stony Brook University, Stony Brook, NY 11794-3651, USA}
\email{schul ``at" math.stonybrook.edu}
\thanks{
J.~Azzam  was supported by the ERC grant 320501 of the European Research Council (FP7/2007-2013).
R.~ Schul was partially supported by NSF  DMS 1361473.}

\maketitle

\begin{abstract}

In his 1990 Inventiones paper,  P. Jones characterized subsets of rectifiable curves in the plane via a multiscale sum of $\beta$-numbers. These $\beta$-numbers are geometric quantities measuring how far a given set deviates from a best fitting line at each scale and location.  
Jones' result is a quantitative way of saying that a curve is rectifiable if and only if it has a tangent at almost every point. 
Moreover, computing this square sum for a curve returns the length of the curve up to multiplicative constant. 
K. Okikiolu extended his result from subsets of the plane to subsets of Euclidean space. 
G. David and S. Semmes extended the discussion to include sets of (integer) dimension larger than one, under the assumption of Ahlfors regularity and using a variant of Jones' $\beta$-numbers. This variant has since been used by others to give structure theorems for rectifiable sets and to give upper bounds for the measure of a set. 

In this paper we give a version of P. Jones' theorem for sets of arbitrary (integer) dimension lying in Euclidean space.
Our main result is a lower bound for  the $d$-dimensional Hausdorff measure of a set in terms of an analogous sum of $\beta$-type numbers. We also show an upper bound of this type. The combination of these results gives a Jones theorem for higher dimensional sets.
\preRef{There is no assumption of Ahlfors regularity, or of a measure on the underlying set, but rather, only of a lower bound on the Hausdorff content.}
 \postRef{While there is no assumption of Ahlfors regularity, or of a measure on the underlying set, there is  an assumption of a lower bound on the Hausdorff content.}
 We adapt David and Semmes' version of Jones' $\beta$-numbers by redefining them using a Choquet integral, {  allowing them to be defined for arbitrary sets (and not just sets of locally finite measure).} A key tool in the proof is G. David and T. Toro's parametrization of Reifenberg flat sets (with holes).

\end{abstract}

\tableofcontents

\section{Introduction}

\subsection{Background}

We will begin by recalling the {\it Analyst's Traveling Salesman Theorem}. 
{ 
It will serve as a model for the two main results in this paper.}
For  sets $E,B\subseteq \bR^{n}$,  define
\begin{equation}
\beta_{E,\infty}^{d}(B)=\frac2{\diam(B)}\inf\limits_L\sup\{\dist(y,L):y\in E\cap B\}
\label{e:euclidean-beta}
\end{equation}
where $L$ ranges over  $d$-planes in $\bR^{n}$.
Thus, $\beta_{E,\infty}^{d}(B)\diam(B)$ is the width of the smallest tube containing $E\cap B$. We will typically have $B$ be a ball or cube. We will denote by  $\Delta$  the collection of dyadic cubes in $\bR^{n}$ (see the beginning of Section \ref{s:preliminaries}).

{

\begin{theorem}(Jones: $\bR^2$  \cite{Jones-TSP}; Okikiolu: $\bR^n$  \cite{O-TSP}) Let  $n\geq 2$.
There is a $C=C(n)$ such that the following holds.
Let $E\subset \bR^n$.
Then there is  a connected set $\Gamma\supseteq E$ such that
\begin{equation}
\cH^{1}(\Gamma)\lec_{n} \diam E+\sum_{Q\in \Delta\atop Q\cap E\neq\emptyset} \beta_{E,\infty}^{1}(3Q)^{2}\diam(Q).
\label{e:betaE}
\end{equation}
Conversely, if $\Gamma$ is connected and $\cH^{1}(\Gamma)<\infty$, then 
\begin{equation}
 \diam \Gamma+\sum_{Q\in \Delta\atop Q\cap \Gamma\neq\emptyset} \beta_{\Gamma,\infty}^{1}(3Q)^{2}\diam(Q) \lec_{n}  \cH^{1}(\Gamma).
\label{e:beta_gamma}
\end{equation}
\label{t:TST}
\end{theorem}
Here, $\cH^k$ is the $k-$dimensional Hausdorff measure.
Given two functions $a$ and $b$ into $\bR$ we say 
$a\lesssim b$ with constant $C$, when there exists a constant $C=C_{a,b}$ such that $a\leq C b$. We say that $a\sim b$ if  $a\lesssim b$ and $b\lesssim a$.

We remark that a version of Theorem \ref{t:TST} holds for $E$ in an infinite dimensional Hilbert spaces \cite{Schul-TSP}, but that requires switching from $\Delta$ to a multi-resolution of balls centered on $E$. We will revisit this point in Section \ref{s:state-flat-case}.
}

{ 
\begin{remark}{
We remark for the expert that the main result of this paper is Theorem \ref{t:thmii-dyadic}, and is  a version of  \eqref{e:beta_gamma}. We also prove Theorem \ref{t:thmiii-dyadic}, a  version of  \eqref{e:betaE}, though the proof of  Theorem \ref{t:thmii-dyadic}  is about six times longer than that of Theorem \ref{t:thmiii-dyadic}.} 

\end{remark}
 }

%
%

Intrinsically, both implications of the \Theorem{TST} are interesting: a sufficient condition for the existence of a short curve is very useful, as well as the ability to quantify how non-flat a curve can be at various scales and locations. The result has applications to harmonic measure \cite{BJ90}, Kleinian groups \cite{BJ97}, analytic capacity \cite{Tol05}, and brownian motion \cite{BJPY97}. 

A $d$-dimensional analogue of the second half of Theorem \ref{t:TST} has been known to be false since the 90's: in Fang's thesis \cite{Fang-Jones-example} he gives an example of a Lipschitz graph where this sum is infinite. David and Semmes, however, realized that one could still develop a theory of $\beta$-numbers if instead one adjusted the definition of a $\beta$-number. We will state a theorem by them below (Theorem \ref{t:DS}). 
Their inspiration seems like it could come  from the following theorem of Dorronsoro.
\begin{theorem}[Dorronsoro, \cite{Dorronsoro}, Theorem 6]
Let $1\leq p<p(d)$ where 

\begin{equation}
\label{e:pd}
p(d):= \left\{ \begin{array}{cl} \frac{2d}{d-2} & \mbox{if } d>2 \\
 \infty & \mbox{if } d\leq 2\end{array}\right. . 
 \end{equation}
 For $x\in \bR^{d}$, $r>0$, and $f\in W^{1,2}(\bR^{d})$, define
\[\Omega_{f,p}(x,r)=\inf_{A} \ps{\avint_{B(x,r)} \ps{\frac{|f-A|}{r}}^{p}}^{\frac{1}{p}}\]
where the infimum is over all affine maps $A:\bR^{d}\rightarrow \bR$. Set
\[\Omega_{p}(f):=\int_{\bR^{d}}\int_{0}^{\infty}\Omega_{f,p}(x,r)^{2}\frac{dr}{r} dx. \]
Then
\[
\Omega_{p}(f)\lec_{d,p} ||\grad f||_{2}^{2}.\]
\label{t:dorronsoro}
\end{theorem}

The result is actually much stronger than stated here.
We refer the reader to \cite{Dorronsoro} for more details.

Define now the following $\beta$-number analogue: Given a closed ball $B \subset \bR^d$ with radius $r_{B}$, a measure $\mu$, an integer $0<d<D$, and $1\leq p <\infty$, let
\[\beta_{\mu,p}^d(B,L) =\left(\frac1{r_{B}^d} \int_{B} \left(\frac{\dist(y,L)}{r_{B}}\right)^p\,d\mu(y)\right)^{1/p}\]
and  
\[\beta_{\mu,p}^d(B)=\inf \{\beta_{\mu,p}^d(B,L) : L\mbox{ is a d-plane in }\bR^{n}\}.\]

When $\mu=\cH^{d}_{E}$ for some set $E$, we will write $\beta_{E,p}^{d}$ instead of $\beta_{\cH^{d}|_{E},p}^{d}$.
Recall $\sigma$ is a Carleson measure on $E\times (0,\infty)$ if $\sigma(B(x,r)\times (0,r))\leq Cr^{d}$.
\begin{theorem}[David, Semmes \cite{DS}]
Let $E\subseteq \bR^{n}$ be an {\it Ahlfors $d$-regular} set $E$, meaning there is $A>0$ so that
\begin{equation}\label{e:AR}
  r^{d}/A\leq \cH^{d}(B(\xi,r)\cap E) \leq Ar^{d} \mbox{ for }\xi\in E, r\in (0,\diam E).
\end{equation}
Then the following are equivalent:
\begin{enumerate}
\item  The set {\it $E$ has big pieces of Lipschitz images}, i.e. there are constants $L,c>0$ so  for all $\xi\in E$ and $r\in (0,\diam E)$, there is an $L$-Lipschitz map $f:\bR^{d}\rightarrow \bR^{n}$ such that $\cH^{d}(f(\bR^{d})\cap B(\xi,r))\geq cr^{d}$.
\item For $1\leq p<p(d)$, $\beta_{{E},1}^d(B(x,r))^{2}\frac{dxdr}{r}$ is a Carleson measure on $E\times (0,\infty)$. 
  \end{enumerate}
  \label{t:DS}
  \end{theorem}

 A closed set $E\subseteq \bR^{n}$ is said to be {\it $d$-uniformly rectifiable} if it is $d$-Ahlfors regular and has big pieces of Lipschitz images. We remark that constants coming out of this theorem (and similar theorems) depend on the constant of Ahlfors-regularity (denoted above by $A$).
 
 The motivation behind this result (and in fact also Jones' original motivation of \Theorem{TST}) was the study of singular integrals. In \cite{DS}, David and Semmes show that the two conditions in \Theorem{DS} are in fact equivalent to five other characterizations, one of them being that all odd singular integrals whose kernels are reasonably behaved are bounded operators. In fact, $\beta_{\mu,p}^d$ has been more amenable in applications to singular integrals even in the plane, see for example \cite{MMM96,Dav98,Leg99,AT15,GT16} and the references therein. 
 
We remark that  using the work of David and Toro \cite{DT12}, for certain kinds of sets $E$, one can obtain a higher dimensional analogue of the first half of \Theorem{TST}---that is, sufficient conditions in terms of $\beta_{E,\infty}^{d}(3Q)^{2}\diam (Q)^{d}$ for when a set can be contained in a $d$-dimensional surface of finite area (see Theorem \ref{t:DT} and \ref{DTtheorem2.5} below). 

{
The quantity $\beta_2$ is closely related to the Singular Value Decomposition of a matrix. 
See  \cite{LW12, ACM12, LMR16} and references therein for applications to data analysis.
 }

Recently, the first author and Tolsa have given a characterization of rectifiability for sets of finite measure beyond the Ahlfors regular category. 

\begin{theorem}[Azzam, Tolsa \cite{AT15}; Tolsa \cite{Tol15})]\label{t:AT15}
Let $E\subseteq \bR^{n}$ be a Borel set with $0<\cH^{d}(E)<\infty$ and $\mu\ll \cH^{d}|_{E}$.
\begin{enumerate}
\item If 
\begin{equation}\label{jonesfunction}
\int_{0}^{1} \beta_{\mu ,2}^d(B(x,r))^{2}\frac{dr}{r}<\infty \;\; \mbox{ for $\cH^{d}$-a.e. $x\in E$,}
\end{equation} 
then $E$ is $d$-rectifiable. 
\item If $E$ is $d$-rectifiable and $p\in [1,2]$, then $\int_{0}^{1} \beta_{\mu,p}^d(B(x,r))^{2}\frac{dr}{r}<\infty$ for $\cH^{d}$-a.e. $x\in E$.
\end{enumerate}
\label{t:AT}
\end{theorem}
The first part of this result was first shown by Pajot under some stronger assumptions (see also  \cite{BS16}).   
In \cite{BS-preprint-1}, Badger and the second author gave a similar characterization for general measures on Euclidean space (see also \cite{BS15}), but with $d=1$.

{ 
{\bf Addendum:}(added 2017/02/15) Part (1) of the above result was very recently improved by Edelen, Naber, and Valtorta in \cite{ENV16}. In their paper, they  give perhaps the most general results to date about how well  the size of $\mu$  is bounded {\it from above}  by the $\beta_{\mu,p}^d$-numbers and describes the support. As a corollary of their work, they show that if $\mu$ is any Radon measure such that \eqref{jonesfunction} holds and 
\[\limsup_{r\rightarrow 0} \mu(B(x,r))/r^{d}>0 \mbox{ for $\mu$-a.e. $x$},\] 
then $\mu$ can be covered by Lipschitz images of $\mathbb{R}^{d}$. This implies teh first part of  \Theorem{AT}.
See also \cite[Sections 3-6]{NV17} for similar results and their connection to singular sets of nonlinear PDEs.  Their results form an analogue of the \eqref{e:betaE}  half of  \Theorem{TST} in higher dimension.  

By comparison, 
our main result, Theorem \ref{t:thmii-dyadic} below, says that for a class of sets $E$, we can obtain {\it lower bounds} for the Hausdorff measure of $E$ in terms of  content-$\beta$-numbers  (see Definition \ref{d:beta-h-def}), that is, a version of the \eqref{e:beta_gamma} half of \Theorem{TST}.  
This is,  to our knowledge, the first time this has been  achieved for sets of dimension larger than one (even for Hausdorff measure, let alone general measures). Additionally, in Theorem \ref{t:thmiii-dyadic} below we can also obtain {\it upper bounds} on the Hausdorff measure of a set $E$ in terms of its $\beta$-numbers (that is, a version of the  \eqref{e:betaE} half of \Theorem{TST}). The proof of Theorem \ref{t:thmiii-dyadic} is much less substantial than Theorem \ref{t:thmii-dyadic}, but the result is interesting for a few reasons: firstly, while the works of \cite{ENV16} and \cite{NV17} obtain upper bounds on more general measures $\mu$ than just the Hausdorff measure, we do not assume the existence of a locally finite measure on our set, and in particular, we don't assume our set $E$ has finite Hausdorff measure a priori. Moreover,  the $\beta$-numbers we use don't require a finite measure in order to be defined like $\beta_{\mu,p}^d$ does, and thus we have a measure-independent and purely geometric method for testing whether a set has finite Hausdorff measure. Finally, the combination of Theorems \ref{t:thmii-dyadic} and \ref{t:thmiii-dyadic} gives an analogue  to the full  \Theorem{TST} for higher dimensional objects.


}

\subsection{Main Results}

Before stating our main results, we first define our $\beta$-number that, firstly, is in some sense an $L^{p}$-average of distances to a plane (rather than an $L^{\infty}$-norm), and secondly, doesn't rely on the underlying measure on the set in question. This seems somewhat self-contradictory, but we are able to achieve this by manipulating the definition of $\beta_{\cH^{d}|_{E},p}^{d}$ and ``integrate" with respect to $\cH^{d}_{\infty}$ (rather than $\cH^{d}$) using a Choquet integral.

\begin{definition}\label{d:beta-h-def}
For  arbitrary  sets $E$, and $B$, a $d$-dimensional plane $L$, define

\[
\beta_{E}^{d,p}(B,L)= \ps{\frac{1}{r_{B}^{d}}\int_{0}^{1}\cH^{d}_{\infty}(\{x\in B\cap E: \dist(x,L)>t r_{B}\})t^{p-1}dt}^{\frac{1}{p}}\]
where $2r_B=\diam(B)$, and set 
\[
\beta_{E}^{d,p}(B)=\inf\{ \beta_{E}^{d,p}(B,L): L\mbox{ is a $d$-dimensional plane in $\bR^{n}$}\}.\]
We will typically take $B$ to be a ball or cube.
\end{definition}

If we assume $E$ is Ahlfors $d$-regular, then this quantity is comparable to  $\beta_{\cH^{d}|_{E},p}^{d}$. Indeed,
for a ball $B$ centered on $E$,
\begin{align*}
\beta_{\cH^{d}|_{E},p}^{d}(B,L)^{p}
&\sim  \frac1{r_{B}^d} \int_{B} \ps{\frac{\dist(y,L)}{r_{B}}}^{p} d\cH^{d}(y) \\
& =\frac{1}{r_{B}^{d}}\int_{0}^{\infty}\cH^{d}(\{x\in B: \dist(x,L)>t r_{B}\})t^{p-1}dt
\end{align*}
and it is not hard to check that this is comparable to $\beta_{E}^{d,p}(B,L)$ using the Ahlfors regularity of $E$. 


{
\begin{definition}\label{d:dxref}
Given  two closed sets $E$ and $F$, and $B$ a set we denote
\[
d_{B}(E,F)=\frac{2}{\diam B}\max\left\{\sup_{y\in E\cap B}\dist(y,F), \sup_{y\in F\cap  B}\dist(y,E)\right\}\]
and
\[\vartheta_{E}^{d}(B)=\inf\{d_{B}(E,L): L\mbox{ is a $d$-dimensional plane in $\bR^{n}$}\}.\]
We will typically have $B$ be a ball or cube.
If $B=B(x,r)$, we will write $d_{B}=d_{x,r}$. We say $E$ is $(\ve,d)$-Reifenberg flat (or just $\ve$-Reifenberg flat when the dimension is given) if 
\[
\vartheta_{E}^{d}(B(x,r))<\ve \mbox{ for all $x\in E$ and $r>0$}.\] 
\end{definition}

}
%
%
%
%
%
%
%
%
%

{
\begin{definition}\label{d:def-theta-delta}
Let $E,B\subseteq \bR^{n}$ and $\ve,r>0$. We define 
\[
\largetheta{E}^{d,\Delta,\ve}(B):=\sum \{\diam(Q)^{d}:  Q\in \Delta,\ Q\cap E\cap B\neq\emptyset, \;\; \ell(Q)\leq r_B,  \mbox{ and } \vartheta_{E}(3Q)\geq \ve \}
.\]
The $\Delta$ in the superscript denotes that this is a quantity involving dyadic cubes.

\end{definition}
}

\begin{remark}
If  $E$ is an $\ve$-Reifenberg flat then $\largetheta{E}^{d,\Delta,\ve}=0$.
\end{remark}

\begin{definition}
A set $E\subseteq \bR^{n}$ is said to be {\it $(c,d)$-lower content regular} in a ball $B$ if 
\[
\cH^{d}_{\infty}(E\cap B(x,r))\geq cr^{d} \;\; \mbox{for all $x\in E\cap B$ and $r\in (0,r_{B})$}.\]
\end{definition}

\postRef{
\begin{remark}
Some examples of sets   $E\subseteq \bR^{n}$ which are  {\it $(c,d)$-lower content regular} for some $c>0$ and integer $d$ are: 
Reifenberg flat sets (see \ref{d:dxref}) and sets satisfying Condition B (see definition \ref{d:condition-b}).  One may also remove subsets from such $E$ in a controlled way to keep lower content regularity, with a smaller $c>0$.
\end{remark}
}
\RS{can the above be improved:  take any set with theta satisfying 
$$\largetheta{E}^{d,\Delta,\ve}(B(x,r))\lesssim r^d$$ and show it is lower content regular?\\
}
\JA{The lower content regularity also comes into the estimates where we pass from beta numbers on the set to beta numbers on a graph. Unless you mean whether this condition implies lower regularity...but that condition also seems too strong, since general Condition B surfaces won't have $\largetheta{E}^{d,\Delta,\ve}(B(x,r))\lesssim r^d$. }

Our first main result is the following. 
{

\begin{maintheorem}
Let $1\leq d<n$, $C_{0}>1$. 
Let $1\leq p<p(d)$ where $p(d)$ was defined in \eqref{e:pd}.  
Let $E\subseteq \bR^{n}$ be a closed set containing 0. 
Suppose that $E$ is $(c,d)$-lower content regular in $B(0,1)$. 
There is $\ve_{0}=\ve_{0}(n,p,c)>0$ 
such that for  $0<\ve<\ve_{0}$. 
Then
\begin{multline}
1+\sum_{Q\in \Delta \atop \ell(Q)\leq 1, Q\cap E\cap B(0,1)\neq\emptyset} \beta_{E}^{d,p}(C_{0}Q)^{2}\diam(Q)^{d}\\
\lec_{C_{0},n,\ve,p,c} 
\cH^{d}(E\cap B(0,1)) + \largetheta{E}^{d,\Delta,\ve}(B(0,1)).
\label{e:beta<hd2-dyadic}
\end{multline}
\label{t:thmii-dyadic}
\end{maintheorem}

}

\postRef{
In the case that $E$ is Reifenberg flat, it will automatically be lower content regular by \eqref{e:sigmalr} below. In this case, $\largetheta{E}^{d,\Delta,\ve}$ will be zero by definition, and so Theorem \ref{t:thmii-dyadic} gives a much cleaner result for this class of sets. See also Section \ref{s:theta} for more scenarios where $\largetheta{E}^{d,\Delta,\ve}$ disappears.}

\begin{remark}
We use  the lower content regularity crucially in two places.
One  is the proof of Lemma \ref{l:2-15}, and the other is  the proof of Lemma \ref{l:zeroboundary}.
\end{remark}

%
%
%
%

\RS{we should point out the paragraph below to the referee.  it was there, but i think it addresses point 4.  We should add a few words with an example saying it is neccessary.  I guess i am not convinced myself:  why can we not have a ``1" instead of Theta on the RHS of our theorem I.  this would entail using a ``free interval" method...  future work?\\}
\JA{Check out example 1.16 below.}

There will be a longer discussion about $\largetheta{E}^{d,\Delta,\ve}$ later after Example \ref{thetaexample}.

A converse to Theorem \ref{t:thmii-dyadic} is also possible.

{

\begin{maintheorem}
Let $1\leq d<n$, $C_{0}>1$, and $1\leq p\leq \infty$ and $E\subseteq \bR^{n}$ be $(c,d)$-lower content regular in $B(0,1)$ such that $0\in E$.  
Let 
$\ve>0$ be given. 
Then
\begin{multline}\label{e:thmiii-upper-bound-dyadic}
 \cH^{d}(E\cap B(0,1)) +
      \largetheta{E}^{d,\Delta,\ve}(0,1)\\
 \lec_{n,c,C_{0},\epsilon} 1+\sum_{Q\in \Delta \atop \ell(Q)\leq 1, Q\cap E\cap B(0,1)\neq\emptyset}\beta_{E}^{d,p}(C_0 Q)^{2}\diam(Q)^{d}.
 \end{multline}
{
Furthermore, if the right hand side of \eqref{e:thmiii-upper-bound-dyadic} is finite, then $E$ is $d$-rectifiable
}

\label{t:thmiii-dyadic}
\end{maintheorem}

}

\postRef{As a corollary of the above two theorems, we have that for any $(c,d)$-lower content regular set and $1\leq p<p(d)$ that for $\ve>0$ small enough
\[
\cH^{d}(E\cap B(0,1)) +
      \largetheta{E}^{d,\Delta,\ve}(0,1)\\
 \sim_{n,c,C_{0},p,\epsilon} 1+\sum_{Q\in \Delta \atop \ell(Q)\leq 1, Q\cap E\cap B(0,1)\neq\emptyset}\beta_{E}^{d,p}(C_0 Q)^{2}\diam(Q)^{d}.
 \]
 }

%
%
%

\begin{remark}
The proof of Theorem \ref{t:thmiii-dyadic} contains more information than presented in its statement.  The rectifiability of $E$ comes about from the construction of a sequence of surfaces, which are bi-Lipschitz images of $d$-dimensional cubes.  The bi-Lipschitz constant and sizes of cubes are controlled.  For example, one may slightly modify the construction to yield a connected, rectifiable set $\Gamma$ such that $\cH^d(E\setminus \Gamma)=0$ and 
for all $x\in \Gamma$ , $r\leq 1$ we have
$\cH^d_\infty(\Gamma\cap B(x,r))\geq c' r^d$ for some (explicit) $c'>0$. This $\Gamma$ will have 
$$\cH^d(\Gamma)\leq 
{  C(n,c)\ps{1+\sum_{Q\in \Delta \atop \ell(Q)\leq 1, Q\cap E\cap B(0,1)\neq\emptyset}\beta_{E}^{d,p}(C_0 Q)^{2}\diam(Q)^{d}}.
}
$$

 In fact, one obtains a coronization similar to that in \cite{DS}, but with bi-Lipschitz surfaces in place of Lipschitz graphs. We refer the reader to \cite{DS} for the definition of a coronization.
\end{remark}

\postRef{

\subsection{Examples and Motivation for using $\beta_{E}^{d,p}$}

Having stated these results, we can further motivate our choice of $\beta_{E}^{d,p}$.
There are a few other natural choices that one could try to make:  using averages (w.r.t. $\cH^d$) to define $\beta$, or using a supremum.  These are addressed in Examples \ref{e:ex-1}, \ref{e:ex-2} and \ref{e:ex-3})

 and discuss the difference between studying the geometry of the set $E$ and the geometry of the surface measure $\cH^{d}|_{E}$. Note that there are other $\beta$-type numbers we could have considered. One option is 
\[
\beta_{E}^{d,p,*}(B) = \inf_{L}\ps{\avint_{E\cap B} \ps{\frac{\dist(y,L)}{r_{B}}}^{p}d\cH^{d}(y)}^{\frac{1}{p}}.
\]
This quantity in some sense measures how flat the measure $\cH^{d}|_{E}$ is by measuring how close it is to lying on a $d$-dimensional plane. However, it does require the a priori assumption that $E$ has locally finite $\cH^{d}$-measure, and even so, it does not necessarily give useful {\it geometric} information if the surface measure is unstable, as the following example will show. 
%
%
\begin{example}\label{e:ex-1}
Consider the following set in the complex plane, let $I_{j}= [0,1]+2^{-j^2}i\subseteq \bC$ and set 
\[
E=\d[0,1]^{2}\cup \bigcup_{j=10}^{N} I_{j}.\]
If $Q$ is a dyadic cube contained in $[0,1]^{2}$, and $Q\cap E \neq\emptyset$ but $3Q$ does not contain any of the endpoints of the $I_{j}$ or points in $i+\bR$ (call this set of cubes $\cG$), then $Q\cap I_{j}\neq\emptyset$ for some $j$. If additionally $3Q\cap I_{k}\neq\emptyset$ for some $k>j$, then in fact $3Q\cap \bR\neq\emptyset$ so in fact $3Q\cap I_{\ell}$ for all $\ell\geq j$. Thus, either $\beta_{E}^{1,2,*}(3Q)=0$ or $\beta_{E}^{1,2,*}(3Q)>0$ and $3Q\cap \bR\neq\emptyset$, in which case, if $j(Q)$ be the smallest integer so that $2^{-j(Q)}<\ell(Q)$, then
\begin{align*}
\beta_{E}^{d,2,*}(3Q)^{2} 
& \sim \frac{1}{\ell(Q)(N-j(Q))}\int_{3Q\cap E}\ps{\frac{\dist(x,\bR)}{\ell(Q)}}^{2} d\cH^{1}(x)\\
& \frac{1}{\ell(Q)(N-j(Q))} \sum_{j=j(Q)-1}^{N}2^{-j^{2}+2j(Q)} \ell(Q)\\
& \lec \frac{1}{N-j(Q)}.
\end{align*}
It is not hard to show that if $\cB$ are those cubes $Q$ for which $3Q$ contains an endpoint of some $I_{j}$ then $\sum_{Q\in \cB}\ell(Q)\lec 1$.  If cubes only intersect $i+\bR$ then $\beta_{E}^{d,2,*}=0$.
Thus,
\begin{align*}
\sum_{Q\subseteq Q_{0}\atop Q\cap E\neq\emptyset}\beta_{E}^{d,2,*}(3Q)^{2} \ell(Q)
& \lec 1+ \sum_{Q\in \cG}\frac{1}{N-j(Q)} \ell(Q)
 \lec 1+ \sum_{j=1}^{N-1} \frac{1}{N-j} \lec \log N.
\end{align*}
However, $\cH^{1}(E)\sim N$, and so it is not possible to obtain a theorem like Theorem \ref{t:thmiii-dyadic}, even if we assumed $E$ had finite measure. Additionally, even if a theorem like Theorem \ref{t:thmii-dyadic} held with $\beta_{E}^{1,p,*}$, this estimate above shows that the square sum of $\beta_{E}^{1,p,*}$ in place of $\beta_{E}^{1,p}$ gives a weaker lower bound for $\cH^{1}(E)+\Theta_{E}^{1,\Delta,\ve}(0,\sqrt{2})$. 
\end{example}

\begin{example}\label{e:ex-2}
Using $\beta_{\cH^{d}|_{E},p}^{d}$ won't work either: if we take $E$ to be the union of boundaries of squares in $[0,1]^{2}$ of sidelength $2^{-N}$, then, for each dyadic cube $Q$ that intersects $E$ with $\ell(Q)=2^{-j}\geq 2^{-N}$, 
\[
\beta_{\cH^{1}|_{E},2}^{d}(3Q)^{2}
\gec \frac{2^{2(N-j)} 2^{-N}}{\ell(Q)} = 2^{N-j} = 2^{N}\ell(Q)\]
and so 
\[
\sum_{Q\subseteq [0,1]^{2} \atop Q\cap E\neq\emptyset} \beta_{\cH^{1}|_{E},2}^{2}(3Q)^{2} \ell(Q) 
\gec \sum_{Q\subseteq [0,1]^{2} \atop Q\cap E\neq\emptyset,\ell(Q)\geq  2^{-N}} 2^{N}\ell(Q)^{2}
\gec 2^{N}N.
\]
However, $\cH^{1}(E)\sim 2^{2N}\cdot 2^{-N}=2^{N}\sim \Theta_{E}^{1,\Delta,\ve}(0,\sqrt{2})$, so again, a version of Theorem \ref{t:thmii-dyadic} cannot hold with these $\beta$-numbers. 
\end{example}

\begin{example}\label{e:ex-3}
Using $\beta_\infty$ dos not work.
The following example was constructed by X. Fang in his Ph.D. dissertation under the supervision of P. Jones \cite{Fang-Jones-example}.
It was communicated to the authors by Guy David and Sean Li.

Let $Q_0\subset \bR^3$ be the unit cube.
We will construct a sequence of functions $f_0,f_1, f_2,...: Q_0\to \bR$ which uniformly converge to a function $f$.  All of these functions will be 1-Lipschitz.
To this end, let $\epsilon_1,  \epsilon_2,...<1$ be a sequence of powers of 2,  such that
$\sum \epsilon_k^2$
diverges, and 
$\sum \epsilon_k^3<\frac12.$
Write $\epsilon_k=2^{-j_k}$.
We will set $\cD_k$ to be dyadic subcubes of $Q_0$  of side 
$2^{-j_1-...-j_k}$.
Let $f_0:Q_0\to \bR$ be the constant function. 
Subdivide $Q_0$ into cubes of side $\epsilon_1$, i.e. into the cubes of $\cD_1$.  
Choose one of these $\cD_1$ cubes, and modify $f_0$ on it to obtain $f_1$ by adding a piecewise linear spike of hight $\epsilon_1/2$. 
Take now all other $\cD_1$ cubes, and in each one choose a $\cD_2$ cube. 
Get $f_2$ from $f_1$ by adding a piecewise linear spike of hight $\epsilon_2/2$ at each of these special $\cD_2$ cubes.
Continue this way.  
To get $f_{i+1}$ from $f_i$ choose in each $\cD_i$ cube on which $f_i$ is constant, a single subcube from $\cD_{i+1}$, and add a piecwise linear spike there of hight $\epsilon_{i+1}/2$.
We now make a few observations.
\begin{itemize}
\item If on $f_i$ was modified on a cube $Q\in \cD_{i+1}$, then for all $k>i$, $f_k$ will never be modified inside $Q$.
\item $\|f_{i+1}-f_i\|_\infty<\epsilon_i$
\item The total volume of cubes from $\cD_{k+1}$ which was modified in going from 
$f_k$ to $f_{k+1}$ was 
$$1-\sum_1^k \epsilon_i^3\geq \frac12.$$ 
\item At each of the cubes modified in going from 
$f_k$ to $f_{k+1}$, 
the, $\beta_\infty$ of the correspoonding dyadic cube in $\bR^4$ (for the graph of the function $f_l$ for any $l>k$) is $\sim \epsilon_{k+1}$.
\end{itemize}
Thus we have that if we compute 
$$\sum \beta_\infty^2(Q) \diam(Q)^3$$
for the graph of the limiting function, the sum is proportional to 
$\sum \epsilon_k^2$ which diverges.
The main point  above is that the exponent of the $\beta$  and that of the $\epsilon$ sum differ because of the dimension.
By scaling this example we may get that it is Reifenberg flat. 
\end{example}

Thus, the $\beta_{E}^{d,p}$ that we have are ideal in that none of the previously studied $\beta$-numbers can achieve the same results. 
}

\bigskip

\postRef{

\subsection{Motivation for $\largetheta{E}^{d,\Delta,\ve}$, and when it disappears}
\label{s:theta}

The presence of $\largetheta{E}^{d,\Delta,\ve}$ in our results may seem odd, but our next example shows that it cannot be ignored.

\begin{example}\label{thetaexample}
Let $f_{j}(z)=z/4+z_{j}$ where $z_{0}=0,$ $z_{1}=3/4$, $z_{2}=3i/4$, and $z_{3}=3(1+i)/4$. These are the usual contractions used in defining the $4$-corner Cantor set in the plane. Let $f(K):=\bigcup_{j=0}^{3} f_{j}(K)$,  $K_{0}:=\d[0,1]^{2}$, and $K_{n}=f(K_{n-1})$. Then it is not hard to show that $\cH^{1}(K_{n})\sim 1$ for all $n$ while
\[
\sum_{Q\cap K_{n}\neq\emptyset}\beta_{E}^{1,2}(3Q)^{2}\ell(Q)
\sim n.
\]
Thus, by Theorems \ref{t:thmii-dyadic} and \ref{t:thmiii-dyadic}, for large $n$, $\Theta_{K_{n}}^{1,\Delta,\ve}(B(0,2))\sim n\gg \cH^{1}(K_{n})$. 
\end{example}
}
This exampe should be viewed in the following way. 
One may estimate the length of the shortest curve containing a set $E$ by \Theorem{TST}. This length, however, could be much bigger than the length of $E$. The set $E$ can be seen as the curve $\Gamma$ punctured by  holes, which  $\largetheta{E}^{d,\Delta,\ve}$ accounts for.  This happens since $\vartheta_{E}^{d}$ also measures  how far an optimal plane is from $E$, and hence, if $\vartheta_{E}^{d}$ is large yet $\beta_{E}^{d,p}$ is small, this means that $E$ is very flat but contains a large $d$-dimensional hole.

Finally, we note that the  $\largetheta{E}^{d,\Delta,\ve}$ quantity is subsumed by the $\cH^d$ quantity in some natural situations. Firstly, \Theorem{TST} implies that 
\[\largetheta{\Gamma}^{1,\Delta,\ve}(x,r)\leq C\ve^{-2}\cH^{1}(B(x,r)\cap \Gamma)\] 
for any curve $\Gamma$. There are also some objects of higher topological dimension for which this holds.

\begin{definition}[Condition B]\label{d:condition-b}
We will say $E\subseteq \bR^{n}$ {\it satisfies Condition B} for some $c>0$ if for all $x\in E$ and $r>0$, one can find two balls of radius $rc$ contained in $B(x,r)$ in two different components of $E^{c}$.
\end{definition}

Usually, this definition also assumes $E$ is Ahlfors regular (see for example \cite{Dav88}, \cite{DS93}, and \cite{DJ90}), and there the authors give different proofs that, in this situation, $E$ is uniformly rectifiable. If $E$ has locally finite $\cH^{n-1}$-measure, then one can show that, for any ball $B$ centered on $E$, there is a Lipschitz graph $\Gamma$ so that $\cH^{n-1}(\Gamma\cap B\cap E)\geq r_{B}^{n-1}$ with Lipschitz constant depending on $\cH^{n-1}( B\cap E)$ (see \cite{JKV97} and \cite{Bad12}). 

\begin{theorem}[David, Semmes \cite{DS93}]  Let $E\subseteq \bR^{n}$ satisfy Condition B for some $c>0$. Then for all $x\in E$, and $\ve,r>0$,
\begin{equation}\label{e:star}
\largetheta{E}^{n-1,\Delta,\ve}(x,r)\leq C(\ve,n) \cH^{n-1}(E\cap B(x,r)).\end{equation}
\end{theorem}

This is not stated in \cite{DS93}, but it is implied by the proof of \cite[Theorem 1.20]{DS93}, since it only really depends on the lower regularity of $E$, that is, 
\[
\cH^{d}(B(x,r)\cap E)\geq cr^{d} \mbox{ for all $x\in E$ and $0<r<\diam E$},\] 
and such sets are always lower regular with constant depending on the Condition B constant.

Thus, by Theorems \ref{t:thmii-dyadic} and \ref{t:thmiii-dyadic}, we have the following.

\begin{maincor}
If $E$ satisfies Condition B for some constant $c>0$, then for $C_{0}>1$ and $1\leq p<p(d)$, 
{
\begin{equation}
 \cH^{n-1}(E\cap B(0,1)) 
\sim_{C_{0},n,c,p} 1+\sum_{Q\in \Delta \atop Q\cap E\cap B(0,1)\neq\emptyset}\beta_{E}^{n-1,p}(C_0 Q)^{2}\diam(Q)^{n-1}
\label{e:beta<hd3}
\end{equation}}

\end{maincor}

An interesting problem would be to determine what simple geometric criteria a set $E$ has to have in order to satisfy 
\begin{equation}
\largetheta{E}^{d,\Delta,\ve}(B(x,r))\leq C(\ve,n) \cH^{d}(E\cap B(x,r)).
\end{equation}
much like the program that already exists for finding criteria for when Ahlfors regular sets are  uniformly rectifiable (see \cite{DS,of-and-on}). \\

%
%

\subsection{Outline of the paper}
Section \ref{s:preliminaries} contains preliminaries, a discussion of a theorem of David and Toro (Theorem \ref{t:DT}), as well as some basic properties of our Choquet-style definition of $\beta$.
Other preliminaries regarding  $\cH^d_\infty$ ``integration'' (in the sense of Choquet) were pushed to the Appendix (Section \ref{s:appendix}).

{ In Section \ref{s:state-flat-case}, we restate Theorem \ref{t:thmii-dyadic} differently as Theorem \ref{t:thmii}, using maximal nets and balls instead of cubes, which will be more natural to prove.
We also have a version of Theorem \ref{t:thmiii-dyadic} with nets and balls.
 We then introduce Theorem \ref{t:flat-case}, which is a version of Theorem \ref{t:thmii} for the Reifenberg flat case. 
}
Section \ref{s:description-of-proof} will contain a loose description of the proofs of Theorems \ref{t:flat-case} and \ref{t:thmii}, the details of which will be carried out in Sections \ref{s:stoppingtime}--\ref{s:TheoremII}. 

Sections \ref{s:pf-of-thmiii}  and \ref{s:pf-of-thmiii-part2} contain a proof of Theorem \ref{t:thmiii}. The proof  can be summarized as  {\it a stopping time applied to 
Theorems \ref{t:DT} and \ref{DTtheorem2.5}} (both  from \cite{DT12}).

\subsection{Constants}

We list some important constants, where they appear, are fixed, and their dependancies.
\begin{itemize}
\item $n$: ambient dimension (as in $\bR^n$). Given in main theorems.
\item $d$: intrinsic dimension (as in $\cH^d$). Given in main theorems.
\item $C_0$: given in main theorems.
\item$c_0$:  fixed in Theorem \ref{t:Christ} as $1/500$.
\item$r_k$: fixed in Theorem \ref{t:DT} as $r_{k}=10^{-k}$.
\item$A$: ball dilation factor. Given in the reformulation of the main theorems which happens in Section \ref{s:state-flat-case}.
\item $C_1$: fixed as $C_1=2$ for Lemma \ref{l:four-ten} and $C_{1}=2C_{2}$ for \Lemma{graph}. In section \ref{s:TheoremII}, it will be fixed so that $20C_{0}\ll C_{1}$ where $C_{0}$ is the constant appearing in Theorem \ref{t:thmii}. 
\item $C_{2}$: a constant introduced in \Lemma{sigmagraph} and fixed in Section \ref{s:beta-omega}, see Remark \ref{r:c2}.

\item $\rho$: scale factor. See Theorem \ref{t:Christ}.  The constant $\rho$ is fixed following the statement of Theorem \ref{t:Christ}.

\item $\alpha$: Allowed angle of rotation between planes in stopping time region. See remark \ref{r:params-1} and definition of $S_Q$ which follows it. 
The constant $\alpha$ may be fixed in the proof of Lemma \ref{l:asum2} and depends only on $n$, see Remark \ref{r:alpha}.

\item $\epsilon$: Reifenberg flatness constant.  Depends on $n$, and $\alpha, \tau$, and should be small enough for \Theorem{DT} below to hold. See Remark \ref{r:params-1}.
In relation to $\alpha$ we will require that $\epsilon\ll\alpha^4$.

\item $\tau$, $\tau_0, \tau_1$: Constant used for extending Stopping time regions. $\tau<\tau_0, \tau_1, \frac{c_0}4$, where $\tau_0=\tau_0(\rho)$, and $\tau_1=\min\{\tau_1(2C_2),\tau_{1}(2),\tau_1(4)\}$. See Remark~\ref{r:tau2}.
The constant $\tau$ is fixed (for the proof of Theorem \ref{t:flat-case}) at the start of  Section \ref{s:layers}.

\item $M$: a dilation factor for balls. Used in a similar fashion to $C_0$ and $A$. 

\item $\theta$: an angle of rotation between planes. First introduced in Lemma \ref{l:graph}. One can take $\theta\lesssim \epsilon$.  

\end{itemize}

\subsection{Acknowledgements}

The authors would like to thank Peter Jones, Xavier Tolsa, and Tatiana Toro for their helpful discussions. In fact, some of the core ideas arose from ongoing work between the first author and Xavier Tolsa  as well as the second author and Peter Jones. \postRef{We also thank Silvia Ghinassi and Michele Villa for their careful proofreading of the manuscript. Finally, we would like to thank the anonymous referee who had many useful comments and suggestions that greatly improved the paper.}

\section{Preliminaries}\label{s:preliminaries}

We will write $a\lesssim b$ if there is $C>0$ such that $a\leq Cb$ and $a\lesssim_{t} b$ if the constant $C$ depends on the parameter $t$. We also write $a\sim b$ to mean $a\lesssim b\lesssim a$ and define $a\sim_{t}b$ similarly.

For sets $A,B\subset \bR^{n}$, let 
\[\dist(A,B)=\inf\{|x-y|:x\in A,y\in B\}, \;\; \dist(x,A)=\dist(\{x\},A),\]
and 
\[\diam A=\sup\{|x-y|:x,y\in A\}.\]

{
Recall the definition of $d_{x,r}(E,F)$ from Definition \ref{d:dxref}. It is not hard to show that, for sets $E,F,G$, while this does not satisfy the triangle inequality, we do have
\begin{equation}\label{e:dxrtrngl}
d_{x,r}(E,G)\lec d_{x,2r} (E,F)+d_{x,2r}(F,G).
\end{equation}

}

{
We will denote by $\Delta$ the standard dyadic grid on $\bR^n$, that is
$$\Delta=\left\{\left[\frac{j_1}{2^k},\frac{j_1+1}{2^k}\right)\times \dots \times \left[\frac{j_n}{2^k},\frac{j_n+1}{2^k}\right],\quad k,j_1,\dots,j_n\in\bZ\right\}.$$

For any cube $Q$ and $\lambda>0$, we let $\lambda Q$ the cube which i a dilation by $\lambda$ of $Q$, that is, $\lambda Q$ is concentric with $Q$,  has sides parallel to $Q$ and with  $\diam(\lambda Q)= \lambda \diam(Q)$.
}

\subsection{Hausdorff measure and content}\label{subsec:meas-and-cont}

For a subset $A\subset\bR^{n+1}$, integer $d>0$, and $0<\delta\leq \infty$ one sets
\[\cH^{d}_{\delta}(A)=\inf\ck{\sum (\diam A_i)^{d}: A\subset \bigcup A_i,\,\diam A_i\leq \delta}.\]
 The {\it $d$-dimensional Hausdorff measure} of $A$ is defined as
\[\cH^{d}(A)=\lim_{\delta\downarrow 0}\cH^{d}_{\delta}(A),\]
and $\cH^{d}_{\infty}(A)$ is called the {\it $d$-dimensional Hausdorff content} of $A$. See \cite[Chapter 4]{Mattila} for more details.  
Note that $\cH^{d}_{\infty}$ is not a  measure. We do, however, want to use the notation of  integration with respect to $\cH^{d}_{\infty}$.
For $0< p<\infty$ and $E\subseteq \bR^{n}$ Borel, we will define the $p$-Choquet integral 
\[
\int_{E} f^{p} d\cH^{d}_{\infty}:=\int_{0}^{\infty} \cH^{d}_{\infty}(\{x\in E:f(x)>t\})t^{p-1}dt\]
and 
\[ \int f d\cH^{d}_{\infty} =  \int_{\bR^{n}} f d\cH^{d}_{\infty}.\]

\medskip

The proofs of the following three lemmas may be found in the Appendix (Section \ref{s:appendix}).
\begin{lemma}\label{l:intineq}
Let $0< p<\infty$.
Let $f_{i}$ be a countable collection of Borel functions in $\bR^{n}$. If the sets $\supp f_{i}=\{f_{i}>0\}$ have bounded overlap, meaning there exists a $C<\infty$ such that 
\[
\sum \one_{\supp f_{i} }\leq C,\]
then
\begin{equation}\label{e:intineq}
 \int \ps{\sum f_{i}}^{p}d\cH^{d}_{\infty} \leq C^{p} \sum \int  f_{i}^{p}d\cH^{d}_{\infty}.\end{equation}
\end{lemma}

\begin{lemma}
\label{l:contentlim}
Let $E\subseteq \bR^{n}$ be compact and $f$ a continuous function on $\bR^{n}$. Let $E_{j}$ be a decreasing sequence of sets containing $E$ and converging to $E$ in the Hausdorff metric. Then 
\begin{equation}\label{e:etetj}
\lim_{j\rightarrow\infty} \int_{E_{j}} f d\cH^{d}_{\infty} \sim \int_{E} fd\cH^{d}_{\infty}.
\end{equation}

\end{lemma}

\begin{lemma}
\label{l:jensens}
Let $E\subseteq \bR^{n}$ be either compact or bounded and open so that $\cH^{d}(E)>0$, and let $f\geq 0$ be continuous on $E$. Then for $1<p\leq \infty$,
\begin{equation}\label{e:jensens}
\frac{1}{\cH^{d}_{\infty}(E)}\int_{E} f d\cH^{d}_{\infty}\lec_{n}  \ps{\frac{1}{\cH^{d}_{\infty}(E)}\int_{E} f^{p}d\cH^{d}_{\infty}}^{\frac{1}{p}} .
\end{equation}
\end{lemma}

\subsection{Reifenberg Flat sets and the theorem of David and Toro}

Our main tool will be the enhanced Reifenberg parametrization theorem of David and Toro. First, we recall Reifenberg's theorem. 

\begin{theorem}\label{t:reif}
For all $0<d<n$ and $0<s<1/10$, we may find $\ve>0$ such that the following holds. Let $E\subseteq \bR^{n}$ be a closed set containing the origin that is $\ve$-Reifenberg flat in $B(0,10)$. Then there is a bijective mapping $g:\bR^{n}\toitself$ such that 
\[ \frac{1}{4}|x-y|^{1+s}\leq |g(x)-g(y)|\leq 10|x-y|^{1-s}\]
and
\[E\cap B(0,1)=g(\bR^{d})\cap B(0,1).\] 
Moreover, $\Sigma:=g(\bR^{n})$ is $C\ve$-Reifenberg flat. 
\end{theorem}

This is the main result as stated, but there are special properties that are implicit in David and Toro's version that we shall employ. What we don't cite below is covered in Sections 2,3 and 4 of \cite{DT12}.

\begin{theorem}\cite{DT12}
For $k\in \bN\cup\{0\}$, set $r_{k}=10^{-k}$ and let $\{x_{jk}\}_{j\in J_{k}}$ be a collection of points so that for some $d$-plane $P_{0}$,
\[\{x_{j0}\}_{j\in J_{0}}\subset P_{0},\] 
 \[|x_{ik}-x_{jk}|\geq r_{k},\]
and, denoting $B_{jk}=B(x_{jk},r_{k})$, 
\begin{equation}
x_{ik}\in V_{k-1}^{2}
\label{V2}
\end{equation}
where
\[ V_{k}^{\lambda}:=\bigcup_{j\in J_{k}}\lambda B_{jk}.\]
To each point $x_{jk}$, associate a $d$-plane $P_{jk}\subset \bR^n$  
{such that $P_{jk} \ni x_{jk}$}
and set
\begin{multline*}
\ve_{k}(x)=\sup\{d_{x,10^{4}r_{l}}(P_{jk},P_{il}): j\in J_{k}, |l-k|\leq 2, i\in J_{l}, \\
x\in 100 B_{jk}\cap 100 B_{il}\}.
\end{multline*}
There is $\ve_{0}>0$ such that if $\ve\in (0,\ve_{0})$ and 
\begin{equation}
\label{e:vek<ve}
\ve_{k}(x_{jk})<\ve \mbox{ for all }k\geq 0\mbox{ and }j\in J_{k},
\end{equation}
then there is a bijection $g:\bR^{n}\rightarrow \bR^{n}$ so that the following hold
\begin{enumerate}
\item We have 
\begin{equation}\label{eqahj0}
E_{\infty}:=\bigcap_{K=1}^{\infty}\overline{\bigcup_{k= K}^{\infty} \{x_{jk}\}_{j\in J_{k}}}\subseteq \Sigma:= g(\bR^{d}).
\end{equation}
\item $g(z)=z$ when $\dist(z,P_{0})>2$.
\item For $x,y\in \bR^{n}$,
\[ \frac{1}{4}|x-y|^{1+\tau}\leq |g(x)-g(y)|\leq 10|x-y|^{1-\tau}.\]
\item $|g(z)-z|\lec \ve$ for $z\in \bR^{n}$
\item For $x\in P_0$, $g(x)=\lim_{k} \sigma_{k}\circ\cdots \sigma_{1}(x)$ where 
\begin{equation}
\sigma_{k}(y)=\psi_{k}(y)y+\sum_{j\in J_{k}}\theta_{j,k}(y) \pi_{j,k}(y).
\label{e:sigmak}
\end{equation}
Here, $\{x_{j,k}\}_{j\in L_{k}}$ is a maximal $\frac{r_{k}}{2}$-separated set  in $\bR^{n}\backslash V_{k}^{9}$, 
\[
B_{j,k}=B(x_{j,k},r_{k}/10) \mbox{ for $j\in L_{k}$},\] 
$\{\theta_{j,k}\}_{j\in L_{k}\cup J_{k}}$ is a partition of unity such that $\one_{9B_{j,k}}\leq \theta_{j,k}\leq \one_{10 B_{j,k}}$ for all $k$ and $j\in L_{k}\cup J_{k}$, and $\psi_{k}=\sum_{j\in L_{k}}\theta_{j,k}$. 
\item \cite[Equation (4.5)]{DT12}For $k\geq 0$,
\begin{equation}\label{e:v10c}
\sigma_{k}(y)=y \mbox{ and } D_{k}(y)=I \mbox{ for }y\in \bR^{n}\backslash V_{k}^{10}.
\end{equation}
\item\cite[Proposition 5.1]{DT12} Let $\Sigma_{0}=P_{0}$ and
\[
\Sigma_{k}=\sigma_{k}(\Sigma_{k-1}).\] 
There is a function $A_{j,k}: P_{j,k}\cap 49B_{jk}\rightarrow P_{j,k}^{\perp}$ of class $C^{2}$ such that $|A_{jk}(x_{jk})|\lec \ve r_{k}$, $|DA_{jk}|\lec \ve$ on $P_{jk}\cap 49B_{jk}$, and if $\Gamma_{jk}$ is its graph over $P_{jk}$, then 
\begin{equation}\label{e:49graph}
\Sigma_{k}\cap D(x_{jk},P_{jk},49r_{k})=\Gamma_{jk}\cap  D(x_{jk},P_{jk},49r_{k})
\end{equation}
where 
\begin{equation}\label{e:D-is-a-cylinder}
D(x,P,r)=\{z+w: z\in P\cap B(x,r), w\in P^{\perp} \cap B(0,r)\}. 
\end{equation}
(Above $P^\perp$ is the plane perpendicular to $P$ going through 0.)
In particular,
\begin{equation}
d_{x_{jk},49r_{jk}}(\Sigma_{k},P_{jk})\lec\ve.
\label{e:49r}
\end{equation}
\item \cite[Lemma 6.2]{DT12} For $k\geq 0$ and $y\in \Sigma_{k}$, there is an affine $d$-plane $P$ through $y$ and a $C\ve$-Lipschitz and $C^{2}$ function $A:P\rightarrow P^{\perp}$ so that if $\Gamma$ is the graph of $A$ over $P$, then 
\begin{equation}\label{e:ygraph}
\Sigma_{k}\cap B(y,19r_{k})=\Gamma\cap B(y,19r_{k}).
\end{equation}
\item \cite[Proposition 6.3]{DT12} $\Sigma=g(P_0)$ is $C\ve$-Reifenberg flat in the sense that for all $z\in \Sigma$, and $t\in (0,1)$, there is $P=P(z,t)$ so that $d_{z,t}(\Sigma,P)\lec \ve$. 
\item For all $y\in \Sigma_{k}$, 
\begin{equation}\label{e:sky-y}
|\sigma_{k}(y)-y|\lec \ve_{k}(y) r_{k}
\end{equation}
This is not stated as such in \cite{DT12}, but it follows from (7.13) in \cite{DT12} and the definition of $\sigma_{k}$. In particular, it follows that
\begin{equation}\label{e:ytosigma}
\dist(y,\Sigma)\lec \ve r_{k} \mbox{ for }y\in \Sigma_{k}.
\end{equation}
\item \cite[Lemma 7.2]{DT12} For $k\geq 0$, $y\in \Sigma_{j}\cap V_{k}^{8}$, choose $i\in J_{k}$ such that $y\in 10 B_{i,k}$. Then
\begin{equation}\label{e:skCloseInBall}
|\sigma_{k}(y)-\pi_{i,k}(y)|\lec \ve_{k}(y)r_{k}
\end{equation} 
and
\begin{equation}
|D\sigma_{k}(y)-D\pi_{i,k}|\lec \ve_{k}(y) \end{equation}
 If $T\Sigma_{k}(x)$ denotes the tangent space at $x\in \Sigma_{k}$, then 
 \begin{equation}\label{e:Tsig}
 \angle(T\Sigma_{k+1}(\sigma_{k}(x)),P_{i,k})\lec \ve_{k}(y) \mbox{ for }x\in \Sigma_{k}\cap B(x_{i,k},10r_{k}).
 \end{equation}
 
\item \cite[Lemma 9.1]{DT12} For $x,y\in \Sigma_{k}$, 
\begin{equation}
\angle(T\Sigma_{k}(x),T\Sigma_{k}(y))\lec \ve \frac{|x-y|}{r_{k}}.
\label{e:Tlip}
\end{equation}
\item \cite[Lemma 13.2]{DT12} For $x\in \Sigma$ and $r>0$, 
\begin{equation}
\cH^{d}_{\infty}(B(x,r)\cap \Sigma)\geq (1-C\ve)\omega_{d} r^{d}
\label{e:sigmalr}
\end{equation}
where $\omega_{d}$ is the volume of the unit ball in $\bR^{d}$. They prove this statement with $\cH^{d}$ in place of $\cH^{d}_{\infty}$, but the same proof works for $\cH^{d}_{\infty}$. 
\end{enumerate}

%
\label{t:DT}
\end{theorem}

\begin{lemma}
\label{l:skflat}
With the notation as in \Theorem{DT}, there is $C>0$ depending only on $n$ so that for all $k\geq 0$, $\Sigma_{k}$ is $C\ve$-Reifenberg flat. 
\end{lemma}

\begin{proof}
Let $x\in \Sigma_{k}$ and $r>0$. If $r\leq r_{k}$, then $\vartheta_{\Sigma_{k}}(x,r)\lec \ve $ by \eqn{ygraph}, so assume $r >r_{k}$.By \eqn{ytosigma}, there is $x'\in \Sigma$ with $|x-x'|\lec \ve r_{k}$.  By \Theorem{DT}, $\Sigma$ is $C\ve$-Reifenberg flat for some $C>0$, and so there is a $d$-plane $P$ passing through $x'$ so that $d_{x',2r}(\Sigma,P')<C\ve$. Let  $P$ be the plane parallel to $P'$ but containing $x$. Then it is not hard to show that, for $\ve>0$ small, $d_{x',2r}(\Sigma,P)\lec \ve$ as well. 

Let $y\in B(x,r)\cap \Sigma_{k}$. Again, there is $y'\in \Sigma$ so that $|y-y'|\lec \ve r_{k}$. For $\ve>0$ small enough, $y'\in B(x',2r)\cap \Sigma$, and so $\dist(y',P')\lec \ve r_{k}$. Since $P$ is $P'$ translated by no more than a constant times $\ve r_{k}$, we also have that $\dist(y',P)\lec \ve r_{k}$. Thus,
\begin{equation}\label{e:yPerk}
\dist(y,P)\leq |y-y'|+\dist(y',P) \lec \ve r_{k}
\end{equation}
and this holds for all $y\in B(x,r)\cap \Sigma_{k}$.

Now let $z\in B(x,r)\cap P$ and $z'=\pi_{P'}(z)$. Then $|z-z'|\lec \ve r_{k}$, and so for $\ve>0$ small enough, $z'\in B(x',2r)\cap P'$. Thus, there is $z''\in \Sigma$ with $|z'-z''|\lec \ve r_{k}$. Recall that $g$ is a homeomorphism, which thus forces every $\sigma_{j}$ to be a homeomorphism (since $g$ is the bi-H\"older composition of them all). Hence $g_{k}$ has an inverse, and by \eqn{sky-y}, there is $z'''\in \Sigma_{k}$ with $|z'''-z''|\lec \ve r_{k}$. Combining these estimates, we have 
\[
\dist(z,\Sigma_{k})
\leq |z-z'''|
\leq |z-z'|+|z'-z''|+|z''-z'''|
\lec \ve r_{k}\]
and this holds for all $z\in B(x,r)\cap P$. Combining this with \eqn{yPerk} gives $\vartheta_{\Sigma_{k}}(B(x,r))\lec \ve$, which proves the lemma.
\end{proof}

The following lemma will allow us to localize by finding a Reifenberg flat surface $\Sigma$ that agrees with a Reifenberg flat surface $E$ in a ball $B$ but whose surface measure is controlled by the measure of $E$ in $B$.

\begin{lemma}
\label{l:rtrim}
Let $E\subseteq \bR^{n}$ be $\ve$-Reifenberg flat containing $0$. Then there is a $C\ve$-Reifenberg flat surface $\Sigma$ so that
\begin{enumerate}
\item $E\cap \cnj{B(0,1)}\subseteq \Sigma\cap \cnj{B(0,1)}$,
\item $\Sigma\backslash B(0,10)=P_{0}\backslash B(0,10)$ for some $d$-plane $P_{0}$,
\item for all $A>0$, 
\begin{equation}\label{e:rtrim}
\cH^{d}(\Sigma\cap B(0,A))\lec A^{d} \cH^{d}(E\cap {B(0,1)}). 
\end{equation}
\end{enumerate}
\end{lemma}

\begin{proof}
Let 
\[
E_{k}=\{x\in E:\dist(x,E\backslash B(0,1))\geq r_{k}\}\]
and $\{x_{jk}\}_{J_{k}}$ be maximally separated $r_{k}$-nets in $E_{k}$ so that $\{x_{j0}\}_{j\in J_{0}}=\{0\}$. Then, with the notation of \Theorem{DT}, 
\[
\{x_{j,k+1}\}_{J_{k+1}}\subseteq E\cap \cnj{B(0,1)}\subseteq V_{k}^{2} \;\; \mbox{ for all }k\geq 0.\]
In this way, $E_{\infty}= E\cap \cnj{B(0,1)}$. If we let $P_{j,k}$ be the plane that infimizes $\vartheta_{E}(B(x_{jk},10^{6}r_{k}))$ and $P_{0}$ the plane that infimizes $\vartheta_{E}(B(0,10^{6}r_{k}))$, then for $\ve>0$ small enough, we may apply \Theorem{DT} to obtain a $C\ve$-Reifenberg flat surface $\Sigma$.  

Let $x\in \Sigma\cap B(0,10)\backslash E_{\infty}$ and let $k(x)$ be the maximal $k$ for which $x\in V_{k-1}^{11}$. Then $x\not\in V_{k(x)}^{11}$, and so if we set $r(x)=r_{k(x)}$, we have $B(x,r(x))\subseteq V_{k(x)}^{10}$. Thus,
\[
\Sigma_{k}\cap B(x,r(x))= \Sigma\cap B(x,r(x)).\]
By \eqn{ygraph}, $\Sigma\cap B(x,r(x))$ is a $C\ve$-Lipschitz  graph, so this and the above equation imply 
\begin{equation}
\label{e:sbxrxgraph}
\cH^{d}(\Sigma\cap B(x,r(x)))\lec r^{d}.
\end{equation}
By the Besicovitch covering theorem, we can find $x_{i}\in \Sigma\cap B(0,10)\backslash E_{\infty}$ so that $B_{i}:=B(x_{i},r(x_{i}))$ cover $\Sigma\cap B(0,10)\backslash E_{\infty}$ with bounded overlap. Since $x_{i}\in V_{k(x_{i})-1}^{11}$, we may find $x_{i}'\in \{x_{j,k-1}\}_{j\in J_{k-1}}$ so that 
\[
x_{i}\in B(x_{i}',10r_{k(x_{i})-1})=  B(x_{i}',r_{k(x_{i})}) .\]
Since $x_{i}' \in  \{x_{j,k-1}\}_{j\in J_{k-1}}\subseteq E_{k-1}$, we know
\begin{equation}\label{e:biinball}
 B(x_{i}',r_{k(x_{i})-1})\cap E\subseteq B(0,1).
 \end{equation}
  By \eqn{sigmalr}, we have 
\begin{equation}\label{e:rxid<E}
r(x_{i})^{d} \lec \cH^{d}(B(x_{i}',r_{k(x_{i})-1})\cap E_{\infty}).
\end{equation}
For $j\in J_{k}$, let 
\[
A_{j}=\{B_{i}: x_{i}'=x_{jk}\}.\] 
Since the balls in $A_{j}$ have bounded overlap, the same radius, and distance eat most $r_{k}$ from $x_{jk}$, we know $\# A_{j}\lec_{n} 1$. Thus,
\begin{align*}
\cH^{d}(\Sigma\cap B(0,10)\backslash E_{\infty})
& \leq \sum_{i} \cH^{d}(B_{i}\cap \Sigma)
\stackrel{\eqn{sbxrxgraph}}{\lec} \sum_{i} r(x_{i})^{d}\\
& \stackrel{\eqn{rxid<E}}{\lec} \sum_{i}\cH^{d}(B(x_{i}',r_{k(x_{i})-1})\cap E)\\
& \leq \sum_{k\geq 0}\sum_{j\in J_{k}} \#A_{j} \cH^{d}(B(x_{jk},r_{k-1})\cap E)\\
& \lec \sum_{k\geq 0}\sum_{j\in J_{k}} \cH^{d}(B(x_{jk},r_{k-1})\cap E)\\
& \stackrel{\eqn{biinball}}{\leq} \cH^{d}(E\cap B(0,1))
\end{align*}
Recalling that $E_{\infty}=E\cap \cnj{B(0,1)}$,
\begin{align*}
\cH^{d}(\Sigma\cap B(0,10))
& \leq \cH^{d}(\Sigma\cap B(0,10)\backslash E_{\infty})+\cH^{d}(E_{\infty})\\
& \lec \cH^{d}(E\cap \cnj{B(0,1)}).
\end{align*}
Now note that $V_{k}^{10}\subseteq V_{0}^{10}$ for all $k$, and so by \eqn{sigmak}, 
\begin{align*}
\cH^{d}(\Sigma\cap B(0,A)\backslash B(0,10))
& =\cH^{d}(P_{0}\cap B(0,A)\backslash B(0,10))\\
& \lec A^{d} \stackrel{\eqn{sigmalr}}{\lec} A^{d} \cH^{d}(E\cap B(0,1)).
\end{align*}
Combining these estimates completes the proof.

\end{proof}

\subsection{Generalized dyadic cubes}

We recall the construction of cubes on a metric space, originally due to by David and Christ (\cite{Dav88}, \cite{Christ-T(b)}), but the current formulation is from Hyt\"onen and Martikainen \cite{HM12}. A metric space $X$ is {\it doubling} if there is $N$ so that any ball can be covered by at most $N$ balls of half the radius. In practice, the metric space $X$ in the theorem will be a subset of Euclidean space and thus doubling.

{
 \begin{definition}
We say that a set $X$ is {\it $\delta$-separated} or a {$\delta$-net} if for all $x,y\in X$ we have $|x-y|\geq \delta$. 
\end{definition}
}

 \begin{theorem}
Let $X$ be a doubling metric space. Let $X_{k}$ be a nested sequence of maximal $\rho^{k}$-nets for $X$ where $\rho<1/1000$ and let $c_{0}=1/500$. For each $n\in\bZ$ there is a collection $\cD_{k}$ of ``cubes,'' which are Borel subsets of $X$ such that the following hold.
\begin{enumerate}
\item For every integer $k$, $X=\bigcup_{Q\in \cD_{k}}Q$.
\item If $Q,Q'\in \cD=\bigcup \cD_{k}$ and $Q\cap Q'\neq\emptyset$, then $Q\subseteq Q'$ or $Q'\subseteq Q$.
\item For $Q\in \cD$, let $k(Q)$ be the unique integer so that $Q\in \cD_{k}$ and set $\ell(Q)=5\rho^{k(Q)}$. Then there is $\zeta_{Q}\in X_{k}$ so that
\begin{equation}\label{e:containment}
B_{X}(\zeta_{Q},c_{0}\ell(Q) )\subseteq Q\subseteq B_{X}(\zeta_{Q},\ell(Q))
\end{equation}
and
\[ X_{k}=\{\zeta_{Q}: Q\in \cD_{k}\}.\]
\end{enumerate}
\label{t:Christ}
\end{theorem}

{
From now on we will let $\cD$ denote the cubes from \Theorem{Christ} for $E$ and will write
\[
B_{Q}=B(\zeta_{Q},\ell(Q)).\]}

Fix $\rho=10^{-4}$, which we want to be a power of $10$ for Section \ref{s:TheoremII}.

\begin{lemma}
Let $Q,R\in \cD$.
\begin{equation}\label{e:cqincr}
\mbox{If } Q\subseteq R,\mbox{ then }CB_{Q}\subseteq CB_{R} \mbox{ for all } C>1000/999. 
\end{equation}
\end{lemma}

\begin{proof}
We can assume $Q\neq R$, so $\ell(Q)<\ell(R)$. In particular, $\ell(Q)\leq \rho \ell(R)$. Thus, for $x\in CB_{Q}$, so long as $C\rho+1<C$ (which happens if $C\geq \frac{1}{1-\rho}>1000/999$ by our choice of $\rho$ in\Theorem{Christ}), we have
\[
|x-x_{R}|\leq |x-x_{Q}|+|x_{Q}-x_{R}|<C\ell(Q)+\ell(R)\leq (C\rho+1)\ell(R)\leq C\ell(R).\]
\end{proof}

Let $j\geq 0$. For $Q\in \cD_{k+j}$, we will denote by $Q^{(k)}$ the (unique) cube in $\cD_{j}$ containing $Q$.
For $R\in  \cD_{k+j}$ we denote by $\Child_k(R)$ the collection of cubes $Q\in \cD_j$ such that $Q\subset R$.
We will call  $Q^{(1)}$ the  the {\it parent} of $Q$. 
If $Q'\in \Child_{1}(Q^{(1)})$ and $Q'\neq Q$, then  we will call $Q'$ a {\it sibling} of $Q$.

\subsection{Preliminaries with $\beta_{E}^{d,p}$}
\label{s:betalemmas}

Again, we not that the Appendix (Section \ref{s:appendix}) contains preliminaries regarding the 
Choquet integral with respect to Hausdorff content, which is what is used for the definition of $\beta$.

For a ball $B$ centered on $E$, we will denote by $P_{B}^{p}$ the $d$-plane for which
\[
\beta_{E}^{d,p}(B,P_{B}^{p})=\beta_{E}^{d,p}(B).\]

\begin{lemma}\label{l:betabetainf}
Assume $0<p<\infty$, $E\subseteq \bR^{n}$, and $B$ is centered on $E$. Then
\begin{equation}
\beta_{E}^{d,p}(B)\leq \frac{2^{d/p}}{p^{1/p}}\beta_{E,\infty}^{d}(B).
\end{equation}
In particular, for $p\geq 1$, we have 
\begin{equation}\label{e:betabetainf}
\beta_{E}^{d,p}(B)\leq 2^{d}\beta_{E,\infty}^{d}(B).
\end{equation}
\end{lemma}

\begin{proof} 
Let $P$ be 
 the minimal $d$-plane for $\beta_{E,\infty}^{d}(B)$. 
Since for any set $A$ we have $\cH^{d}_{\infty}(A)\leq (\diam A)^{d}$, we know $\cH^{d}_{\infty}(E\cap B)\leq (2r_{B})^{d}$, and so
\begin{align*}
\beta_{E}^{d,p}(B)^{p} 
& \leq \frac{1}{r_{B}^{d}}\int_{0}^{\infty}\cH^{d}_{\infty}\{x\in B\cap E: \dist(x,P)>tr_{B}\}t^{p-1}dt\\
& \leq  \frac{1}{r_{B}^{d}}\int_{0}^{\beta_{E,\infty}^{d}(B)}\cH^{d}_{\infty}\{x\in B\cap E: \dist(x,P)>tr_{B}\}t^{p-1}dt \\
& \leq \frac{\cH^{d}_{\infty}(B\cap E) }{r_{B}^{d}} \int_{0}^{\beta_{E,\infty}^{d}(B)}t^{p-1}dt
\leq \frac{2^{d}}{p} \beta_{E,\infty}^{d}(B)^{p}.
\end{align*}
\end{proof}

\begin{lemma}
Assume $E\subseteq \bR^{n}$ and there is $B$ centered on $E$ so that for all $B'\subseteq B$ centered on $E$ we have $\cH^{d}_{\infty}(E\cap B')\geq c r_{B'}^{d}$. Then
\begin{equation}\label{e:betainfbeta}
\beta_{E,\infty}^{d}\ps{\frac{1}{2}B}\leq  2c^{-\frac{1}{d+1}}\beta_{E}^{d,1}(B)^{\frac{1}{d+1}}
\end{equation}
\label{l:betainfbeta}
\end{lemma}

\begin{proof}
We can assume $E$ is closed and that $r_{B}=1$. Let $P=P_{B}^{1}$ and let $y\in \cnj{\frac{1}{2}B\cap E}$ be farthest from $P$. Set $\tau =\dist(y,P)$ and note
\[
B(y,\tau/2)\subseteq \{x\in B\cap E: \dist(x,P)>t\}\mbox{ for }t\leq \tau/2.\]
Thus,
\begin{align*}
1
& \leq \frac{1}{c(\tau/2)^{d}}\cH^{d}_{\infty}(B(y,\tau/2)) \\
&\leq \frac{2^{d}}{c\tau^{d}} \avint_{0}^{\tau/2}\cH^{d}_{\infty}(x\in B\cap E: \dist(x,P)>t\}dt
\leq \frac{2^{d+1}}{c\tau^{d+1}} \beta_{E}^{d,1}(B) 
\end{align*}
Hence,
\[
\beta_{E,\infty}^{d}\ps{\frac{1}{2}B}
\leq 2\tau
\leq 2c^{-\frac{1}{d+1}}\beta_{E}^{d,1}(B)^{\frac{1}{d+1}}.\]

\end{proof}

 {
\begin{lemma}
Let $1\leq p < \infty$, $E\subseteq \bR^{n}$ be a closed set and $B$ a ball centered on $E$ so that $\cH^{d}_{\infty}(E\cap B)>0$. Then 
 \begin{equation}\label{e:bjensens}
 \beta_{E}^{d,1}(B)\lec_{n} \beta_{E}^{d,p}(B) 
 \end{equation}
\end{lemma}

\begin{proof}
We can assume without loss of generality that  $r_{B}=1$. Let $P=P_{B}^{p}$ and $f(x)=\dist(x,P)$. Let 
\[
E_{j}=E\cap \cnj{(1-j^{-1})B}.\]
Then $E$ is compact, and so we may apply \Lemma{jensens} to these sets. By Frostmann's Lemma \cite[Theorem 8.8]{Mattila}, for each $t>0$ there is a measure $\mu_t$ with
\[\supp \mu_{t}\subseteq \{x\in E\cap B: f(x)>t\}\] 
so that 
\[
\mu_{t}(B(x,r))\leq r^{d} \mbox{ for all $x\in \bR^{n}$ and $r>0$}\] 
and 
\[
\mu_{t}(\{x\in E: f(x)>t\})\gtrsim_n \cH^{d}_\infty(\{x\in E: f(x)>t\}).\]
For each $t>0$,

\begin{multline*}
\lim_{j}  \cH^{d}_{\infty} (\{x\in E_{j}\cap B: f(x)>t\})
 \geq \liminf_{j} \mu_{t} (\{x\in E_{j}: f(x)>t\})\\
=\mu_{t} (\{x\in E: f(x)>t\})  
\gtrsim_n \cH^{d}_\infty(\{x\in E: f(x)>t\})
\end{multline*}
where the first limit converges because it is monotone (since the $E_{j}$ are nested). Hence, by the Monotone Convergence Theorem,
\begin{align*}
\beta_{E}^{d,1}(B) & \leq \int_{E\cap B} f d\cH^{d}_{\infty}
= \int_{0}^{\infty} \cH^{d}_{\infty}(\{x\in E: f(x)>t\})dt\\
& \lec_n \lim_{j}  \int_{0}^{\infty} \cH^{d}_{\infty}(\{x\in E_{j}: f(x)>t\})dt \\
& = \lim_{j} \int_{E_{j}} f d\cH^{d}_{\infty}\\
& \stackrel{\eqn{jensens}}{\lec_n} \lim_{j} \cH^{d}_{\infty}(E_{j})^{1-\frac{1}{p}} \ps{ \int_{E_{j}} f^{p} d\cH^{d}_{\infty}}^{\frac{1}{p}}\\
& \leq \cH^{d}_{\infty}(E\cap B)^{1-\frac{1}{p}} \ps{ \int_{E\cap B} f^{p} d\cH^{d}_{\infty}}^{\frac{1}{p}}\\
& \lec_{d} 1 \cdot \beta_{E}^{d,p}(B).
\end{align*}
and this implies \eqn{bjensens}.
\end{proof}
}

\begin{lemma}
Let $1\leq p < \infty$ and let $E\subseteq \bR^{n}$. Then for balls $B'\subseteq B$ centered on $E$,
\begin{equation}
\beta_{E}^{d,p}(B')\leq \ps{\frac{r_{B}}{r_{B'}}}^{d+p}\beta_{E}^{d,p}(B).
\end{equation}
\end{lemma}

\begin{proof}
Let $P=P_{B}^{p}$. Using a change of variables, we get
\begin{align*}
\beta_{E}^{d,p}(B')^{p}r_{B'}^{d}
& \leq \int_{0}^{\infty} \cH^{d}_{\infty}\{x\in B'\cap E:\dist(x,P)>tr_{B'}\}t^{p-1}dt\\
& = \ps{\frac{r_{B}}{r_{B'}}}^{p}\int_{0}^{\infty}  \cH^{d}_{\infty}\{x\in B'\cap E:\dist(x,P)>tr_{B}\}t^{p-1}dt\\
& \leq  \ps{\frac{r_{B}}{r_{B'}}}^{p} \int_{0}^{\infty}  \cH^{d}_{\infty}\{x\in B\cap E:\dist(x,P)>tr_{B}\}t^{p-1}dt\\
& \leq \ps{\frac{r_{B}}{r_{B'}}}^{p}r_{B}^{d} \beta_{E}^{d,p}(B)^{p}.
\end{align*}
\end{proof}

For two planes $P,P'$ containing the origin, we define
\[\angle(P,P')=d_{B(0,1)}(P,P').\]
For general affine planes $P,P'$, let $x\in P$ and $y\in P'$. We set 
\[\angle(P,P')=\angle(P-x,P'-y).\]

If $P_{1}$ and $P_{2}$ are both $d$-planes containing $x$, we clearly have
\begin{equation}\label{e:dxardxr}
d_{x,ar}(P_{1},P_{2})=d_{x,r}(P_{1},P_{2}) \mbox{ for all }a,r>0
\end{equation}
and it is also not hard to show that if $P_{1},P_{2}$, and $P_{3}$ are all $d$-planes containing $x$,
\begin{equation}
\label{e:ptrngl}
d_{x,r}(P_{1},P_{3})
\leq d_{x,r}(P_{1},P_{2})+d_{x,r}(P_{2},P_{3}).
\end{equation}
Indeed, if $y\in P_{1}\cap B(x,r)$, $y'=\pi_{P_{2}}(y)$, and $y''=\pi_{P_{3}}(y')$, then $y,y',y''\in B(x,r)$ and
\begin{align*}
r\dist(y,P_{3})
& \leq |y-y''|\leq |y-y'|+|y'-y''|
= \dist(y,P_{2})+\dist(y',P_{3})\\
& \leq rd_{x,r}(P_{1},P_{2})+rd_{x,r}(P_{2},P_{3})
\end{align*}
and supremizing over all $y\in P_{1}\cap B(x,r)$ gives \eqn{ptrngl}.

\begin{lemma}
Suppose $d_{z,r}(V_{1},V_{2})<\ve$ are $d$-planes. Then
\begin{equation}\label{e:vvperp}
|\pi_{V_{1}^{\perp}}(x-y)|\leq \ve |x-y| \;\;\; \mbox{ for }x,y\in V_{2}.
\end{equation}
\end{lemma}

\begin{proof}
Let $x,y\in V_{2}$. By subtracting $y$ from $V_{1}$ and $V_{2}$, we may assume without loss of generality that $y=0$, so $V_{2}$ contains the origin. Recall that $V_{1}^{\perp}$ is the $(n-d)$-plane orthogonal to $V_{1}$ and containing the origin. Let $V_{1}'$ be the translate of $V_{1}$ that contains the origin. Then
\begin{align*}
|\pi_{V_{1}^{\perp}}(x)|
& =|x-\pi_{V_{1}'}(x)|
\leq\dist(x,V_{1}')\leq |x| d_{B(0,|x|)}(V_{2},V_{1}') \\
& \stackrel{\eqn{dxardxr}}{=}|x| d_{B(0,1)}(V_{2},V_{1}')=\angle(V_{1},V_{2}) <\ve|x|.
\end{align*}

\end{proof}

\begin{lemma}\label{l:betabeta'}
Suppose $E\subseteq \bR^{n}$ and there is $B$ centered on $E$ so that for all $B'\subseteq B$ centered on $E$ we have $\cH^{d}|_{\infty}(B')\geq c r_{B'}^{d}$. Let $P$ and $P'$ be two $d$-planes. Then
\begin{equation}
d_{B'}(P,P')\lec_{d,c} \ps{\frac{r_{B}}{r_{B'}}}^{d+1}\beta_{E}^{d,1}(B,P)+\beta_{E}^{d,1}(B',P')
\end{equation}
and in particular
\begin{equation}
\angle(P,P')\lec_{d,c} \ps{\frac{r_{B}}{r_{B'}}}^{d+1}\beta_{E}^{d,1}(B,P)+\beta_{E}^{d,1}(B',P').
\end{equation}
In particular, it follows that
\begin{equation}
d_{B'}(P_{B'}^{1},P_{B}^{1})\lec_{d,c} \ps{\frac{r_{B}}{r_{B'}}}^{d+1}\beta_{E}^{d,1}(B)
\end{equation}
and 
\begin{equation}
\angle(P_{B'}^{1},P_{B}^{1})\lec_{d,c} \ps{\frac{r_{B}}{r_{B'}}}^{d+1}\beta_{E}^{d,1}(B).
\end{equation}
\end{lemma}

The proof of Lemma \ref{l:betabeta'}  will use the following lemma.

\begin{lemma}\cite[Lemma 6.4]{AT15} Suppose $P_{1}$ and $P_{2}$ are $d$-planes in $\bR^{n}$ and $X=\{x_{0},...,x_{d}\}$ are points so that
\begin{enumerate}
\item[(a)] $\eta\in (0,1)$ where 
\[
\eta=\eta(X)=\min\{\dist(x_{i},\spn X\backslash\{x_{i}\})/\diam X\in (0,1)\] 
and
\item[(b)] $\dist(x_{i},P_{j})<\ve\,\diam X$ for $i=0,...,d$ and $j=1,2$, where $\ve<\eta d^{-1}/2$.
\end{enumerate}
Then
\begin{equation}
\dist(y,P_{1}) \leq \ve\ps{\frac{2d}{\eta}\dist(y,X)+\diam X} \;\; \mbox{ for all }y\in P_{2}.
\end{equation}
\label{l:ATlemma}
\end{lemma}

\begin{proof}[Proof of Lemma \ref{l:betabeta'}]

Let $X=\{x_{0},...,x_{d}\}$ be vectors in $\frac{1}{2}B'\cap E$ such that 
\[ \dist(x_{i+1},\spn \{x_{0},...,x_{i}\})\gec r_{B'} \mbox{ for }i=1,...,d.\] 
These can be found by induction using the fact that $\cH^{d}_{\infty}(E\cap B')\geq c r_{B'}^{d}$ (see for example \cite[Section 5]{DS}). Then $\eta(X)\sim 1$. Let $B_{i}=B(x_{i},r_{B'}/100)$ and for $t>0$, set
\[
E_{t,i}=\{x\in E\cap B_{i}: \dist(x,P')>tr_{B'}\mbox{ or } \dist(x,P)>tr_{B'}\}.\]
Let $T>0$.
Suppose $E_{t,i}=B_{i}\cap E$ for all $t \leq T$. Observe that 
\[
\cH^{d}_{\infty}(B_{i}\cap E)\geq cr_{B_{i}}^{d}=\frac{c}{100^{d}}r_{B'}^{d}.\]
Hence,
\begin{align*}
T
& \leq \cH^{d}_{\infty}(B_{i}\cap E)^{-1} \int_{0}^{T} \cH^{d}_{\infty}(E_{t})dt
\lec_{c} 
 r_{B_{i}}^{-d}\int_{0}^{T} \cH^{d}_{\infty}(E_{t})dt \\
& \lec r_{B'}^{-d} \int_{0}^{T}\cH^{d}_{\infty}\{x\in E\cap B_{i}: \dist(x,P')>tr_{B'}/2\}dt\\
& \qquad  + r_{B'}^{-d} \int_{0}^{T}\cH^{d}_{\infty}\{x\in E\cap B_{i}: \dist(x,P)>tr_{B'}/2\}dt\\
& \leq \frac{2}{r_{B'}^{d}}\int_{0}^{\infty}\cH^{d}_{\infty}\{x\in E\cap B': \dist(x,P')>tr_{B'}\}dt\\
& \qquad  +\frac{2r_{B}^{d+1}}{r_{B'}^{d+1}}  \frac{1}{r_{B}^{d}}\int_{0}^{\infty}\cH^{d}_{\infty}\{x\in E\cap B: \dist(x,P)>tr_{B}\}dt\\
& \lec \beta_{E}^{d,1}(B',P')+\ps{\frac{r_{B}}{r_{B'}}}^{d+1}\beta_{E}^{d,1}(B,P)=:\lambda.
\end{align*}
Thus, there is $c'=c'(d,c)$ so that $T\leq c' \lambda$. Hence, for $t=2c'\lambda$, $(B_{i}\cap E)\backslash E_{t,i}\neq \emptyset$, and so there are points $y_{i}\in B_{i}\backslash E_{c't,i}$. Since $y_{i}\in B_{i}$ and $\eta(X)\sim 1$, it is not hard to show that $\eta(\{y_{0},...,y_{d}\})\sim 1$ as well. By the definition of $E_{t,i}$, the lemma follows from  \Lemma{ATlemma}. 
\end{proof}

\begin{lemma}
Let $M>1$, $E$ be a Borel set so that $\cH^{d}_{\infty}(E\cap B)\geq cr_{B}^{d}$ for all balls $B$ centered on $E$ with $0<r<\diam E$. Let $\cD$ be the cubes from \Theorem{Christ} for $E$ (or $E\cap B(0,1)$), and $Q_{0}\in \cD$. For $Q\in\cD$, let $P_{Q}=P_{MB_{Q}}^{1}$.
Suppose that for all balls $B\subseteq B_{Q_{0}}$ centered on $E$ that $\cH^{d}_{\infty}(B\cap E)\geq cr^{d}$. Let $Q,R\in \cD$, $Q,R\subseteq Q_{0}$, and suppose that for all cubes $S\subseteq Q_{0}$ so that $S$ contains either $Q$ or $R$ that 
$\beta_{E}^{d,1}(MB_{Q})<\ve$. 
Then for $\Lambda>0$, if $\dist(Q,R)\leq \Lambda\max\{\ell(Q),\ell(R)\}\leq \Lambda^{2}\min\{\ell(Q),\ell(R)\}$, then
\begin{equation}
\angle(P_{Q},P_{R})\lec_{M,\Lambda}\ve.
\label{e:PQPR}
\end{equation}
\label{l:PQPR}
\end{lemma}

\begin{proof}
Suppose $\ell(Q)\leq \ell(R)$. Let $Q'$ be the smallest ancestor of $Q$ so that $3B_{R}\subseteq 3B_{Q'}$ (note that since both $Q$ and $R$ are contained in $Q_0$ and hence $3B_{R}\subseteq 3B_{Q_0}$ by \eqn{cqincr}, this is well defined). Then 
\begin{equation}\label{e:RQ'dQR+Q}
\ell(R)\sim \ell(Q')\sim \dist(Q,R)+\ell(Q)+\ell(R) \lec \Lambda^{2}\ell(Q)
\end{equation}
and so by 
\Lemma{betabeta'},
\[\angle(P_{Q},P_{3B_{Q'}})
\stackrel{\eqn{ptrngl}}{\leq} \angle(P_{Q},P_{3B_{Q}})+\angle(P_{3B_{Q}},P_{3B_{Q'}})
\lec_{\Lambda,M} \ve.\]
Let $R'$ be the largest ancestor of $R$ so that $3B_{R'}\subseteq 6B_{Q'}$. Again, 
\[
 \ell(R')\sim \ell(Q') \sim \ell(R)\]
and so $\ell(R')\lec \ell(Q)$ by \eqn{RQ'dQR+Q}. Hence, applying \Lemma{betabeta'}
\begin{multline*}
\angle(P_{R},P_{3B_{Q'}})
\stackrel{\eqn{ptrngl}}{\leq} \angle(P_{R},P_{3B_{R}})+\angle(P_{3B_{R}},P_{3B_{R'}})+\angle(P_{3B_{R'}},P_{6B_{Q'}}) \\
+\angle(P_{6B_{Q'}},P_{3B_{Q'}})
\lec_{\Lambda,M} \ve.
\end{multline*}
Combining these two chains of inequalities proves the lemma.
\end{proof}

\begin{lemma}
\label{l:bigproj}
Suppose $E$ is $\ve$-Reifenberg flat, $0<\ve<\ve_{0}$, $B$ is a ball centered on $E$, and $P$ is some $d$-plane. For $\ve_{0}$ small enough, if $\beta_{E,\infty}^{d}(B,P)<1/100$, then
\begin{equation}
\pi_{P}(B\cap E)\supseteq B(\pi_{P}(x_{B}),r_{B}/2).
\label{e:bigproj}
\end{equation}
\end{lemma}

\begin{proof}
Without loss of generality, we will assume $B=\bB$ and $P=\bR^{d}+\pi_{P}(0)$. Then
\[
\beta_{E,\infty}^{d}(\bB,\bR^{d})
\leq 2\beta_{E,\infty}^{d}(\bB,P)<\frac{1}{50}.\]

Let $g:\bR^{n}\toitself$ be the map from \Theorem{DT} with $P_{0}=\bR^{d}$, $\pi$ the orthogonal projection onto $\bR^{d}$, and $h=\pi\circ g$. Then $E\cap \bB=g(\bR^{d})\cap \bB$. For $x\in \frac{3}{4}\bB$, $|g(x)-x|\leq C\ve$ for some $C=C(d)$ by \Theorem{DT} (4), and so $g(x)\in (\frac{3}{4}+C\ve)\bB\cap E \subseteq \bB\cap E$. Thus, for $x\in \bR^{d}\cap \frac{3}{4} \bB$, $\dist(g(x),\bR^{d})\leq \beta_{E,\infty}^{d}(\bB,\bR^{d})$. Combining these estimates, we get that for $x\in \frac{3}{4}\bB_{d}$ and $\ve_{0}>0$ small enough,
\[
|h(x)-x|\leq |\pi(g(x))-g(x)|+|g(x)-x|\leq \frac{1}{50}+C\ve<\frac{1}{10}.\]
If there is $x_{0}\in \frac{1}{2}\bB_{d}\backslash h(\frac{3}{4}\bB_{d})$, then it is not hard to show that $h|_{\frac{3}{4}\bB_{d}}$ is contractible in $\bR^{d}\cap \frac{3}{4} \bB \backslash \{x_{0}\}$ (since it is homotopic to a map of $\frac{3}{4}\bB_{d}$ into $\frac{3}{4}\bS^{d-1}$ and maps of the disk into the sphere are always contractible). Thus $h|_{\frac{3}{4}\bS^{d-1}}$ is contractible, but $h(x,t)=(t h(x)+(1-t)x)/|th(x)+(1-t)x|$ for $x\in \frac{3}{4}\bS_{d}$ is homotopic in $\bR^{d}\cap \frac{3}{4} \bB\backslash\{x_{0}\}$ to the identity map on $\frac{3}{4}\bS^{d}$, which is not contractible in $\bR^{d}\cap \frac{3}{4} \bB\backslash\{x_{0}\}$, and we get a contradiction. We have thus shown $\frac{1}{2}\bB_{d}\subseteq h(\frac{3}{4}\bB_{d})\subseteq \pi(E\cap \bB)$. It then follows that 
\begin{equation}
\pi_{P}(E\cap \bB)=\pi_{P}(\pi_{\bR^{d}}(E\cap \bB))\supseteq \pi_{P}(\tfrac{1}{2}\bB_{d})=B(\pi_{P}(0),1/2)\cap P.
\end{equation}
\end{proof}

\begin{lemma}
Suppose $E$ is $\ve$-Reifenberg flat, $0<\ve<\ve_{0}$, $B$ is a ball centered on $E$, and $P$ is some $d$-plane. There is $\beta_{0}>0$ small so that if $\beta_{E,\infty}^{d}(B,P)<\beta_{0}$, then
\begin{multline}\label{e:thetabeta}
2^{-d}\beta(\tfrac{1}{2}B,P)\leq \beta_{E,\infty}^{d}(\tfrac{1}{2}B,P)\leq d_{\frac{1}{2}B}(E,P) \leq 2\beta_{E,\infty}^{d}(B,P) \\
\leq \frac{4}{(1-C\ve)\omega_{d}}  \beta_{E}^{d,1}(2B,P)^{\frac{1}{d+1}}.\end{multline}
In particular, 
\begin{multline}
2^{-d}\beta(\tfrac{1}{2}B)\leq \beta_{E,\infty}^{d}(\tfrac{1}{2}B)\leq \vartheta_{E}^{d}\ps{\tfrac{1}{2}B}\leq 2\beta_{E,\infty}^{d}(B) \\
\leq \frac{4}{(1-C\ve)\omega_{d}}  \beta_{E}^{d,1}(2B)^{\frac{1}{d+1}}.
\end{multline}
\end{lemma}

\begin{proof}
By the previous lemma, \eqn{bigproj} holds. In particular, any $x\in B(\pi_{P}(0),r_{B}/2)\cap P$ is the image of a point $y\in B\cap E$ under $\pi_{P}$, and thus $\dist(x,E)\leq \beta_{E,\infty}^{d}(B,P)$. By definition, any $x\in E\cap \frac{1}{2} \bB$ is at most $\beta_{E,\infty}^{d}(B,P)$ from $P$. We have thus shown that $d_{\frac{1}{2}\bB}(E,P)\leq 2\beta_{E,\infty}^{d}(B,P)$. The other inequalities in \eqn{thetabeta} follow from Lemmas \ref{l:betabetainf}, \ref{l:betainfbeta}, and the definitions.
\end{proof}

The following lemma says that the $\beta_{E}^{d,p}$-numbers for two sets are approximately the same, with error depending on the average distance between one and the other, where ''average" is taken with respect to a Choquet integral.
\preRef{
\begin{lemma}\label{l:2-15}
Let $E_{1},E_{2}\subseteq \bR^{n}$. Suppose $B^{1}$ is a ball centered on $E_{1}$ and $B^{2}$ is a ball of same radius but centered in $E_{2}$ so that $B_{1}\subseteq 2B_{2}$.  
Suppose that for $i=1,2$ and all balls $B\subseteq 2B^{i}$ centered on $E_{i}$ we have 
$\cH^{d}|_{\infty}(B\cap E_i)\geq c r_{B}^{d}$ 
for some $c>0$. Then
\begin{equation}\label{e:pushbeta}
\beta_{E_{1}}^{d,p}(B^{1},P)
\lec_{c,p} \beta_{E_{2}}^{d,p}(x,2B^{2},P)+\frac{1}{r_{B^{1}}^{d}}\int_{E_{1}\cap 2B^{1}} \frac{\dist(y,E_{2})}{r_{B^{1}}}d\cH^{d}_{\infty}(y).\end{equation}
\end{lemma}
}

\postRef{\begin{lemma}\label{l:2-15}
Let $1\leq p < \infty$ and $E_{1},E_{2}\subseteq \bR^{n}$. Suppose $B^{1}$ is a ball centered on $E_{1}$ and $B^{2}$ is a ball of same radius but centered in $E_{2}$ so that $B^{1}\subseteq 2B^{2}$.  
Suppose that for $i=1,2$ and all balls $B\subseteq 2B^{i}$ centered on $E_{i}$ we have 
$\cH^{d}|_{\infty}(B\cap E_i)\geq c r_{B}^{d}$ 
for some $c>0$. Then
\begin{equation}\label{e:pushbeta}
\beta_{E_{1}}^{d,p}(B^{1},P)
\lec_{c,p} \beta_{E_{2}}^{d,p}(2B^{2},P)+\ps{\frac{1}{r_{B^{1}}^{d}}\int_{E_{1}\cap 2B^{1}} \ps{\frac{\dist(y,E_{2})}{r_{B^{1}}}}^{p}d\cH^{d}_{\infty}(y)}^{\frac{1}{p}}\end{equation}
\end{lemma}}

We will first need an intermediary lemma.
\begin{lemma}\label{l:tech-integrating-sum}
Let $F_{1},F_{2}\subseteq \bR^{n}$. Suppose $B_{0}$ is a ball centered on $F_{2}$ and for all balls $B\subseteq 2B_{0}$ centered on $F_{2}$ 
we have 
$\cH^{d}|_{\infty}(B\cap F_2)\geq c r_{B}^{d}$ 
for some $c>0$. Finally, let $\alpha\in (0,1)$, $\{B_{j}\}_{j\in \cX}$ be a collection of disjoint balls with centers $z_{j}\in F_{1}$ such that $\dist(z_{j},F_{2})< \alpha r_{B_{j}} $, $f$ a nonnegative function on $\{z_{j}\}$, $z_{j}'\in F_{2}\cap B(z_{j},\alpha r_{B_{j}})$, $B_{j}'=B(z_{j}',r_{B_{j}})$. Then 
\begin{multline}\label{e:1<4}
\int_{F_{1}} \sum_{j\in \cX} \one_{B_{j}}(y)f(z_{j}) d\cH^{d}_{\infty}(y)\\
\lec_{\alpha} \int_{F_{2}} \sum_{j\in \cX} \one_{(1-\alpha)B_{j}'}(y)f(z_{j}) d\cH^{d}_{\infty}(y).\end{multline}
\end{lemma}
The proof of  Lemma \ref{l:tech-integrating-sum} may be found in the Appendix (Section \ref{s:appendix}).

\begin{proof}[Proof of Lemma \ref{l:2-15}]
Without loss of generality, we may assume $B^{1}=\bB$, the unit ball, so $r_{B^{1}}=1$. Let $\delta>0$ be small and $\delta(z)=\dist(z,E_{2})+\delta$. We may assume without loss of generality that $\delta(z)<1/8$ for all $z\in E_{1}\cap \bB$, for otherwise the inequalities are trivial by lower regularity. By the Besicovitch covering theorem, there are collections of points $\cX_{1},...,\cX_{N}\subseteq E_{1}\cap \bB$ with $N=N(n)$ so that if 
\[
B_{z}=B(z,2\delta(z)),\] 
then the balls $\{B_{z}:z\in \cX_{k}\}$ are pairwise disjoint and
\begin{equation}\label{e:e1overlap}
\one_{E_{1}\cap \bB}\leq \sum_{k=1}^{N} \sum_{z\in \cX_{k}} \one_{B_{z}} \lec_{n} 1.
\end{equation}

Note that since $\delta(z)<1/8$, $B_{z}\subseteq 2\bB$ for all $z\in E_{1}\cap \bB$. Since $d(\cdot,P)$ is $1$-Lipschitz, we have that for $y\in B_{z}$,

\[
\dist(y,P)^{p}\leq (\dist(z,P)+2\delta(z))^{p}\lec_{p} \dist(z,P)^{p}+\delta(z)^{p}\]
\preRef{$d(\cdot,P)^{p}$ is $L$-Lipschitz on $2\bB$ with $L$ depending on $p$.}  Hence,
\begin{align}\label{e:be1}
\beta_{E_{1}}^{d,p}(\bB,P)^{p}
& = \int_{E_{1}\cap \bB} \dist(y,P)^{p}d\cH^{d}_{\infty}(y) \notag \\
& \stackrel{\eqn{intineq} \atop \eqn{e1overlap}}{\lec} 
\sum_{k=1}^{N} \int_{E_{1}} \sum_{z\in \cX_{k}} \one_{B_{z}}(y)\dist(y,P)^{p}d\cH^{d}_{\infty}(y) \notag \\
& \stackrel{\eqn{intineq}}{\lec}  
\sum_{k=1}^{N} \int_{E_{1}} \sum_{z\in \cX_{k}} \one_{B_{z}}(y)\dist(z,P)^{p}d\cH^{d}_{\infty}(y) \notag \\
& \qquad +   \sum_{k=1}^{N} \int_{E_{1}} \sum_{z\in \cX_{k}} \one_{B_{z}}(y)\delta(z)^{p}d\cH^{d}_{\infty}(y) \notag \\
& = I_{1}+I_{2}.
\end{align}
By \Lemma{tech-integrating-sum} with $F_{1}=F_{2}=E_{1}$, $f(z)=\delta(z)^{p}$ (since this is also Lipschitz), and $\alpha=\frac{7}{8}$, we have that 
\[
I_{2}
\lec  \sum_{k=1}^{N} \int_{E_{1}} \sum_{z\in \cX_{k}} \one_{\frac{1}{8}B_{z}}(y)\delta(z)^{p}d\cH^{d}_{\infty}(y) .\]
Since $\delta$ is $1$-Lipschitz, we have for each 
$z$ and $y\in E_{1}\cap \frac{1}{8}B_{z}$ where $\frac{1}{8}B_{z}= B(z,\delta(z)/4)$,
\[
\delta(z)\leq \delta(y) + |z-y|< \delta(y) + \frac{\delta(z)}{4}.\]
Hence $\delta(z)<\frac{4}{3} \delta(y)$, and so for $1\leq k\leq N$,
\[
\int_{E_{1}} \sum_{z\in \cX_{k}} \one_{\frac{1}{8}B_{z}}(y)\delta(z)^{p}d\cH^{d}_{\infty}(y) 
\leq \frac{4}{3}
\int_{E_{1}} \sum_{z\in \cX_{k}} \one_{\frac{1}{8}B_{z}}(y)\delta(y)^{p}d\cH^{d}_{\infty}(y) .
\]
Thus, since the $B_{z}$ are disjoint for $z\in \cX_{k}$, and since $N$ depends only on $n$,
\begin{multline}\label{e:pushI2}
I_{2} 
\lec \sum_{k=1}^{N}  \int_{E_{1}} \sum_{z\in \cX_{k}} \one_{\frac{1}{8}B_{z}}(y)\delta(y)^{p}d\cH^{d}_{\infty}(y)
\leq \sum_{k=1}^{N}  \int_{E_{1}\cap 2B^{1} }\delta(y)^{p}d\cH^{d}_{\infty}(y)\\
\lec_{n}\int_{E_{1}\cap 2B^{1} }\delta(y)^{p}d\cH^{d}_{\infty}(y).
\end{multline}

To bound $I_{1}$, again by \Lemma{tech-integrating-sum} but with $F_{1}=E_{1}$, $F_{2}=E_{2}$, 
{$f=\dist(\cdot,P)^{p}$, and $\alpha=3/4$, and because $\dist(\cdot,P)^{p}$ is Lipschitz in $2\bB$}, we have
\begin{align*}
I_{1}
& \lec \sum_{k=1}^{N} \int_{E_{2}} \sum_{z\in \cX_{k}} \one_{\frac{1}{4} B_{z}'}(y)\dist(z,P)^{p}d\cH^{d}_{\infty}(y)\\
& \stackrel{\eqn{intineq}}{\lec} 
\sum_{k=1}^{N} \int_{E_{2}} \sum_{z\in \cX_{k}} \one_{\frac{1}{4} B_{z}'}(y)\dist(y,P)^{p}d\cH^{d}_{\infty}(y)\\
& \qquad +\sum_{k=1}^{N} \int_{E_{2}} \sum_{z\in \cX_{k}} \one_{\frac{1}{4} B_{z}'}(y)\delta(z)^{p}d\cH^{d}_{\infty}(y)\\
& = I_{11}+I_{12}
\end{align*}
Note that 
\[
\frac{1}{4}B_{z}'= B(z',\delta(z)/2)\subseteq B(z,4\delta(z)/2)= B_{z}\]
and so the $\{\frac{1}{4}B_{z}'\}_{z\in \cX_{k}}$ are disjoint. Hence,
\begin{align*}
I_{11} 
& \leq \sum_{k=1}^{N} \int_{E_{2}\cap 2B_{2}} \dist(z,P)^{p}d\cH^{d}_{\infty}(y).
\end{align*}
Next, by \Lemma{tech-integrating-sum} with $F_{1}=E_{2}$, $F_{2}=E_{1}$, $\{\frac{1}{4}B_{z}'\}_{z\in \cX_{k}}$  as our collection of balls with centers $z'\in F_{1}$, $f(z')=\delta(z)^{p}$, and $\alpha=\frac{1}{2}$,
\[
I_{12}
\lec \sum_{k=1}^{N} \int_{E_{1}} \sum_{z\in \cX_{k}} \one_{\frac{1}{8}B_{z}}(y)\delta(z)^{p}d\cH^{d}_{\infty}(y)\stackrel{\eqn{pushI2}}{\lec} \int_{E_{1}\cap 2B_{1} }\delta(y)^{p}d\cH^{d}_{\infty}(y).
\]
Combining our estimates for $I_{1}$, $I_{2}$, $I_{12}$, and $I_{22}$, we obtain 
{
\[
\beta_{E_{1}}^{d,p}(B^{1},P)^{p}
\lec_{c,p} \beta_{E_{2}}^{d,p}(2B^{2},P)^{p}+\int_{E_{1}\cap 2B^{1}} \delta(y)^{p} d\cH^{d}_{\infty}(y)
\]
Now using the fact that $(a+b)^{\frac{1}{p}}\leq 2^{\frac{1}{p}}\max\{a,b\}^{\frac{1}{p}}\leq 2^{\frac{1}{p}}(a^{\frac{1}{p}}+b^{\frac{1}{p}})$, we are done. 
}

\end{proof}

It is not hard to show that \Theorem{dorronsoro} can be rewritten in the following way. For a cube $I$ in $\bR^{d}$ and $f:\bR^{d}\rightarrow \bR^{n-d}$ Lipschitz, we set
\begin{equation}\label{e:omdyadic}
\Omega_{f,p}(I)^{p} = \inf_{A} \avint_{I}\av{\frac{f-A}{r}}^{p}
\end{equation}
Then
\begin{equation}
\label{e:dorronsorocubes}
\sum_{I\subseteq \bR^{d}}\Omega_{f,p}(3I)^{2}\cH^d(I)\lec_{p,n} ||\grad f||_{2}^{2}
\end{equation}
where the sum is taken over all dyadic cubes $I$ in $\bR^{d}$.

\begin{lemma}
\label{l:b-to-om}
Let $f:\bR^{d}\rightarrow \bR^{n-d}$ be an $L$-Lipschitz function , $\Gamma=\{(x,f(x)):x\in \bR^{d}\}$ be its graph, $x\in \Gamma$, $r>0$, and let $I$ be a cube in $\bR^{d}$ that contains $\pi_{\bR^{d}}(B(x,r))$. Then
\begin{equation}\label{e:b-to-om}
\beta_{\Gamma}^{d,p}(x,r) \lec_{d,L,p} \ps{\frac{\cH^d(I)}{r^{d}}}^{1/p} \Omega_{f,p}(I).
\end{equation}
\end{lemma}
\begin{proof}
Let $A$ be the plane that infimizes \eqn{omdyadic}. Since $f$ is $L$-Lipschitz, we have that $\Gamma$ is $C$-Ahlfors regular with $C$ depending on $d$ and $L$. In particular, $\cH^{d}(E)\sim_{L} \cH^{d}_{\infty}(E)$ for all $E\subseteq \Gamma$. Set $F(x)=(x,f(x))$, which is $2L$-Lipschitz. 
\begin{align*}
\beta_{\Gamma,p}^{d}(x,r)^{p}
& \leq r^{-d}\int \cH^{d}_{\infty}(\{y\in \Gamma\cap B(x,r): \dist(y,A(\bR^{d}))>tr\})t^{p-1}dt\\
& \sim_{L} r^{-d}\int \cH^{d}(\{y\in \Gamma\cap B(x,r): \dist(y,A(\bR^{d}))>tr\})t^{p-1}dt\\
& = r^{-d}\int_{B(x,r)\cap \Gamma} \ps{\frac{\dist(y,A(\bR^{d}))}{r}}^{p}d\cH^{d}|_{\Gamma}(y) \\
& = r^{-d}\int_{\pi_{\bR^{d}}(B(x,r))} \ps{\frac{\dist(F(z),A(\bR^{d}))}{r}}^{p}J_{F}(z)dz\\
& \lec_{L} r^{-d}\int_{I} \ps{\frac{|F(z) - A(z)|}{r}}^{p}dz\\
& \leq \frac{\cH^d(I)}{r^{d}} \Omega_{f,p}(I)^{p}
\end{align*}
\end{proof}

\section{The theorems we actually prove
}\label{s:state-flat-case}

{

We first reformulate Theorem \ref{t:thmii-dyadic} with an equivalent version that uses balls and nets rather than dyadic cubes. This is slightly more technical to state, but will be more natural to prove. First, we define $\largetheta{E}^{d,A,\ve}$ in terms of a net (as opposed to $\largetheta{E}^{d,\Delta,\ve}$ from definition \ref{d:def-theta-delta}).

\begin{definition}
Let $E\subseteq \bR^{n}$ be a Borel set, $0<\rho<1/1000$. For $k\in \bZ$, let $X_{k}$ be a sequence of maximal $\rho^{-k}$-separated sets of points for $E$ and $\cX=\{B(x,\rho^{-k}):k\in \bZ,\; x\in X_{k}\}$.
For $\ve,A,r>0$ and $x\in E$, we define
\begin{multline*}
\largetheta{E}^{d,A,\ve}(x,r):=\sum \{r_{B}^{d}:B\in \cX,\;\; x_{B}\in B(x,r), \\  0<r_{B}<r, \mbox{ and }
\vartheta_{E}(AB)\geq \ve\}.
\end{multline*}

\end{definition}

\begin{theorem}
Let $1\leq d<n$, $C_{0}>1$, and $A>\max\{C_{0},10^{5}\}$. 
Let $1\leq p<p(d)$. 
Let $E\subseteq \bR^{n}$ be a closed set containing 0.
Suppose that $E$ is $(c,d)$-lower content regular in $B(0,1)$. 
There is $\ve_{0}=\ve_{0}(n,A,p,c)>0$ such that the following holds. 
Let  $0<\ve<\ve_{0}$. 
For integer $k\geq 0$, let $X_{k}\subset X_{k+1}$ be a maximal $2^{-k}$-separated set of points in $E\cap B(0,1)$.
Suppose further that for each $k$ we have  $X_{k}\subset X_{k+1}$. Let
 $\cX_{k}=\{B(x,2^{-k}):x\in X_{k}\}$. Then
\begin{multline}
1+\sum_{k\geq 0}\sum_{B \in \cX_{k}} \beta_{E}^{d,p}(C_{0}B)^{2}r_{B}^{d}\\
\leq C(A,C_{0},n,\ve,p,c) 
\ps{\cH^{d}(E\cap B(0,1)) + \largetheta{E}^{d,A,\ve}(0,1)}.
\label{e:beta<hd2}
\end{multline}
\label{t:thmii}
\end{theorem}

Similarly, Theorem \ref{t:thmiii-dyadic} has a version for nets as follows.
\begin{theorem}
Let $1\leq d<n$.
Let $1\leq p\leq \infty$ and $E\subseteq \bR^{n}$ be such that $0\in E$. 
Suppose that $E$ is $(c,d)$-lower content regular in $B(0,1)$. 
Let $X_{k}$ be a nested sequence of maximal $2^{-k}$-separated points in $E$ and $\cX_{k}=\{B(x,2^{-k}):x\in X_{k}\}$. 
Let $A>1$ and $\ve>0$ be given as well.
Then for $C_0$ sufficiently large (depending only on $n$ and $A$),
\begin{multline}\label{e:thmiii-upper-bound}
 \cH^{d}(E\cap B(0,1)) +
      \largetheta{E}^{d,A,\ve}(0,1)\\
 \leq C(A,n,c, \ve)\ps{1+\sum_{k\geq 0}\sum_{B \in \cX_{k}\atop x_{B}\in B(0,1)} \beta_{E}^{d,p}(C_0B)^{2}r_{B}^{d}}.
 \end{multline}
{
Furthermore, if the right hand side of \eqref{e:thmiii-upper-bound} is finite, then $E$ is $d$-rectifiable
}

\label{t:thmiii}
\end{theorem}

It is not hard to show that \Theorem{thmii} implies Theorem \ref{t:thmii-dyadic}.  To do this, and for what follows, we will need to relate the two versions of $\largetheta{}$. 
Indeed
\[ \largetheta{E}^{d,\Delta,\ve}(B(0,1)) \lec_{A,n,\ve}  1+\largetheta{E}^{d,A,\ve}(0,1),\]
and
\[ \largetheta{E}^{d,A,\ve}(B(0,1)) \lec_{A,n,\ve}  \largetheta{E}^{d,\Delta,\ve}(0,1).\]
The sketch of the argument for the first inequality is that cubes of diameter less than one are contained in balls from some $\cX_{k}$, so we can use monotonicity of $\vartheta_{E}$; the cubes of diameter larger than one have geometrically decaying $\vartheta$, and so only $\sim\log(\epsilon)$ generations need to be counted. The second inequality follows in a similar manner (but note that we have cubes of all scales). Moreover, it is not hard to show that $\beta_{E}^{d,p}(Q)\lec \beta_{E}^{d,p}(B)$ if $Q$ is a cube and $B\supseteq Q$ is a ball of comparable size. Thus, we have 

\[
\sum_{Q\in \Delta \atop \ell(Q)\leq 1,  Q\cap E\cap B(0,1)\neq\emptyset} \beta_{E}^{d,p}(C_{0}Q)^{2}\diam(Q)^{d}
\lec 1+ \sum_{k\geq 0}\sum_{B \in \cX_{k}} \beta_{E}^{d,p}(C_{0}B)^{2}r_{B}^{d}.\]
Again, this follows since the cubes of sidelength much less than one are contained in a ball $C_{0}B$ from the sum on the right with comparable size and the number of cubes associated to each ball can be chosen to be uniformly bounded. 
This proves the above inequality, and thus  we can reduce from the dyadic version Theorem \ref{t:thmii-dyadic} to the ball version, i.e.  Theorem \ref{t:thmii}. A similar argument shows Theorem \ref{t:thmiii} implies Theorem \ref{t:thmiii-dyadic}. \postRef{The main thing to note is that any ball $B$ has a dyadic cube $Q$ so that $B\subseteq C_{0}Q$ and $r_{B}\sim_{C_{0},n}\diam(Q)$. 

}

\begin{remark}
We will from now on focus on showing Theorems \ref{t:thmii} and \ref{t:thmiii}, as they will imply our main results.
\end{remark}

\begin{remark}
One open question is whether one can generalize these techniques to give a version of our main results that hold for Hilbert spaces. We believe this is possible after suitably proving some of the results from \cite{DT12} for this setting. Such a result would take on this form of statement, using nets and balls instead of dyadic cubes.
\end{remark}

Theorem \ref{t:flat-case} just below  is a version  of Theorem \ref{t:thmii} specialized to Reifenberg flat surfaces. Theorem \ref{t:flat-case} will be used in the proof of Theorem \ref{t:thmii}.

\begin{theorem}\label{t:flat-case}
{
Let $1\leq d<n$, $C_{0}>1$, $1\leq p<\frac{2d}{d-2}$ if $d>2$ and $1\leq p<\infty$ if $d\leq 2$ . There is $\ve_{0}=\ve_{0}(n,C_{0},p)>0$ such that the following holds. Let $0<\ve<\ve_{0}$, $\Sigma$ be a $(\ve,d)$-Reifenberg flat surface so that for some $d$-plane $P_{0}$, 
\begin{equation}\label{e:S-B10}
\Sigma\backslash V_{0}^{1}=\Sigma\backslash B(0,1)=P_{0}\backslash B(0,1).
\end{equation}
Let $\cX_{k}$ be a nested sequence of maximally separated $\rho^{k}$-nets in $\Sigma$ 
so that $0\in \cX_{0}$. For $C_{0}>0$,
\[
\sum_{k\geq 0}\sum_{x\in \cX_{k}}\beta_{\Sigma}^{d,p}(B(x,C_{0}\rho^{k}))\rho^{kd}\lec_{p,C_{0},n,\ve} \cH^{d}(\Sigma\cap B(0,1)).\]}
\end{theorem}

\section{Sketch of the proof of Theorems \ref{t:flat-case} and Theorem \ref{t:thmii}}\label{s:description-of-proof}

This section  will contain a loose description of the proofs of Theorems \ref{t:flat-case} and \ref{t:thmii}. 
The proof of Theorem \ref{t:flat-case} is carried out in Sections \ref{s:stoppingtime}--\ref{s:beta-omega}. 
Theorem \ref{t:thmii} is then proven in Section \ref{s:TheoremII}  by building a stopping time argument on top of Theorem \ref{t:flat-case}.
We  assume below that the reader is fluent with the language and results of the previous sections.
We expect that some readers may not gain much from this section, but hope that others will appreciate it.

The fundamental tool one would like to use to upper bound a sum of $\beta_{E}^{d,p}$-numbers as in the left-hand-side of Theorem
 \ref{t:flat-case}  is Dorronsoro's theorem (see Theorem \ref{t:dorronsoro}).  That requires having a graph.  
Thus, we would like to describe the surface $\Sigma$ 
 as the limit of a progression of graphs 
in which, locally, each successive surface is a graph over the previous surface, 
in the same spirit as the progression of curves that leads to the von Koch snowflake.  Indeed, Theorem \ref{t:DT} supplies us with $\Sigma$ as a limit of (constructed) surfaces $\Sigma_k$ which are locally $C^2$ graphs (with control on what `locally' means) but these are not  enough for the accounting we need to do. Instead we need to stop the construction of the $\Sigma_k$'s in different places, and let it keep running in others. 
We'll be more specific below:

Consider a family of Christ-like `cubes' for $\Sigma$  as assured by Theorem \ref{t:Christ}. 
To each cube $Q$, associate an affine  $d$-plane $P_Q$ which well-approximates $\Sigma$ inside a ball $B_Q$, where $B_Q\supset Q$ and is  not too big.
Now we generate layers of cubes as follows. Our zeroth layer are the cubes of sidelengh one. Now, run a  stopping time on each cube $Q$ from this family to construct a new layer by adding smaller and smaller cubes $R$ in $Q$ and stop if  $P_R$ changes its angle too much from the initial plane $P_Q$. Let $S_{Q}$ be the resulting stopping time region (i.e. those cubes for which you didn't stop). We will restart on the stopped cubes to get more   stopping time regions. Our first layer is the union of all $S_{Q}$ over all stopped cubes $Q$ from the zeroth layer. Now look at the minimal cubes of this layer and run the same stopping-time process to create a new collection of stopping times that form the $2$nd layer, and so on.

{%
These layers are {\em almost} good enough to work with, but not quite and we will now describe why. 
Take 
a layer and apply the construction of David and Toro (Theorem \ref{t:DT}) to it, in the sense that you use the centers of the cubes $Q$ and the planes $P_Q$ (really, you first have to shift them to go through the centers). The resulting surface is in fact a graph around each maximal cube in the layer. Thus  you can use Dorronsoro's theorem on each of these pieces to bound the $\beta_{E}^{d,p}$-numbers by a portion of $E$ around that maximal cube. However, the sizes of these maximal cubes can vary wildly, and so they may not have bounded overlap. This is a problem, which is why we said the layers are {\em almost} good enough. 
To solve this, when constructing our layers, we add an extra step that extends them so that adjacent minimal cubes have comparable sizes (Lemma \ref{l:QSQS'}) and then we repeat our construction on the minimal cubes of these extended layers and so on.  Note that they will still give rise to graphs over the planes corresponding to the maximal cubes in each layer (Lemma \ref{l:four-ten}).
%
}

Section \ref{s:stoppingtime} describes these regions and their properties.  
Section \ref{s:approxsurf} describes building a progression of surfaces $\Sigma^N$ corresponding to a progression of stopping times 
($N$ here corresponds to how many generations from the root our extended stopping time is) and records some properties.  
Each Surface there is its own application of Theorem \ref{t:DT}.

In Section \ref{s:layers} we will introduce a bi-Lipschitz map $F_N:\Sigma^N\to \Sigma^{N+1}$, whose image has (locally) a graph structure (see Figure \ref{f:fn}). 
Then, in Section \ref{s:tel-sum}, we show that the sum of the area of the cubes where we stopped 
{%
in our construction of the $N$th layer,} 
is bounded by the total area of the limit surface 
$\Sigma=\lim \Sigma^N$. Here the area of $\Sigma$ is really thought of as the limit of a telescoping sum, where,  the $N$th summand  bounds (with controllable errors) the area of the stopped cubes in $N$th layer.  (Note that this is really not a precise statement, since some summands in the telescoping sum could even be negative!)
Finally, in Section \ref{s:beta-omega}, we reduce upper bound  for $\beta_{E}^{d,p}$-numbers for a single stopping time to the upper bound  given by Dorronsoro's theorem Theorem \ref{t:dorronsoro}, which is bounded by the area of the root cube.  These are exactly the cubes whose total area was bounded in Section \ref{s:tel-sum}.  
This will complete the proof of Theorem \ref{t:flat-case}. 

Theorem \ref{t:thmii} follows in a similar fashion, except we also need to stop every time we cease to be Reifenberg flat. The onerous details of this are carried out in 
Section \ref{s:TheoremII}.

\section{The stopping-time}\label{s:stoppingtime}
Sections \ref{s:stoppingtime}, \ref{s:approxsurf}, \ref{s:layers}, \ref{s:tel-sum}, and \ref{s:beta-omega} will together give Theorem \ref{t:flat-case}.

\smallskip




Note that \eqn{sigmalr} means we can use the lemmas from Section \ref{s:betalemmas}. 
Fix $M>1$.
Let $\cD$ be the cubes for $\Sigma$ from \Theorem{Christ} using our maximal nets $\cX_{k}$ and set $P_{Q}'$ be such that 
\[
\beta_{\Sigma,\infty}^{d}(P_{Q}',MB_{Q})= \beta_{\Sigma,\infty}^{d}(MB_{Q})<\ve\]
and set $P_{Q} \ni x_{Q}$ be the plane parallel to $P_{Q}'$, and so
\begin{equation}\label{e:pqdef}
\beta_{\Sigma,\infty}^{d}(P_{Q},MB_{Q})\leq 2 \beta_{\Sigma,\infty}^{d}(MB_{Q})<2\ve 
\end{equation}
We can assume there is $Q_{0}\in \cD_{0}$ with center $0$. Without loss of generality, we may assume $P_{Q_{0}}=P_{0}$


\begin{definition}
(\cite[I.3.2]{of-and-on}) A {\it stopping-time region} $S\subseteq \cD$ is  a collection of cubes such that the following hold.
\begin{enumerate}
\item All cubes $Q\in S$ are contained in a maximal cube $Q(S)\in S$.
\item $S$ is {\it coherent}, meaning that if     $Q\in S$ and  $Q\subseteq R \subseteq Q(S)$ then $R\in S$.
\item For all $Q\in S$, each of its siblings of $Q$ are also in $S$.
\end{enumerate}
\label{d:st}
\end{definition}

Let $\alpha>\ve>0$. 
\begin{remark}\label{r:params-1}
We will adjust the value of $\ve$ (the Riefenberg flatness  parameter) as we go along, but it shall be much smaller than $\alpha$ (more precisely, $\epsilon\ll\alpha^4$), and will also depend on a parameter $\tau$. See Remark \ref{r:alpha} below as well as Section \ref{s:tel-sum}. The constant $\alpha$ may be fixed in the proof of Lemma \ref{l:asum2}.

\end{remark}

Let $Q\in \cD$ and let $S_{Q}$ be the stopping time region constructed by adding cubes $R$ to $S$ if both of the following are satisfied.
\begin{enumerate}
\item If $T\in \cD$ and $R\subseteq T\subseteq Q$ then $T\in S_{Q}$.
\item We have $\angle(P_{T},P_{Q})<\alpha$ for any sibling $T$ of $R$ (including $R$ itself).
\end{enumerate}
In this way, $Q(S_{Q})=Q$. \\

\begin{remark}
A property which $S_Q$ does not enjoy is that if two minimal cubes are `near' each other, then they have `similar' sizes.
Because of this, we will define `extensions' of these regions, $S'_Q$ later on (see \eqref{d:S-prime-def} and its preceding Lemma \ref{l:QSQS'}, as well as Lemma \ref{l:four-ten}).  The reason we care about this is that we will have a sequence of surfaces $\Sigma_N$, where $\Sigma_N$ comes from cubes stopped $N$ times.  Had we not had `near by cubes have similar sizes', we could not not control well enough the relationship between  $\Sigma_N$ and $\Sigma_{N+1}$. This happens in Section \ref{s:approxsurf}.
\end{remark}

For a collection of cubes $\cC$, we define a distance function
\[
d_{\cC}(x)=\inf \{\ell(Q)+\dist(x,Q): Q\in \cC\}\]
and for $Q\in \cD$, set
\[d_{\cC}(Q)=\inf_{x\in Q} d_{\cC}(x) =\inf \{\ell(R)+\dist(Q,R): R\in \cC\} \]

Let $m(S)$ be the set of minimal cubes of $S$, i.e. those $Q\in S$ for which there are no cubes $R\in S$ properly contained in $Q$, and define
\begin{equation}\label{e:z-def}
 z(S)=Q(S)\backslash \bigcup\{Q:Q\in m(S)\}.
\end{equation}

%

We will define a sequence of collections of cubes $\Layer(N)$, $\Up(N)$, $\Stop(N)$, 
and a sequence of collections of stopping-times $\cF^{N}$, 
$N=0,1,2,...$ as follows.

{First, set 
\[
\Stop(-1)=\cD_{0}.\]
Let $\tau\in(0,1)$ be small. 

\begin{remark}\label{r:tau1}
The constant $\tau$ will be upper bounded in Lemmas \ref{l:3.4},  \ref{l:QSQS'}, and \ref{l:graph}, see Remark \ref{r:tau2} below.
\end{remark}
%

Suppose we have defined $\Stop(N-1)$ for some integer $N\geq 0$. Let 
$$
\Layer(N)=\bigcup \{S_{Q} : Q\in \Stop(N-1)\}.
$$
Clearly 
\begin{equation}
\label{e:layer=stop}
\bigcup\{Q\in \Layer(N)\}=\bigcup\{Q\in \Stop(N-1)\,.
\end{equation}
Let $\Stop(N)$ be the set of maximal cubes $Q$ which have a sibling $Q'$ (possibly themselves) such that
 $\ell(Q')<\tau d_{\Layer(N)}(Q')$, i.e. (recalling that $ \Child_{(1)}Q^{(1)}$ is the collection of children of $Q$'s parent)
{\begin{multline}\label{e:sib}
\Stop(N)=
  \{Q\in \cD : Q \mbox{ max. such that there is } Q'\in \Child_{(1)}Q^{(1)} \mbox{ with }\\
   \ell(Q')<\tau d_{\Layer(N)}(Q')\}. 
\end{multline}}
See Corollary \ref{c:whitney-cor} below, which may elucidate why we use the terminology ``stop".  

{Set $\Up(-1)=\emptyset$ and
\begin{multline}
\Up(N)=
\Up(N-1)\cup \\
\{Q\in \cD : 
Q\supseteq R \mbox{ for some } R\in \Layer(N)\cup \Stop(N) \}
\end{multline}}
\begin{lemma}\label{l:up-alt-def}
\begin{equation}\label{e:up-alt-def}
\Up(N)=
\{Q\in \cD:\ Q\not\subset R \mbox{ for any } R\in \Stop(N)\} 
\end{equation}
\end{lemma}
\begin{proof}
Denote by $U$ the right hand side of \eqref{e:up-alt-def}.
First suppose $Q\in \Up(N)$.
If we have $Q\not\in U$, then there is an $R$ such that  $Q\subsetneq R\in \Stop(N)$.
By \eqref{e:sib}, we then have that $R$ cannot contain a $\Layer(N)$ cube, and so the same holds for $Q$. Thus, $Q$ cannot contain a cube from either $\Layer(N)$ or $\Stop(N)$, which implies $Q\not\in \Up(N)$. 
 
If $Q\in U$ then either $Q\supset T\in \Stop(N)$ (in which case $Q\in \Up(N)$) or 
$Q$ is disjoint from $\cup \Stop(N)$.  Hence, if $x\in c_{0} B_{Q}\cap \Sigma$, then $\tau d_{\Layer(N)}(x)\leq \ell(Q')$ for all cubes $Q'$ containing $x$, so in particular, $d_{\Layer(N)}(x)=0$, so we may find $R\in \Layer(N)$ so that $R\subseteq c_{0}B_{Q}\cap \Sigma\subseteq Q$, and thus $Q\in \Up(N)$ as well.
\end{proof}

\begin{lemma}
Let $\cC\subseteq \cD$ and $Q,Q'\in \cD$. 
Then
\begin{equation}\label{e:dtriangle}
d_{\cC}(Q)\leq 2\ell(Q)+\dist(Q,Q')+2\ell(Q')+d_{\cC}(Q').
\end{equation}
\end{lemma}

\begin{proof}
Let $x,w\in Q$, $y,z\in Q'$, and $Q''\in \cC$. Since $Q'\subseteq B_{Q'}$, $\diam Q'\leq 2\ell(Q')$, and so
\begin{align*}
d_{\cC}(x)
& \leq \dist(x,Q'')+\ell(Q'')
\leq |x-y|+\dist(y,Q'')+\ell(Q'')\\
& \leq |x-w|+|w-z|+|y-z|+\dist(y,Q'')+\ell(Q'')\\
& \leq 2\ell(Q)+|w-z|+2\ell(Q')+\dist(y,Q'')+\ell(Q'').\end{align*}
Now we infimize $\dist(y,Q'')+\ell(Q'')$ over all $Q''\in \cC$, then over all $y\in Q'$, and then over all $w\in Q$ and $z\in Q'$ and we get \eqn{dtriangle}. 
\end{proof}

Below, the constant $\rho$  comes from the statement of Theorem \ref{t:Christ}.
\begin{lemma}\label{l:3.4}
Let $\tau_{0}=\frac{1}{4(2+\rho^{-1})}$. For $0<\tau<\tau_{0}$,
\begin{equation}
\frac{\rho}{5}  \tau d_{\Layer(N)}(Q) \leq \ell(Q) \leq 2 \tau d_{\Layer(N)}(Q) \mbox{ for all }Q\in \Stop(N).
\label{e:Q'simd}
\end{equation}
\end{lemma}

\begin{proof}
Indeed, if $Q'$ is the sibling of $Q$ (or $Q$ itself) satisfying \eqn{sib}, then
\[
\dist(Q,Q')\leq \diam Q^{(1)} \leq \diam B_{Q^{(1)}}= \frac{2}{\rho}\ell(Q)\] 
and so
\begin{align*}
\ell(Q)
&  =\ell(Q')< \tau d_{\Layer(N)}(Q')\\
& \stackrel{\eqn{dtriangle}}{\leq} \tau\big(2\ell(Q')+\dist(Q,Q')+2\ell(Q)+d_{\Layer(N)}(Q)\big)\\
& \leq \tau \big(2\ell(Q)+2\rho^{-1}\ell(Q)+2\ell(Q)+d_{\Layer(N)}(Q)\big)\\
& =2\tau( 2+\rho^{-1})\ell(Q)+\tau d_{\Layer(N)}(Q)< \frac{\ell(Q)}{2} + \tau d_{\Layer(N)}(Q)
\end{align*}
since $\tau<\tau_{0}$, and so we may regroup the terms and get $\ell(Q)<2 \tau d_{\Layer(N)}(Q)$. By the maximality of $Q$,
\[
\tau d_{\Layer(N)}(Q^{(1)})\leq \ell(Q^{(1)})
\]
Recalling $\diam Q^{(1)}=2\ell(Q^{(1)})= 2\rho^{-1}\ell(Q)$,
\begin{align*}
\tau d_{\Layer(N)}(Q)
& \stackrel{\eqn{dtriangle}}{\leq}\tau( 2\ell(Q'')+2\ell(Q)+\dist(Q,Q''))+\ell(Q^{(1)})\\
& \leq \tau(4\ell(Q)+\diam Q^{(1)})+\ell(Q^{(1)}) \\
&  \tau(4+3\rho^{-1})\ell(Q)<5\rho^{-1}\ell(Q)
\end{align*}
since $\rho<1/1000$.

%
\end{proof}

\begin{corollary}\label{c:whitney-cor}
Let $\tau$ be as in Lemma \ref{l:3.4}.
Let 
\[
Z=\{x:d_{\Layer(N)}(x)=0\}.\] 
Then 
\[\bigcup\{Q\in \Stop(N-1)\}\setminus Z = \bigcup\{Q\in \Stop(N)\}.\] 
\end{corollary}
\begin{proof}
By Lemma \ref{l:3.4}, it remains to see that 
$\cup\Stop(N-1)\setminus Z = \cup \Stop(N)$.
Let $x\in \bigcup\{Q\in \Stop(N-1)\}$ be such that $d_{\Layer(N)}(x)>0$.
By \eqn{layer=stop} and because $x\not\in Z$, we have a minimal (with respect to containment) $Q_{N}$ such that 
$x\in Q_{N}\in \Layer(N)$. 
Let $Q_{N,k}$ be such that $x\in Q_{N,k}\in \Child_k(Q_N)$.  If $Q_{N,k}$ is also in $\Layer(N)$, then
\[
0<d_{\Layer(N)}(x) \leq \ell(Q_{N,k})= \rho^k\ell(Q_{N})\]
which is a contradiction for $k$ large. Thus,  there is a $k$ such that $x\in Q_{N,k}\in \Stop(N)$. 

Conversely, let $x\in Q\in \Stop(N)$. Then by the previous lemma
\[
\tau d_{\Layer(N)}(x)\geq \tau d_{\Layer(N)}(Q)\gec  \ell(Q)\]
and so $x\not\in Z$. Moreover, $Q\subseteq R$ for some $R\in \Stop(N-1)$, and this finishes the proof.
\end{proof}

\begin{corollary}\label{c:up-positive-then-stop}
Suppose $d_{\Up(N)}(x)>0$.  Then there is a $Q\ni x$ such that $Q\in \Stop(N)$. 
\end{corollary}
\begin{proof}
We have for all $n\leq N$ that 
\[
d_{\Layer(N)}(x)\geq d_{\Up(n)}(x)\geq  d_{\Up(N)}(x)>0.\] 
The result now follows  by the first part of  Corollary \ref{c:whitney-cor}.
\end{proof}

\begin{lemma}
\label{l:QSQS'}
Let $C_{1}>1$, $\tau_{1}(C_{1})=\min\{\tau_{0},(16+8C_{1})^{-1}\}$, and assume $0<\tau<\tau_{1}(C_{1})$.
If $\ve<\alpha$, $Q,Q'\in\Stop(N)$ and $C_{1}B_{Q}\cap C_{1}B_{Q'}\neq\emptyset$,  then 
$\ell(Q)\sim \ell(Q')$  (with constants independant of $C_{1}$), and $\angle(P_{Q},P_{Q'})\lec_{C_{1}} \ve$  .
\end{lemma}

\begin{proof}
Assume $\ell(Q')\leq \ell(Q)$, we will first show $\ell(Q)\lec \ell(Q')$. Since $C_{1}B_{Q}\cap C_{1}B_{Q'}\neq\emptyset$, we know 
\[
\dist(Q,Q')\leq C_{1}\ell(Q)+C_{1}\ell(Q')\leq 2C_{1}\ell(Q).\]
Thus, since $Q,Q'\in \Stop(N)$, 
\begin{align*}
\ell(Q)\stackrel{\eqn{Q'simd}}{\leq} 2\tau d_{\Layer(N)}(Q)
& \stackrel{\eqn{dtriangle}}{\leq} 2\tau d_{\Layer(N)}(Q') \\
& \qquad +2\tau(2\ell(Q)+\dist(Q,Q')+2\ell(Q'))\\
& \stackrel{\eqn{Q'simd}}{\leq} \frac{2}{\rho}\ell(Q')+2\tau (2\ell(Q)+2C_{1}\ell(Q)+2\ell(Q))\\
& =  \frac{2}{\rho}\ell(Q')+(8+4C_{1})\tau \ell(Q)< \frac{2}{\rho}\ell(Q')+\frac{\ell(Q)}{2}
\end{align*}
where in the last inequality we used  $\tau<\frac{1}{16+8C_{1}}$. This then gives
 $\ell(Q')\leq \ell(Q)\leq \frac{4}{\rho} \ell(Q')$. The lemma now follows from \Lemma{PQPR} and the fact that $\Sigma$ is lower regular.
\end{proof}

\preRef{For $Q\in {\Stop(N)}$, let\\ }

\postRef{For $Q\in {\Stop(N-1)}$, let}

\begin{equation}\label{d:S-prime-def}
S_{Q}'=\{R\in \Up(N):R\subseteq Q\}\supseteq S_{Q}. 
\end{equation}
Observe that this is again a stopping-time region by construction (this is why we defined $\Stop(N)$ using siblings). 

\begin{lemma}\label{l:four-ten}
For $0<\tau<\tau_{1}(3)$ and  $Q\in \Stop(N-1)$
\begin{equation}
\angle(P_{R},P_{Q})\lec
\alpha\mbox{ for all }R\in S_{Q}'.
\label{e:PQPQS} 
\end{equation}
\end{lemma}

\begin{proof}
Let $R\in S_{Q}'$ and $T\in \Layer(N)$ be such that 
\begin{equation}
\ell(T)+\dist(T,R)\leq 2d_{\Layer(N)}(R).
\label{e:Q''d}
\end{equation}
Then $T\in S_{Q'}$ for some $Q'\in \Stop(N-1)$. Hence,
\begin{align}
\dist(Q,Q')
& \leq \dist(R,T)
\stackrel{\eqn{Q''d}}\leq 2d_{\Layer(N)}(R) \notag \\
& \leq 2(\ell(Q)+\dist(R,Q))=2\ell(Q). \label{e:QSQS'<2d}\end{align}
In particular, $3B_{Q'}\cap 3B_{Q}\neq \emptyset$, and since $Q,Q'\in \Stop(N-1)$, we have by \Lemma{QSQS'} with $C_1=2$ if $\tau<\tau_{1}(3)$ that
\begin{equation}
\label{e:QSQS'2}
\ell(Q) \sim \ell(Q') \mbox{ and }\angle(P_{Q},P_{Q'})\lec_{C_{1}}\ve.
\end{equation}  

We split into two cases.
\begin{enumerate}
\item Suppose first that $\ell(Q')\leq \frac{2}{\rho\tau}\ell(R)$. Then 
\[
\ell(R)\leq \ell(Q)
\stackrel{\eqn{QSQS'2}}{\sim} \ell(Q')\leq \frac{2}{\rho\tau}\ell(R),\] 
and so $\ell(R)\sim_{\tau} \ell(Q')$. Moreover, 
\[
\dist(Q,Q') \stackrel{\eqn{QSQS'<2d}}{\leq} 2\ell(Q) \stackrel{\eqn{QSQS'2}}{\lec} \ell(Q'),\]
hence \Lemma{PQPR} implies $\angle(P_{R},P_{Q'})\lec_{\tau} \ve$, and so
\[
\angle(P_{R},P_{Q})\leq \angle(P_{R},P_{Q'})+\angle(P_{Q'},P_{Q})
\stackrel{\eqn{QSQS'2}}{\lec}_{\tau} \ve.\]

\item Now suppose $\ell(Q')>\frac{2}{\rho\tau}\ell(R)$. Note that
\[
\ell(T)
\stackrel{\eqn{Q''d}} \leq 2d_{\Layer(N)}(R)
\stackrel{\eqn{Q'simd}}{\leq} \frac{2}{\rho\tau}\ell(R)< \ell(Q').\] 
Let $Q''\in S_{Q'}$ be the largest parent of $T$ for which 
\[\ell(Q'')\leq \frac{2}{\rho\tau}\ell(R).\]
By the above inequality, this is well defined and $\ell(Q'')\sim \tau^{-1}\ell(R)$. Moreover, 
\[\dist(R,Q'')\leq \dist(R,T)\stackrel{\eqn{Q''d}}{\leq} 2d_{\Layer(N)}(R)\stackrel{\eqn{Q'simd}}{\lec} \tau^{-1}\ell(R).\]
Thus, \Lemma{PQPR} implies
\begin{equation}\label{e:pqpq'}
\angle(P_{R},P_{Q''})\lec_{\tau} \ve.
\end{equation}
Also note that since $Q''\in S_{Q'}$, 
\begin{equation}
\angle(P_{Q''},P_{Q'})\lec \alpha.
\label{e:Q'QS'<a}
\end{equation}
Thus, combining \eqn{QSQS'2}, \eqn{pqpq'}, and \eqn{Q'QS'<a}, we get
\begin{multline} 
\angle(P_{R},P_{Q})
\leq \angle(P_{R},P_{Q''})+\angle(P_{Q''},P_{Q'})+\angle(P_{Q'},P_{Q})
\\ \lec_{\tau} \ve +\alpha+\ve \lec\alpha .
\end{multline}
\end{enumerate}
\end{proof}

\begin{lemma}
\label{l:>a}
Let $0<\tau<\min\{\tau_{0},c_{0}/4\}$. For $\ve>0$ small enough depending on $\tau$ and $\alpha$, the following holds. 
Suppose $Q\in m(S)$ where $S=S_{Q(S)}$, with   
$Q(S)\in \Stop(N-1)$ (so $S\subset \Layer(N)$) .
Then 
then there is $R$ such that
\begin{enumerate}
\item $R\in \Stop(N)$,
\item $R\subseteq Q$,
\item $\ell(R)\sim \tau \ell(Q)$, and
\item $\angle(R,Q(S))\gec_{\tau} \alpha$.
\end{enumerate}
\end{lemma}

\begin{proof}
If $Q\in \cD_{k}$, let $R\in \Stop(N)$ be the cube with the same center as $Q$, so $x_{R}=x_{Q}$. Then $R\in \cD_{k+k_{1}}$ for some $k_{1}\geq 0$. Observe that by Lemma \ref{l:3.4}
\begin{equation}\label{e:R<tQ}
5\rho^{k+k_{1}}=\ell(R)<2\tau d_{\Layer(N)}(R)\leq 2\tau \ell(Q)=10\tau \rho^{k}
\end{equation}
and so
\begin{equation}\label{e:r<2t}
\rho^{k_{1}}<2\tau.\end{equation}
We claim that 
\begin{equation}\label{e:tclaim}
d_{\Layer(N)}(R)\geq \frac{c_{0}}{2}\ell(Q).
\end{equation}
First note that $Q\in \Layer(N)$ and
\[\dist(R,Q)+\ell(Q)=\ell(Q).\]
Thus when considering $d_{\Layer(N)}(R)$ we may look at the quantity
$\dist(R,T)+\ell(T)$ where $T\in \Layer(N)$, and we assume without loss of generality 
$\ell(T)< \ell(Q)$ and 
$\dist(T,R)< \ell(Q)$.  Since $Q\in m(S)$, we must have that $T\subseteq Q^{c}$, hence $T\subseteq c_{0}B_{Q}^{c}$. By \eqn{R<tQ} and assuming $\tau<c_{0}/4$, since $x_{R}=x_{Q}$, 
\[
R\subseteq 2\tau B_{Q}\subseteq \frac{c_{0}}{2} B_{Q},\]
and thus $\dist(T,R)\geq \frac{c_{0}}{2}\ell(Q)$. Infimizing $\dist(R,T)+\ell(T)$ over all such $T$ gives \eqref{e:tclaim}.



Thus,
\[  \frac{1}{\tau \rho} \ell(R) 
\stackrel{\eqn{Q'simd}}{\geq } d_{\Layer(N)}(R)
 \stackrel{\eqn{tclaim}}{\geq} \frac{c_{0}}{2}\ell(Q)\]
and so $\ell(R)\gec \tau \ell(Q)$. This and \eqn{r<2t} imply $\ell(R)\sim \tau \ell(Q)$. 

Let $Q'$ be a child of $Q$ such that $\angle(Q',Q(S))\geq \alpha$ (which exists by minimality of $Q$ in $S$). Then 
$\angle(P_{Q'},P_{Q})\lec \ve$ by \Lemma{PQPR}. By the same  lemma, $\angle(P_{Q},P_{R})\lec \ve$. Thus, for some constant $C>0$ and $\ve>0$ small enough depending on $\alpha$ and $\tau$.
\[
\angle(P_{R},P_{Q(S)})\geq \angle(P_{Q'},P_{Q(S)})-\angle(P_{Q},P_{R})-\angle(P_{Q},P_{Q'})
\geq \alpha-C\ve \gec \alpha.\]
\end{proof}

Recall the definition of $z_\cC$ (for a collection of cubes $\cC$) given by equation \eqref{e:z-def}. Define
\[ \cF_{N}=\{S_{Q}': Q\in \Stop(N-1)\}.\]



%


\begin{lemma}\label{l:5.1}
If $S'_1\in \cF_{N}$ and $S'_2\in \cF_M$ are distinct, then $z(S_1') \cap z(S_2')=\emptyset$.
 \end{lemma}
\begin{proof}
{ First note that for $i=1,2$, $S'_i$ are stopping times, and by construction, if they are distinct, then $S_{1}'\cap S_{2}'\subseteq \{Q(S_{1}'),Q(S_{2}')\}$. However, if $x\in z(S_1') \cap z(S_2') $, then every cube $Q$ containing $x$ with $\ell(Q)<\min\{\ell(Q(S_{1}')),\ell(Q(S_{2}'))\}$ is in $S_{1}'\cap S_{2}'$, which is a contradiction.}
%
%
 \end{proof}
%
%

\begin{lemma}\label{l:412}
Let $0<\tau<\tau_{0}$. For $N\geq 0$, 
\begin{equation}\label{e:412a}
d_{\Layer(N)}(x)\sim d_{\Up(N)}(x) \mbox{ for all }x\in \bR^{n}
\end{equation}
and in particular
\begin{equation}\label{e:412b} 
\tau d_{\Up(N)}(Q) \sim  \ell(Q) \mbox{ for all }Q\in \Stop(N).
\end{equation}
 \end{lemma}
\begin{proof}
Since $\Layer(N)\subseteq \Up(N)$, we have $d_{\Layer(N)}(x)\geq d_{\Up(N)}(x)$, so we just have to verify the opposite inequality. 


Let $T\in \Up(N)$ be such that 
$\ell(T)+\dist(T,Q) < 2 d_{\Up(N)}(x)$. We consider two cases.
\begin{enumerate}
\item If $T\supset R\in \Layer(N)$, then 
\[
\ell(R) + \dist(x,R) \leq \ell(T) + \dist(x,R) 
\leq 3\ell(T)+\dist(x,T)
\leq 6  d_{\Up(N)}(x),\] 
and so
$d_{\Layer(N)}(x) \lesssim d_{\Up(N)}(x)$, and we are done with this case.

\item For the second case, suppose  $T\supset R\in \Stop(N)$. Without loss of generality, we can assume $T\not\in \Layer(N)$. Then there is $Q\in \Layer(N)$ containing $R$, and by our assumption, $T\supset Q$. Hence, just as in the previous case,
\begin{align*}
\ell(R)+\dist(x,R)
& \leq 3(\ell(Q)+\dist(x,Q))
\leq 9 (\ell(T)+\dist(x,T))\\
& \leq 18 d_{\Layer(N)}(x).
\end{align*}

%
\end{enumerate}
The last part of the lemma now follows from \Lemma{3.4}.
%

 \end{proof}

%
%

\section{The sequence of approximating surfaces}\label{s:approxsurf}

For $k\geq 0$ an integer, let $s(k)$ be such that $5\rho^{s(k)}\leq r_{k}< 5\rho^{s(k)-1}$. Set
\[
\Up(N)_{k}=\cD_{s(k)}\cap \Up(N)\]
and let $\cX^{N}_{k}=\{x_{j,k}\}_{j\in J_{k}^{N}}$ be a maximal $r_{k}$-separated set of points for the set
\[\cC_{k}^{N}=\{x_{Q}:Q\in \Up(N)_{k}\}.\]
For $j\in J_{k}^{N}$, let $Q_{j,k}\in \Up(N)_{k}$ be such that $x_{Q_{j,k}}=x_{j,k}$ and let $P_{j,k}=P_{Q_{j,k}}$. Note that in this way $x_{j,k}\in P_{j,k}$ and 
\begin{equation}\label{e:qjkrk}
\ell(Q_{j,k})\leq r_{k} < \rho^{-1} \ell(Q_{j,k}).
\end{equation}

\begin{lemma}\label{l:DTready}
For each $N$, $\{x_{j,k}\}_{j\in J_{k}^{N}}$ satisfies the conditions of \Theorem{DT}.
\end{lemma}

\begin{proof}
Let $j\in J_{k}^{N}$. If $s_{k}=s_{k-1}$, then there is $i\in J_{k-1}$ so that $x_{j,k}=x_{Q_{j,k}}\in B_{i,k-1}$ since $\{x_{i,k-1}\}_{i\in J_{k-1}}$ is a maximal net for $\cC_{k-1}^{N}=\cC_{k}^{N}$. Otherwise, if $s_{k}>s_{k-1}$, then there is $i\in k-1$ so that $x_{Q_{j,k}^{(1)}}\in B_{i,k-1}$. Since 
\[
\ell\ps{Q_{j,k}^{(1)}}=5\rho^{s(k)-1}\leq 5\rho^{s(k-1)}\leq r_{k-1},\]
we have 
\begin{align*}
x_{j,k} & \in Q_{j,k}\subseteq Q_{j,k}^{(1)}
\subseteq B\ps{x_{Q_{j,k}^{(1)}},\ell\ps{Q_{j,k}^{(1)}}}
\subseteq B\ps{x_{i,k-1},r_{k-1}+\ell\ps{Q_{j,k}^{(1)}}}\\
& \subseteq B(x_{i,k-1},2r_{k-1})
 = 2B_{i,k-1}.
\end{align*}

Furthermore, $\ve_{k}(x_{j,k})\lec \ve$ by \Lemma{PQPR}, so for $\ve>0$ small enough, the lemma follows.

\end{proof}

Let $P_{0}=P_{Q_{0}}$ and let $\sigma_{k}^{N}$, $\Sigma_{k}^{N}$, $\Sigma^{N}$, and $\sigma^{N}$ be the functions and surfaces obtained from \Theorem{DT} with the nets $\cX^{N}_{k}$. In this way, $\Sigma_{0}^{N}=P_{0}$ for all $N$.

\begin{lemma}\label{l:up-zero-intersection}
If $d_{\Up(N)}(x)=0$ then $x\in \Sigma^N\cap \Sigma$
 \end{lemma}
\begin{proof}
Suppose $d_{\Up(N)}(x)=0$. This guarantees a sequence of cubes $Q_i\in \Up(N)$ such that $\ell(Q_i) + \dist(x,Q_i)\to 0$.
From the definition of $\cX^{N}_{k}$ as a net,
and  Theorem \ref{t:DT}, equation \eqref{eqahj0}, we have $x\in \Sigma^N$.
It remains to see that $x\in\Sigma$.
To this end, note that $\dist(x,\Sigma)\leq \dist(x,Q_{i})\downarrow 0$, and so $x\in \Sigma$. 

 \end{proof}


For $x\in \Sigma^{N}$, we define $k_{N}(x)$ to be the maximal integer such that $x\in V_{k-1}^{11}$.

\begin{lemma}\label{l:lemma-5-4}
For $x\in \Sigma^{N}$, let $k=k_{N}(x)$ . Then we have 
\begin{equation}\label{e:dx<dT}
\dist(x,\Sigma)\lec   \ve r_{k} \sim \ve d_{\Up(N)}(x)
\end{equation}
and
\begin{equation}
B(x,r_{k})\cap \Sigma^{N}=B(x,r_{k})\cap \Sigma^{N}_{k}.
\label{e:bsnbsnk}
\end{equation}
\label{l:dxsigma}
\end{lemma}

\begin{proof}
We first prove the left-hand-side inequality \eqn{dx<dT}. There is nothing to show if $x\in \Sigma$, so assume $x\not\in \Sigma$. If $x\not\in B(0,1)$, then $x\in P_{Q_{0}}$ by \eqn{v10c} and so $x\in \Sigma$ by \eqn{S-B10}, so we may assume $x\in B(0,1)$. Hence, there is a maximal $k\geq 1$ for which $x\in V_{k-1}^{11}$. Let $x'\in \Sigma_{k-1}^{N}$ be such that $x=\lim_{K\rightarrow\infty}\sigma_{k+K}^{N}\circ\cdots\circ \sigma_{k-1}^{N}(x')$. Then by \eqn{ytosigma},
\begin{equation}
\label{e:x'-x<erk}
|x'-x|\lec \ve r_{k-1}\lec \ve r_{k}.
\end{equation}
Hence, for $\ve$ small enough, $x'\in V_{k-1}^{12}$. 
Thus, there is $j\in J_{k-1}$ so that $x'\in 12B_{j,k-1}$. By \eqn{49r}, $\dist(x',P_{j,k-1})\lec\ve r_{k}$ and $\pi_{j,k-1}(x')\in 13B_{j,k-1}$. By our choice of $P_{j,k-1}$, $\dist(\pi_{j,k-1}(x'),\Sigma)\lec \ve r_{k-1}$. Combining these inequalities, we get that $\dist(x,\Sigma)\lec \ve r_{k}$. 
It now remains to show the right-hand-side inequality \eqn{dx<dT}, i.e. that $ r_{k}\sim d_{\Up(N)}(x)$. First, since $x'\in 12B_{j,k-1}$,
\begin{align*}
d_{\Up(N)}(x) & \leq \ell(Q_{j,k-1})+\dist(x,Q_{j,k-1})\\
& \stackrel{\eqn{qjkrk}}{\lec} r_{k-1}+|x-x'|+\dist(x',Q_{j,k-1})
\\
& \stackrel{\eqn{x'-x<erk}}{\lec} r_{k-1}+\ve r_{k-1}+|x'-x_{j,k-1}|
\stackrel{x'\in 12B_{j,k-1}}{\lec} r_{k-1}\lec r_{k}.
\end{align*}
Next, let $Q\in \Up(N)$ be such that $d_{\Up(N)}(x)\sim \ell(Q)+\dist(x,Q)$. Suppose $d_{\Up(N)}(x)\sim r_{\ell}$ for some $\ell\geq k$. Since $|x-x_{Q}|\lec d_{\Up(N)(x)}$ we know $|x-x_{Q}|\leq  Cr_{\ell}$ for some universal constant $C$, and if
 $\ell - k\gtrsim\log(C)$, $|x-x_{Q}|<r_{k}$. Then $x_{Q}\in B_{i\ell}$ for some $i\in J_{\ell}^N$ (since $\cX^{N}_{\ell}$ is a maximal net), so in particular, $x_{Q}\in V_{\ell}^{1}$. Thus, $x_{Q}\in V_{k}^{2}$, and hence $x\in V_{k}^{11}$, contradicting our choice of $k$. Thus, $|\ell-k|$ is bounded by a universal constant, implying $d_{\Up(N)}(x)\gec r_{k}$, we have \eqn{dx<dT}.

To get \eqn{bsnbsnk}, notice that since $x\not\in V_{k}^{11}$ and by the maximality of $k$, $B(x,r_{k})\subseteq (V_{\ell}^{10})^{c}$ for all $\ell\geq k$, and so $\sigma_{\ell}^{N}$ is the identity on $B(x,r_{k})$ for all $\ell\geq k$. This and \eqn{ygraph} imply \eqn{bsnbsnk}.
\end{proof}

\begin{lemma}\label{l:qxn}
For $\ve>0$ small enough (depending on $\tau$) and $x\in \Sigma^{N}\backslash \Sigma$, there is $Q_{x}^{N}\in \Stop(N)$ for which $x\in 2B_{Q_{x}^{N}}$ and $\ell(Q_{x}^{N})\sim d_{\Up(N)}(x)$. If $x\in c_{0}B_{Q}$ for some $Q\in \Stop(N)$ (where $c_{0}$ is as in \Theorem{Christ}), we may set $Q_{x}^{N}=Q$.
\end{lemma}

\begin{proof}
As in the statement,  if $x\in c_{0}B_{Q}$ for some $Q\in \Stop(N)$, we just set $Q_{x}^{N}=Q$, and it is easy to check
$\ell(Q_{x}^{N})\sim d_{\Up(N)}(x)$. 
Otherwise, let $x'\in \Sigma$ be such that 
\begin{equation}\label{e:x-x'<dL}
|x-x'|=\dist(x,\Sigma) \stackrel{\eqn{dx<dT}}{\lec} \ve d_{\Up(N)}(x).
\end{equation}
Note that if $d_{\Up(N)}(x)=0$, then $x\in \Sigma$ by Lemma \ref{l:up-zero-intersection}, which contradicts our choice of $x$. Thus, $d_{\Up(N)}(x)>0$.  Since $d_{\Up(N)}$ is $1$-Lipschitz, \eqn{x-x'<dL} implies 
\begin{equation}\label{e:dcompar}
\frac{1}{2}d_{\Up(N)}(x)< d_{\Up(N)}(x')
< 2d_{\Up(N)}(x)
\end{equation}
if $\ve>0$ is small enough. Thus, by Corollary \ref{c:up-positive-then-stop} $x'\in Q$ for some $Q\in \Stop(N)$, and since
\[
|x-x'| 
 \stackrel{ \eqn{x-x'<dL}}{\lec} \ve d_{\Up(N)}(x)
\stackrel{\eqn{dcompar}}{\lec}\ve d_{\Up(N)}(x')
 \stackrel{\eqn{up-alt-def}}{\sim} \ve d_{\Up(N)}(Q)
  \stackrel{\eqn{412b}}{\sim} \frac{\ve}{\tau}  \ell(Q),\]
 we have that $x\in 2B_{Q}$ if $\ve>0$ small (depending on $\tau$). Set $Q_{x}^{N}=Q$. Then $\ell(Q_{x}^{N})\sim \tau d_{\Up(N)}(Q_{x}^{N})$ by \eqref{e:412b}.
\end{proof}

{
\begin{lemma}
Let $M'=M+11$. Then for all $k\geq 0$, 
\begin{equation}\label{e:sknp0}
\Sigma_{k}^{N}\backslash B(0,1+M'r_{k})=P_{0}\backslash B(0,1+M'r_{k}).
\end{equation}
In particular, for all $N\geq 0$,
\begin{equation}\label{e:sn01}
\Sigma^{N}\backslash B(0,1)=P_{0}\backslash B(0,1)
\end{equation}
\end{lemma}

\begin{proof}
Recall by \eqn{v10c} that $\sigma_{k}^{N}(x)=x$ for $x\not\in V_{k}^{10}$ and $k\geq 0$. Let $x_{j,k}\in \Sigma\backslash B(0,1+Mr_{k})$. Then $x_{j,k}=x_{Q_{j,k}}$ and 
\[
|x_{Q_{j,k}}|\geq 1+Mr_{k}\stackrel{\eqn{qjkrk}}{\geq} 1 + M\ell(Q_{j,k})\]
which implies $MB_{Q_{j,k}}\cap B(0,1)=\emptyset$. By \eqn{S-B10} we have
\[
MB_{Q_{j,k}}\cap \Sigma= MB_{Q_{j,k}}\cap P_{0}\]
and thus $P_{j,k}=P_{MB_{Q_{j,k}}}=P_{0}$. 
If $x\in \Sigma_{k-1}^{N}\backslash B(0,1+(M+10)r_{k})$, then for each $j\in J_{k}$ with $x\in 10 B_{j,k}$, we must have $x_{j,k}\in \Sigma\backslash B(0,1+Mr_{k})$, and so 
\begin{equation}\label{e:pjk=p0}
\pi_{j,k}=\pi_{P_{0}}.
\end{equation}
Moreover, by \eqn{49r}, if $x\in 10 B_{j,k}$ for some $j\in L_{k}$, then 
\begin{equation}
\label{e:x-pj,kx<verl} |x-\pi_{j,k}(x)|\lec \ve r_{k}
,
\end{equation} 
and so for $\ve>0$ small enough, 
\begin{equation}\label{e:in11b}
\pi_{j,k}(x)\in 11 B_{j,k}.
\end{equation}
 Since $\Sigma$ is $\ve$-Reifenberg flat, there is a plane $P$ passing through $x_{j,k}$ so that $d_{x_{j,k},11r_{k}}(\Sigma,P)<\ve$, thus
\begin{align}\label{e:pjkp}
d_{x_{j,k},11 r_{k}}(P_{j,k},P)
& \stackrel{\eqn{dxardxr}}{=}d_{x_{j,k}, r_{k}/2}(P_{j,k},P)
\stackrel{\eqn{dxrtrngl}}{  \lec} 
d_{x_{j,k}, r_{k}}(P_{j,k},\Sigma)
+ d_{x_{j,k}, r_{k}}(\Sigma,P) \notag \\
& 
\lec \ve + 11  d_{x_{j,k}, 11 r_{k}}(\Sigma,P)
\lec \ve 
\end{align}

Thus,
\begin{align*}
d_{x_{j,k},11 r_{k}}(P_{j,k},\Sigma)
& \leq 
d_{x_{j,k},11 r_{k}}(P_{j,k},P)
+d_{x_{j,k},11 r_{k}}(P,\Sigma) \stackrel{\eqn{pjkp}}{\lec}  \ve 
\end{align*}

This and \eqn{in11b} imply $\dist(\pi_{j,k}(x),\Sigma)\lec \ve r_{k}$, and along with \eqn{x-pj,kx<verl} gives $\dist(x,\Sigma)\lec \ve r_{k}$. For $\ve>0$ small enough, this implies $\dist(x,\Sigma)<r_{k}$ and hence $x\in V_{k}^{2}$. Recall the notation from \Theorem{DT} that $B_{j,k}=B(x_{j,k},r_{j,k}/10)$ for $j\in L_{k}$, and that since $\{x_{j,k}\}_{j\in L_{k}}\subseteq \bR^{n}\backslash V_{k}^{9}$,
\[
\supp \psi_{k}\subseteq \bigcup_{j\in L_{k}} 10 B_{j,k}
=\bigcup_{j\in L_{k}} B(x_{j,k},r_{k})
\subseteq \bR^{n}\backslash V_{k}^{7},\]
and thus $\psi_{k}=0$ on $V_{k}^{2}$; in particular, $\psi_{k}(x)=0$. Hence
\[
\sigma_{k}^{N}(x)
\stackrel{\eqn{sigmak}}{=}
\psi_{k}(x)x+\sum_{j\in J_{k}}\theta_{j,k}(x) \pi_{j,k}(x)
\stackrel{\eqn{pjk=p0}}{=}0+ \pi_{P_{0}}(x)\sum_{j\in J_{k}} \theta_{j,k}(x) 
=\pi_{P_{0}}(x).\]
Thus, $\sigma_{k}^{N}=\pi_{P_{0}}$ on $\Sigma_{k-1}^{N}\backslash B(0,1+(M+10)r_{k}))$. Recall that $\Sigma_{k}^{N}=\sigma_{k}^{N}(\Sigma_{k-1}^{N})$. By \eqn{sky-y}, $|\sigma_{k}^{N}(y)-y|\lec \ve r_{k}$, and so for $\ve>0$ small enough,
\[
\Sigma_{k}^{N}\backslash B(0,1+(M+11)r_{k})
\subseteq \sigma_{k}^{N}(\Sigma_{k-1}^{N} \backslash B(0,1+(M+10)r_{k})
\subseteq P_{0}.\]
Let $M'=M+11$ and, contrary to \eqn{sknp0}, assume that 
\[
P_{0}\backslash (\Sigma_{k}^{N}\cup B(0,1+M'r_{k})\neq \emptyset.\]
Note that $\Sigma_{k}^{N}\backslash \cnj{B(0,1+M'r_{k})}\neq\emptyset$ since $\Sigma_{k}^{N}$ is $C\ve$-Reifenberg flat by \Lemma{skflat}. Thus, it is possible to find $x\in P_{0}$ and $r>0$ so that  
\[
B(x,8r)\cap \cnj{B(0,1+M'r_{k})}=\emptyset, \;\; B(x,r)\cap \Sigma_{k}^{N}=\emptyset,\] 
and so that there is 
\[
y\in \d B(x,r)\cap \Sigma_{k}^{N}\backslash   \cnj{B(0,1+M'r_{k})}. \]
As $\Sigma_{k}^{N}$ is $C\ve$-Reifenberg flat, there is a $d$-plane $P$ passing through $y$ so that $d_{y,4r}(\Sigma_{k}^{N},P)\lec \ve$. By \eqn{sigmalr}, for $\ve>0$ small enough, 
\[
\cH^{d}_{\infty}(\Sigma_{k}^{N}\cap B(y,r))>\frac{\omega_{d}}{2}r^{d}\]
and so we may find points $X=\{x_{0},...,x_{d}\}\subseteq \Sigma_{k}^{N}\cap B(y,r)$ for which $\eta(X)\gec_{d} 1$, where $\eta$ is as in \Lemma{ATlemma}; by this same lemma, we can conclude that $d_{y,4r}(P,P_{0})=d_{y,r}(P,P_{0})\lec \ve$. Hence,
\[
d_{y,2r}(\Sigma_{k},P_{0})
\stackrel{\eqn{dxrtrngl}}{\lec} d_{y,4r}(\Sigma_{k},P)+d_{y,4r}(P_{0},P)\lec \ve .\]
Hence, there is $z\in \Sigma_{k}^{N}$ so that $|z-x|\lec \ve r$, and $z\in B(x,r)$ for $\ve>0$ small enough. Since $\Sigma_{k}^{N}\backslash \cnj{B(0,1+M'r_{k})}\subseteq P_{0}$ and $B(x,r)\subseteq \cnj{B(0,1+M'r_{k})}^{c}$, this means $z\in P_{0}\cap B(x,r)\subseteq (\Sigma_{k}^{N})^{c}$, a contradiction.

Now \eqn{sn01} follows since $\Sigma^{N}$ is the limit of the $\Sigma^{N}_{k}$.

\end{proof}

{
\begin{lemma}
\label{l:wherestop}
If $Q\in \Stop(N)$ for some $N\geq 0$, then $Q\subseteq B(0,1)$. 
\end{lemma}

\begin{proof}
Recall $\cD_0$ as defined in Theorem \ref{t:Christ}.
Let $x\in \Sigma\backslash B(0,1)$, and $Q\subseteq R\in \Stop(-1)=\cD_{0}$ be such that $x\in Q$. Then
\[
\cH^{d}_{\infty}\ps{\frac{M}{2}B_{Q}\cap \Sigma\cap P_{0}\setminus B(0,1)}\gec \ell(Q)^{d}\]
and so we may find points $X=\{x_{0},...,x_{d}\}\in \frac{M}{2}B_{Q}\cap \Sigma\cap P_{0}\setminus B(0,1)$ so that $\eta(X)\gec 1$. By our choice of $P_{Q}$ (see \eqn{pqdef}) and since $X\subseteq P_{0}$, we have by \Lemma{ATlemma} that $\angle(P_{0},P_{Q})\lec \ve$. In particular, this also holds if $Q=R$, and so for $\ve>0$ small enough $\angle(P_{R},P_{Q})<\alpha$. Thus, there are no cubes from $\Stop(0)$ containing $x$, which implies that every $Q$ with $x\in Q\subseteq R\in \cD_{0}$  is in $\Up(0)\subseteq \Up(N)$ for all $N\geq 0$. In particular, $Q\not\in \Stop(N)$ for any $N\geq 0$.
\end{proof}
}

}

%
%
%
%
%

\begin{lemma}
For $x\in \Sigma$, 
\begin{equation}\label{e:dxsigman}
\dist(x,\Sigma^{N})\lec \frac{\ve}{\tau} d_{\Up(N)}(x).
\end{equation}
\label{l:dxsigman}
\end{lemma}

\begin{proof}
If $d_{\Up(N)}(x)=0$, then $x\in \Sigma^{N}$ by \Lemma{up-zero-intersection} and there is nothing to show, so assume $d_{\Up(N)}(x)>0$. 
Then, by Corollary \ref{c:up-positive-then-stop},  $x\in Q$ for some $Q\in \Stop(N)$. Let $k$ be such that $Q\in \Up(N)_{k}$, then $x_{Q}\in B_{jk}$ for some $j\in J_{k}$. Hence 
\[
\dist(x,P_{jk})\lec \ve r_{k} \stackrel{\eqn{qjkrk}}{\sim} \ve\ell(Q) \stackrel{\eqn{412b} }{\sim} \frac{\ve}{\tau} d_{\Up(Q)}(x).\]
By \eqn{49r}, there is $y\in \Sigma_{k}^{N}$ so that $|\pi_{jk}(x)-y|\lec \ve r_{k}$, and \eqn{ytosigma} implies $\dist(y,\Sigma^{N})\lec \ve r_{k}$. Combining these estimates gives 
\[
\dist(x,\Sigma^{N})\lec \ve r_{k} \sim  \frac{\ve}{\tau} d_{\Up(Q)}(x).\]
\end{proof}

\begin{lemma}\label{l:dssn}
For $x\in \Sigma^{N}$,
\begin{equation} \label{e:dssn}
\dist(x,\Sigma^{N+1})\lec \frac{\ve}{\tau} d_{\Up(N)}(x).
\end{equation}
\end{lemma}

\begin{proof}
This follows from Lemmas \ref{l:dxsigma} and \ref{l:dxsigman}.
\end{proof}

Let $C_{2}>1$. 

\begin{lemma}
Let $0<\tau<\tau_{1}(4)$. There is $M_{0}=M_{0}(C_{2})>0$ and $\ve_{0}=\ve_{0}(C_2)>0$ so that for $M>M_{0}$ and $0<\ve<\ve_{0}$, the following holds. For $N\geq 0$, $x\in \Sigma^{N}$, and $r>0$, there are planes $P_{x,r}^{N}$ that satisfy the following.
\begin{enumerate}
\item For all $x\in \Sigma^{N}$ and $r>0$,
\begin{equation}\label{e:snpxr}
d_{x,r}(\Sigma^{N},P_{x,r}^{N})\lec \ve.
\end{equation}
\item Suppose 
$x\in Q\in  \Stop(N-1)$.
For $0<r<2C_{2}\ell(Q)$, 
\begin{equation}\label{e:pxrpqs}
\angle(P_{x,r}^{N},P_{Q})\lec_{C_{2}}\alpha.
\end{equation}
\item Suppose $x\not\in \Sigma$ (so that $d_{\Up(N)}(x)>0$), and let  $Q_{x}^{N}$ be as in Lemma  \ref{l:qxn}.
For $k=k_{N}(x)$ as in Lemma \ref{l:lemma-5-4} and $0<r\leq r_{k}$,
\begin{equation}\label{e:pxrpq}
\angle(P_{x,r}^{N},P_{Q_{x}^{N}})\lec_{C_{2}}\ve
\end{equation}
\end{enumerate}
and there is a twice-differentiable $C\ve$-Lipschitz graph $\Gamma_{x}^{N}$ over $P_{x,r_{k}}^{N}$ so that
\begin{equation}
B(x,r_{k})\cap \Sigma^{N}=B(x,r_{k})\cap \Sigma_{k}^{N}=B(x,r_{k})\cap \Gamma_{x}^{N}
\label{e:sigmagraph}
\end{equation}
\label{l:sigmagraph}
\end{lemma}

\begin{proof}
Let $x\in \Sigma^{N}$ and $k=k_{N}(x)$. For $0<r\leq r_{k}$, let $\Gamma_{x}$ and $P_{x,r}^{N}=P_{x,r_{k}}^{N}$ be the graph and plane given by \Theorem{DT} (8), then \eqn{snpxr} and \eqn{sigmagraph} follow from \eqn{ygraph} and \eqn{bsnbsnk}. For $r>r_{k}$, \Lemma{skflat} implies the existence of a plane  $P_{x,r}^{N}$ satisfying \eqn{snpxr} again. Thus, we just need to verify  \eqn{pxrpqs} and \eqn{pxrpq}.

\begin{itemize}
\item \eqn{pxrpqs} for $r\geq r_{k}$: Suppose $r\geq r_{k}$. Let $x'\in \Sigma$ be the closest point to $x$ in $\Sigma$, so $x'\in Q_{x}^{N}$ by construction (of Lemma  \ref{l:qxn}). Let $R\in \cD$ be the largest parent 
such that $\ell(R)<r$. 
We have that for $M$ large 
that  $B(x',r/2)\subseteq MB_{R}$, and hence
\begin{equation}
\beta_{E,\infty}(B(x',r/2),P_{R})\lec \ve.
\label{e:Bx'r/21}
\end{equation}
Let $z\in B(x',r/2)\cap \Sigma\subseteq B(x,r)$. Then by Lemma \ref{e:dxsigman}
there is $z'\in \Sigma^{N}$ with 
\begin{align*}
|z-z'|
& \lec \ve d_{\Up(N)}(z)
\leq \ve \big(|x-x'|+|x'-z|+d_{\Up(N)}(x)\big)\\
 & \stackrel{\eqn{dx<dT}}{\lec} \ve(\ve r_{k} +r+r_{k}) 
 \lec \ve r.
\end{align*}
In particular, $z\in B(x,r)\cap \Sigma^{N}$ for $\ve$ small enough. By the definition of $P_{x,r}^{N}$, $\dist(z',P_{x,r}^{N})\lec \ve r$, and so $\dist(z,P_{x,r}^{N})\lec \ve r $ as well. Since this holds for all $z\in B(x',r/2)\cap \Sigma$, this implies
\begin{equation}
 \beta_{\Sigma,\infty}^{d}(B(x',r/2),P_{x,r}^{N}) \lec \ve .
\label{e:Bx'r/22}
\end{equation}
Equations \eqn{Bx'r/21} and \eqn{Bx'r/22} along with Lemmas \ref{l:betabeta'} and \ref{l:betabetainf} imply $\angle(P_{R},P_{x,r}^{N})\lec \ve$, \
and since $\angle(P_{R},P_{Q})\lec_{C_{2}} \ve$ (by \eqn{PQPQS} and $r<2C_{2}\ell(Q)$), we have $\angle(P_{x,r}^{N},P_{Q})\lec_{C_{2}} \alpha$.

\item \eqn{pxrpq} ($0<r\leq r_{k}$):  Let $0<r\leq r_{k}$. Since 
\[
\ell(Q_{x}^{N})\sim d_{\Up(N)}(x) \sim r_{k},\] 
we have that for $M>1$ large enough, $ MB_{Q_{x}^{N}}\supseteq B(x,r_{k})$. Note that $\eqn{dx<dT}$ and \eqn{dxsigman} imply that
$d_{x,r_{k}}(\Sigma,\Sigma^{N})
\lec \ve$ 
and so 
\begin{align*}
\angle(P_{x,r}^{N},P_{Q_{x}^{N}})
& =\angle(P_{x,r_{k}}^{N},P_{Q_{x}^{N}})
\stackrel{\eqn{dxardxr}}{=} d_{x,r_{k}/4}(P_{x,r_{k}}^{N},P_{Q_{x}^{N}})\\
& \stackrel{\eqn{dxrtrngl}}{\lec} d_{x,r_{k}}(P_{x,r_{k}}^{N},\Sigma^{N}) +  d_{x,r_{k}}(\Sigma^{N},\Sigma)+ d_{x,r_{k}}(\Sigma,P_{Q_{x}^{N}})\\
& \stackrel{\eqn{snpxr}}{\lec} \ve + \ve + d_{MB_{Q_{x}^{N}}}(\Sigma,P_{Q_{x}^{N}})
\lec \ve 
\end{align*}

\item \eqn{pxrpqs} for $r< r_{k}$: By the two previous cases,
\[
\angle (P_{x,r}^{N},P_{Q})
\leq \angle(P_{x,r}^{N},P_{Q_{x}^{N}})+\angle(P_{Q_{x}^{N}},P_{x,r_{k}}^{N})
+\angle(P_{x,r_{k}}^{N},P_{Q})
\lec_{C_{2}} \alpha.
\]

\end{itemize}

\end{proof}

Recall the definition of $\cF_N$ of extended stoping times before Lemma \ref{l:5.1}. 
Recall also that $D(\cdot, \cdot,\cdot)$ is a cylinder, as in equation \eqref{e:D-is-a-cylinder}. 

\begin{lemma}
Let $M>M_{1}(C_{2}):=\max\{M_{0}(2C_{2}),4C_{2}\}$, $\tau<\tau_{1}(2C_{2})$, $Q\in S\in \cF_{N}$ and $P$ be a $d$-plane such that $d_{B_{Q}}(P,P_{Q})<\theta$. If $\alpha,\theta$, and $\ve$ are small enough (depending on $C_{2}$), then there is a $C(\alpha+\theta)$-Lipschitz map $A_{P,Q}^{N}:P\rightarrow P^{\perp}$ that is zero outside of $P\cap B(\pi(x_{Q}),2C_{2}\ell(Q)))$ such that if $\Gamma_{P,Q}^{N}$ is the graph of $A_{P,Q}^{N}$ along $P$, then
\begin{equation}\label{e:isagraph}
 \Sigma^{N}\cap D(x_{Q},P,C_{2}\ell(Q))=\Gamma_{P,Q}^{N}\cap D(\pi_{P}(x_{Q}),P,C_{2}\ell(Q)).
 \end{equation}
 If $Q=Q(S)$, we will set $A_{S}=A_{P,Q(S)}^{N}$ and $\Gamma_{S}=\Gamma_{P,Q(S)}^{N}$. If $Q=Q_{x}^{N}$ for some $x\in \Sigma^{N}$, then $A_{P,Q}^{N}$ is $C(\ve+\theta)$-Lipschitz.
 \label{l:graph}
 \end{lemma}
 
 \begin{proof}
First, we claim that
 \begin{equation}\label{e:qynqa}
 \angle(P_{Q_{y}^{N}},P_{Q})\lec\alpha \mbox{ for }y\in 2C_{2}B_{Q}\cap \Sigma^{N}.
 \end{equation} 
 Since $Q_{y}^{N}\in \Stop(N)$, there is $S'\in \cF_N$ such that $Q_{y}^{N}\in S'$. Since $y\in 2B_{Q_{y}^{N}}$, we know $y\in 2B_{Q(S')}$ by \eqn{cqincr}. Since we also have $y\in 2C_{2}B_{Q}\subseteq 2C_{2}B_{Q(S)}$, this implies $2C_{2}B_{Q(S)}\cap 2C_{2}B_{Q(S')}\neq \emptyset$. Hence, by \Lemma{QSQS'} with $C_{1}=2C_{2}$,
 \begin{multline*}
 \angle(P_{Q},P_{Q_{y}^{N}})
 \leq \angle(P_{Q},P_{Q(S)})+\angle(P_{Q(S)},P_{Q(S')})+\angle(P_{Q(S')},P_{Q_{y}^{N}})\\
 \stackrel{ \eqn{PQPQS}}{\lec}_{C_{2}}\alpha+\ve+\alpha\lec\alpha.
 \end{multline*}
 This proves the claim.

Let $x,y\in \Sigma^{N}\cap B(x_{Q},C_{2}\ell(Q))$, $r=|x-y|$, and $P_{x,r}^{N}$ be from \Lemma{sigmagraph}. Note that
\begin{equation}\label{e:QxNcase}
\angle(P_{x,r}^{N},P_{Q}) 
\leq \angle(P_{x,r}^{N}, P_{Q_{x}^{N}})
+\angle(P_{Q_{x}^{N}},P_{Q}) 
\stackrel{ \eqn{pxrpq} \atop \eqn{qynqa} }{\lec}_{C_{2}}  \alpha.
\end{equation}

Let $x',y'$ be the projections of $x$ and $y$ into $P_{x,r}^{N}$. Then 
\begin{equation}\label{e:x'xy'y}
|x'-x|+|y'-y|\leq c'\ve r
\end{equation} 
for some universal constant $c'>0$. Moreover,

\[ \angle(P,P_{x,r}^{N})\leq \angle(P,P_{Q})+\angle(P_{Q},P_{x,r}^{N})\stackrel{\eqn{QxNcase}}{<}\theta+c\alpha=:t.\] 

Thus, by a bit of trigonometry, setting $t'=\sqrt{1-(t/2)^{2}}$, for $\ve$ small enough (depending on $M,\alpha$, and $\theta$),
\begin{align*}
|\pi_{P}(x-y)|
& \stackrel{\eqn{x'xy'y}}{\geq}  |\pi_{P}(x'-y')|-2c'\ve r
\geq t'|x'-y'|-2c'\ve |x-y|\\
&  \stackrel{\eqn{x'xy'y}}{\geq} t'|x-y|-(1+t')2c'\ve|x-y|\\
& \geq (t'-(t/2)^{2})|x-y|
\geq (1-t^{2}/2)|x-y|
\end{align*}


Assume $x_{Q}\in P$. For $x\in \pi_{P}(\Sigma^{N}\cap B(x_{Q},C_{2}\ell(Q))$, we now set $A_{P,Q}^{N}(x)=y$ where $y\in \Sigma^{N}\cap B(x_{Q},C_{2}\ell(Q))$ is such that $x=\pi_{P}(y)$. For $x\not\in P\cap B(\pi(x_{Q},2C_{2}\ell(Q)))$, we set $A_{P,Q}^{N}=0$. For $\alpha$ and $\theta$ small enough, the resulting map is $C'(\alpha+\theta)$-Lipschitz where it is defined for some universal constant $C'$, and we can extend to the rest of $P$ so that it is still $C'(\alpha+\theta)$-Lipschitz and 
\[
\Sigma^{N}\cap D(x_{Q},2C_{1}\ell(Q)) = \Gamma_{P,Q}^{N}\cap D(x_{Q},2C_{1}\ell(Q)).\]
If $x_{Q}\not\in P$, then 
\[
\dist(x_{Q},P)\leq r_{B_{Q}}d_{B_{Q}}(P,P_{Q})<\theta \ell(Q)\] 
and so we can apply our previous work to the plane $P-x_{Q}+\pi_{P}(x_{Q})$ in place of $P$, then translating the resulting graph by $x_{Q}-\pi_{P}(x_{Q})$. For $\theta>0$ small enough (depending on $C_{2}$), this gives the result.

If $Q=Q_{x}^{N}$, we run the same proof, only instead of \eqn{QxNcase} we use \eqn{pxrpq} to bound $\angle(P_{x,r}^{N},P_{Q}) \lec\ve$, and then replace each instance of $\alpha$ in the proof with $\ve$.
 \end{proof}

\section{The map between layers}\label{s:layers}

\begin{remark}\label{r:tau2}
We fix $\tau>0$ so that the results of the previous two sections hold, that is, we pick $\tau<\min\{\tau_{0},\tau_{1}(4),\tau_{2}(3C_{2}),c_{0}/4\}$, where $C_{2}$ can be freely fixed to be any constant larger than $4C_0$.  We explicitly fix $C_2$ in  Remark \ref{r:c2} where it is relevant.
\end{remark}

In this section, we will construct a map $F_{N}:\Sigma^{N}\rightarrow \Sigma^{N+1}$. First,  set
\begin{equation}
F_{N}(x) =x \mbox{ for }x\in \Sigma^{N}\cap \Sigma.
\end{equation}
For $x\in \Sigma^{N}\backslash \Sigma$, let $k=k_{N}(x)$, so by \Lemma{dxsigma}, $r_{k}\sim d_{\Up(N)}(x)$ and
\[B(x,r_{k})\cap\Sigma^{N}=B(x,r_{k})\cap \Sigma_{k}^{N}.\]
In particular, $B(x,r_{k})\cap \Sigma^{N}$ is a smooth surface. Let $V_{x}^{N}$ be the $d$-dimensional tangent plane to $\Sigma^{N}$ at $x$, and $W_{x}^{N}=(V_{x}^{N})^{\perp}+x$.

\begin{lemma}
For $x\in \Sigma^{N}\backslash\Sigma$, let $k=k_{N}(x)$, $W$ be a $(n-d)$-plane passing through $x$ with $\angle(W,W_{x}^{N})\leq \theta$. Then for $\theta$ small enough, $W\cap B(x,10r_{k})\cap \Sigma^{N+1}$ contains exactly one point $z$ with $|x-z|\lec \ve d_{\Up(N)}(x) $.
\label{l:W}
\end{lemma}

\begin{proof}
Let $V$ the $d$-plane perpendicular to $W$ passing through $x$ and $\Gamma=\Gamma_{V,Q_{x}^{N}}^{N+1}$. 
Recalling \Lemma{sigmagraph}, as $x\in \Sigma^{N}\backslash \Sigma$, $V_{x}^{N}$ is tangent to $\Gamma_{x}^{N}$ which is a $C\ve$-Lipschitz graph over $P_{x,r_{k}}^{N}$, we have $\angle(V_{x}^{N},P_{x,r_{k}}^{N})\lec \ve$.  Thus, 
\[
\angle(P_{Q_{x}^{N}},V_{x}^{N})
\leq \angle(P_{Q_{x}^{N}},P_{x,r}^{N})
+
\angle(P_{x,r_{k}}^{N},V_{x}^{N}) \stackrel{\eqn{pxrpq}}{\lec} \ve.\]

Moreover, $\angle(V,V_{x}^{N})=\angle(W,W_{x}^{N})<\theta$. By \Lemma{graph}, for $\ve$ and $\theta$ small enough, we have
\[
\Sigma^{N+1}\cap D\ps{x_{Q_{x}^{N}},V,C_{2}\ell(Q_{x}^{N})}=\Gamma_{V,Q_{x}^{N}}^{N+1}\cap D\ps{x_{Q_{x}^{N}},V,C_{2}\ell(Q_{x}^{N})}.\]
In particular, 
\[
\Sigma^{N+1}\cap D\ps{x_{Q_{x}^{N}},V,C_{2}\ell(Q_{x}^{N})}\cap W
=\{A_{V,Q_{x}^{N}}^{N+1}(x)+x\}.\]
Let $z=A_{V,Q_{x}^{N}}^{N+1}(x)+x$. Since $A_{V,Q_{x}^{N}}^{N+1}$ is $C(\alpha+\theta)$-Lipschitz and vanishes outside $V\cap B(\pi_{V}(x_{Q_{x}^{N}}),2C_{2}\ell(Q_{x}^{N}))$ by \Lemma{graph}
\[
|z-x|
=|A_{V,Q_{x}^{N}}(x)| 
\lec (\alpha +\theta)\ell(Q_{x}^{N})
\sim (\alpha +\theta)  d_{\Up(N)}(x)
\sim (\alpha +\theta) r_{k}\]
For $\alpha$ and $\delta$ small, we know that $|z-x|<\ell(Q_{x}^{N})$, and since $x\in 2B_{Q_{x}^{N}}$ by \Lemma{qxn}, we have $z\in 3B_{Q_{x}^{N}}$. Then
\begin{align*}
|z-x|
& \leq \dist(z,V_{x}^{N})
 \leq 3\ell(Q_{x}^{N}) d_{3B_{Q_{x}^{N}}}(\Sigma^{N+1},V_{x}^{N})\\
& \lec  \ell(Q_{x}^{N}) \ps{ d_{6B_{Q_{x}^{N}}}(\Sigma^{N+1},\Sigma) + d_{6B_{Q_{x}^{N}}}(\Sigma, V_{x}^{N})}\\
& \lec  \ell(Q_{x}^{N}) \ps{  d_{6B_{Q_{x}^{N}}}(\Sigma^{N+1},\Sigma)   + d_{12B_{Q_{x}}^{N}}(\Sigma, P_{Q_{x}^{N}}) + d_{12B_{Q_{x}}^{N}}( P_{Q_{x}^{N}},V_{x}^{N})})\\ 
& = \ell(Q_{x}^{N})(I_{1}+I_{2}+I_{3}).
\end{align*}
Frist off, using \eqn{dx<dT} and \eqn{dxsigman} and the fact that $x\in 2B_{Q_{x}}^{N}$ and $Q_{x}^{N}\in \Up(N)\subseteq \Up(N+1)$, we have
\begin{align*}
I_{1}
& \leq \frac{\sup_{y\in 6B_{Q_{x}^{N}}\cap \Sigma^{N+1}} \dist(y,\Sigma)+\sup_{y\in 6B_{Q_{x}}^{N}\cap \Sigma} \dist(y,\Sigma^{N+1})}{6\ell(B_{Q_{x}}^{N})}\\
& \lec  \frac{\sup_{y\in 6B_{Q_{x}}^{N}\cap (\Sigma^{N+1}\cup \Sigma)} \ve d_{\Up(N+1)}(y)}{6\ell(B_{Q_{x}^{N}})}\\
& \leq \ve \frac{d_{\Up(N+1)}(x)+6\ell(B_{Q_{x}^{N}})}{6\ell(B_{Q_{x}^{N}})}
\leq \ve \frac{ 2\ell(Q_{x}^{N})+6\ell(B_{Q_{x}^{N}})}{6\ell(B_{Q_{x}^{N}})}\lec \ve
\end{align*}
Next, by our choice of $P_{Q_{x}^{N}}$,
\[
I_{2}= d_{12B_{Q_{x}}^{N}}(\Sigma, P_{Q_{x}^{N}}) 
\lec d_{MB_{Q_{x}}^{N}}(\Sigma, P_{Q_{x}^{N}}) \lec \ve.\]
Finally, note that by \eqn{dx<dT}, there is $y\in \Sigma$ with $|x-y|\lec \ve d_{\Up(N)}(x)\sim \ve \ell(Q_{x}^{N})$. For $\ve>0$ small, enough, $y\in \Sigma\cap MB_{Q_{x}^{N}}$, and so by our choice of plane $P_{Q_{x}^{N}}$,
\[
\dist(x,P_{Q_{x}^{N}})
\leq |x-y|+\dist(y,P_{Q_{x}^{N}})
\lec \ve \ell(Q_{x}^{N})+ \ve M\ell(Q_{x}^{N})\lec \ve \ell(Q_{x}^{N}).\]
Hence, since $12B_{Q_{x}}^{N}\subseteq B(x,14\ell(Q_{x}^{N}))$,
\begin{align*}
I_{3}
& \leq 
 d_{12B_{Q_{x}^{N}}}( P_{Q_{x}^{N}}, P_{Q_{x}^{N}}-\pi_{P_{Q_{x}^{N}}}(x)  )
 + d_{12B_{Q_{x}}^{N}}(P_{Q_{x}^{N}}-\pi_{P_{Q_{x}^{N}}}(x)+x ,V_{x}^{N}))\\
& \lec \ve + d_{x,14\ell (Q_{x}^{N})}(P_{Q_{x}^{N}}-\pi_{P_{Q_{x}^{N}}}(x)+x,V_{x}^{N}))
=\ve + \angle (P_{Q_{x}^{N}},V_{x}^{N}) \lec \ve.
\end{align*}
Combining our estimate for $I_{1}$, $I_{2}$, and $I_3$ gives $|x-z|\lec \ve \ell(Q_{x}^{N})\sim \ve d_{\Up(N)}(x)$, which finishes the lemma.

\end{proof}

Now we set
\[A_{x}^{N+1}:=A_{V_{x},Q_{x}^{N}}^{N+1} \;\mbox{ and }\; \Gamma_{x}^{N}=\Gamma_{V_{x},Q_{x}}^{N+1}.\]
By the previous lemma, applied to $W=V_{x}^{\perp}$ (and so $\theta=0$), and using the fact that $Q_{x}^{N}\in \Stop(N)$, we know $A_{x}$ is $(1+C\alpha)$-Lipschitz by \Lemma{graph}. We also set
\begin{equation}
F_{N}(x)=A_{x}(x)+x\in \Sigma^{N+1}\mbox{ for }x\in \Sigma^{N}\backslash \Sigma. 
\end{equation}
By \eqn{sn01}, 
\begin{equation}\label{e:fnid}
F_{N}(x)=x \mbox{ for }x\in \Sigma^{N}\backslash B(0,1)=P_{0}\backslash B(0,1)
\end{equation}
and by \Lemma{W},
\begin{equation}\label{e:fnx-x}
|F_{N}(x)-x|\lec \ve d_{\Up(N)}(x).
\end{equation}

\begin{lemma}
$F_{N}:\Sigma^{N}\rightarrow \Sigma^{N+1}$ is $(1+C\alpha)$-bi-Lipschitz. Moreover, (and recalling that $\epsilon\ll\alpha$)
 there is $C>0$ so that
\begin{equation}\label{e:FN}
(1-C\ve)|x-y|\leq |F_{N}(x)-F_{N}(y)|\leq (1+C\alpha)|x-y| \; \mbox{ for all } \; x,y\in\Sigma^{N},
\end{equation}
\label{l:FN}
\end{lemma}

\begin{proof}
Let $x,y\in \Sigma^{N}$ and let $k=k_{N}(x)$. If $y\not\in B(x,r_{k}/2)$, then since $d_{\Up}$ is $1$-Lipschitz,
\begin{align}
\big||F_{N}(x) & -F_{N}(y)|-|x-y|\big|
 \leq |F_{N}(x)-x|+|F_{N}(y)-y| \notag \\
& \stackrel{\eqn{fnx-x}}{\lec} \ve( d_{\Up(N)}(x)+  d_{\Up(N)}(y))
 \leq \ve (2d_{\Up(N)}(x)+|x-y|) \notag \\
& \stackrel{\eqn{dx<dT}}{\lec} \ve (r_{k}+|x-y|)
\lec \ve |x-y|
\label{e:FN2}
\end{align}
Thus, it suffices to consider the case that $y\in B(x,r_{k}/2)$. Let $x'=F_{N}(x)$ and $y'=F_{N}(y)$.
See Figure \ref{f:fn}.

\begin{figure}[h]
\includegraphics[width=350pt]{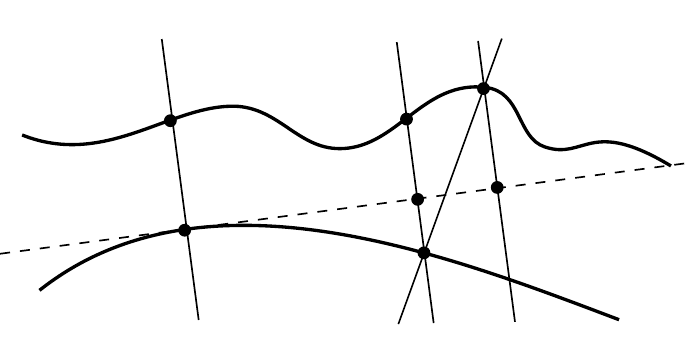}
\begin{picture}(0,0)(350,0)
\put(75,65){$x$}
\put(70,120){$x'$}
\put(0,58){$V_{x}^{N}$}
\put(90,0){$W_{x}^{N}$}
\put(245,140){$y'$}
\put(185,120){$y''$}
\put(215,55){$y$}
\put(237,86){$z$}
\put(195,80){$z'$}
\put(185,0){$W_{y}^{N}$}
\put(220,0){$W$}
\put(275,35){$\Sigma^{N}$}
\put(300,115){$\Sigma^{N+1}$}
\end{picture}
\caption{The diagram for the proof of \Lemma{FN}}
\label{f:fn}
\end{figure}
Let $W=W_{x}^{N}-x+y$. By \Lemma{W}, there is a (unique) $y''\in W\cap \Sigma^{N+1}$ with 
\[
|y-y''|\lec  \ve d_{\Up(N)}(y)
\leq \ve(|x-y|+d_{\Up(N)}(x))\lec \ve r_{k}.\]
Let $z=\pi_{V_{x}^{N}}(y')\in V_{x}^{N}$ and 
\[z'=\pi_{V_{x}^{N}}(y)=\pi_{V_{x}^{N}}(y'')\in W\cap V_{x}^{N}.\]
 Since $y\in B(x,r_{k})$ and by \Lemma{graph}, $B(x,r_{k})\cap\Sigma^{N}$ is contained in the graph of some $C\ve$-Lipschitz function $A$ along $V_{x}^{N}$, we have
\begin{align*} 
|y-z'|
& =|A(z')|
\leq |A(x)|+\ve|x-z'|
=0+\ve|x-z'|\\
& \leq \ve|x-y|+\ve|y-z'|\end{align*}
 and hence for $\ve$ small
 \begin{equation}\label{e:y-z'}
 |y-z'|\lec \ve |x-y|.
 \end{equation}
Let $y'''=\pi_{W}(y')\in W\cap \cnj{B(y,C\ve r_{k})}$. Then

\begin{align}\label{e:z-z'}
 |z-z'|
&  =|\pi_{V_{x}^{N}}(y''-y')|=|\pi_{V_{x}^{N}}(y'''-y')|
 \leq |y'''-y'| \notag \\
 & \leq d_{y,C\ve r_{k}}(W,W_{y}^{N}) C\ve r_{k}
 = \angle(W,W_{y}^{N}) C\ve r_{k}
  =\angle(V_{x}^{N},V_{y}^{N})C\ve r_{k} \notag \\
&  \stackrel{\eqn{Tlip}}{ \lec} \frac{\ve}{r_{k}}|x-y| C\ve r_{k}
 =C\ve^{2}|x-y|.
 \end{align}

 Hence,
 \begin{equation}\label{e:xz-xy}
\av{ |x-z|-|x-y|}
\leq |y-z|
\leq |y-z'|+|z'-z|
\stackrel{\eqn{y-z'} \atop \eqn{z-z'}}{\lec} 
 \ve |x-y|.\end{equation}
Thus, recalling that $A_{x}^{N+1}$ is $(1+C\alpha)$-Lipschitz,
\begin{align*}
|x'-y'|
& =\sqrt{|x-z|^{2}+|A_{x}(x)-A_{x}(z)|^{2}}
 \leq |x-z|+|A_{x}(x)-A_{x}(z)|\\
& \leq (1+C\alpha)|x-z|
  \stackrel{\eqn{xz-xy}}{\leq}
(1+C\alpha)(1+C\ve)|x-y|\\
& \leq (1+C\alpha)|x-y|
\end{align*}
for some perhaps larger value of $C$ in the last inequality and $\ve\ll \alpha$. 
Moreover, since $\pi_{V_{x}^{N}}(x')=x$ and $\pi_{V_{x}^{N}}(y')=z$ and projections are $1$-Lipschitz
\[
|x'-y'| 
\geq |x-z| 
\stackrel{\eqn{xz-xy}}{\geq} (1-C\ve)|x-y|.
\]
Since $F_{N}(x)=x'$ and $F_{N}(y)=y'$, these two inequalities finish the proof of \eqn{FN}.
\end{proof}

\section{The telescoping sum}\label{s:tel-sum}

The proof Theorem \ref{t:flat-case} will be completed in Section \ref{s:tel-sum} and \ref{s:beta-omega}. 
The main objective of this section is to prove the following proposition.

\begin{proposition}\label{p:asum}
\begin{equation}
\sum_{N\geq 0}\sum_{Q\in \Stop(N)}\ell(Q)^{d}
\lec_{d,\ve,\alpha} \cH^{d}(\Sigma\cap B(0,1)).
\label{e:asum}
\end{equation}
\end{proposition}

This will follow from the lemmas below. For $Q\in \Stop(N)$, set
\begin{align*}
a_{Q}
& =\cH^{d}(\Sigma^{N+1}\cap D(x_{Q},P_{Q},C_{2}\ell(Q)))-|B(x_{Q},C_{2}\ell(Q))\cap P_{Q}|\\
& =\cH^{d}(\Gamma_{Q}^{N+1}\cap D(x_{Q},P_{Q},C_{2}\ell(Q))-|B(x_{Q},C_{2}\ell(Q))\cap P_{Q}|
\end{align*}

\begin{lemma}\label{l:asum2}
For $0<\ve\ll \alpha^{4}$,
\begin{equation}
\label{e:asum2}
\sum_{N\geq 0}\sum\{\ell(Q)^{d}: Q\in \Stop(N),\  \sqrt{\ve} \ell(Q)^{d}>a_{Q}\}
\lec_{d} \cH^{d}(\Sigma\cap B(0,1)).
\end{equation}
\end{lemma}

\begin{proof}

Let $Q\in \Stop(N)$, 
and assume $a_{Q}<\sqrt{\ve}\ell(Q)^{d}$. 

Let $R\in m(S_{Q})$. By \Lemma{>a}, there is $R'\in \Stop(N+1)$ with $R'\subseteq R$, $\ell(R)\sim \ell(R')$, and $\angle(P_{R'},P_{Q})\gec \alpha$. By \Lemma{dxsigman}, there is $x\in \Sigma^{N+1}$ so that $|x_{R'}-x|\lec \ve \ell(R')$, so in particular, $x\in c_{0} B_{R'}$ for $\ve$ small enough. By \Lemma{qxn}, $Q_{x}^{N+1}=R'$. 
{ Let $k=k_{N+1}(x)$. By \Lemma{sigmagraph}, $B(x,r_{k})\cap \Sigma^{N+1}=B(x,r_{k})\cap \Gamma_{x}^{N+1}$ for some Lipschitz function along $P_{x,r_{k}}^{N+1}$ where $\angle(P_{x,r_{k}}^{N+1},P_{Q_{x}^{N}})=\angle(P_{x,r_{k}}^{N+1},P_{R'})\lec \ve$. Thus, for $\ve\ll \alpha$,
\[
\angle(P_{x,r_{k}}^{N+1},P_{Q})\geq \angle(P_{R'},P_{Q})-\angle(P_{R'},P_{x,r_{k}}^{N+1})\gec \alpha-\ve \gec \alpha.\]
}
Since $\Gamma_{x}^{N+1}$ is $C\ve$-Lipschitz and $C^{2}$, the tangent plane at any $y\in B(x,r_{k})\cap \Sigma^{N+1}$ to $\Sigma^{N+1}$ has angle $\ve$ from $P_{x,r}^{N+1}$, and thus $\gec \alpha$ from $P_{R'}$ for $\ve$ small. Thus, $|DA_{Q,P_{Q}}^{N+1}(\pi_{Q}(y))|\gec \alpha$ for all $y\in B(x,r_{k})\cap \Sigma^{N+1}$, so in particular, for all $y\in \frac{c_{0}}{2}B_{R'}\cap \Sigma^{N+1}$. For $\alpha$ small enough, we have that $\{\pi_{Q}( \frac{c_{0}}{2}B_{R'}\cap \Sigma^{N+1}):R\in m(S_{Q})\}$ are disjoint sets contained in $2B_{Q}\cap P_{Q}$. If we set 
\[
f_{Q}=x+A_{Q,P_{Q}}^{N+1}(x),\] 
then
\begin{align*}
\sum_{R\in m(S_{Q})} \ell(R)^{d}
& \sim \sum_{R\in m(S)} \av{\pi_{Q}\ps{ \frac{c_{0}}{2}B_{R}\cap \Sigma^{N+1}}}\\
& \lec  \alpha^{-2} \int_{2B_{Q}\cap P_{Q}} |DA_{Q,P_{Q}}^{N+1}|^{2}\\
& \lec \alpha^{-2}\int_{2B_{Q}\cap P_{Q}}(\sqrt{1+|DA_{Q,P_{Q}}^{N+1}|^{2}}-1)\\
& =\alpha^{-2}\int_{2B_{Q}\cap P_{S}} (J_{f_{Q}}-1)\\
& =\alpha^{-2}\Big(\cH^{d}(\Gamma_{Q,P_{Q}}^{N+1}\cap D(x_{Q},P_{Q},C_{2}\ell(Q)))\\
& \qquad -|B(x_{Q},C_{2}\ell(Q))\cap P_{Q}|\Big)\\
& \leq \alpha^{-2}{a_{Q}}
\lec \frac{\sqrt{\ve}}{\alpha^{2}}\ell(Q)^{d}
\end{align*}
Thus, since $S_{Q}\subseteq S_{Q}'$, for $\ve\ll \alpha^{4}$,
(and recalling  that $z(S)$ which was defined in \eqref{e:z-def})
\begin{align*}
\cH^{d}(z(S_{Q}'))
& \geq  \cH^{d}(z(S_{Q}))
 \geq \cH^{d}_{\infty}(Q)-\cH^{d}_{\infty}\ps{\bigcup_{R\in m(S_{Q})}R}\\
& \gec \ell(Q)^{d}- C\frac{\sqrt{\ve}}{\alpha^{2}}\ell(Q)^{d}
\gec \ell(Q)^{d}.\end{align*}
Since the $\{z(S):S\in \cF_{N},\;\; N\geq0 \}$ are mutually disjoint by Lemma \ref{l:5.1} and $z(S(Q))\subseteq Q\subseteq B(0,1)$ for each $Q\in \Stop(N)$ by \Lemma{wherestop}, 
\begin{align*}
\sum_{N\geq 0}\sum_{Q\in \Stop(N)\atop \sqrt{\ve} \ell(Q)^{d}>a_{Q}}\ell(Q)^d
& \lec\sum_{N\geq 0}\sum_{Q\in \Stop(N)\atop \sqrt{\ve} \ell(Q)^{d}>a_{Q}} \cH^{d}(z'(S_{Q}))
\leq \cH^{d}(\Sigma\cap B(0,1))
\end{align*}

\end{proof}

\begin{remark}\label{r:alpha}
We now fix $\alpha>0$ so that the previous lemma holds.
\end{remark}

While $a_Q$ was the difference in area from a flat disk, we will need another quantity, $t_Q$, which is the local difference in area between  $\Sigma^{N+1}$ and $\Sigma^N$. Specifically,
\begin{multline*}
t_{Q}
 = \ps{\cH^{d}(F_{N}(C_{1}B_{Q}\cap \Sigma^{N}))-\cH^{d}(C_{1}B_{Q}\cap \Sigma^{N})}\\
=\int_{C_{1}B_{Q}\cap \Sigma^{N}}(J_{F_{N}}(x)-1)d\cH^{d}(x).
\end{multline*}

\begin{lemma}
For $Q\in \Stop(N)$,
\begin{equation}\label{e:a<t}
a_{Q}\lec t_{Q}+\ve\ell(Q)^{d}.
\end{equation}
\end{lemma}

\begin{proof}
For $x\in \Sigma^{N}\cap B(x_{Q},C_{1}\ell(Q))$, let
\[ f(x)=\pi_{P_{Q}}\circ F_{N}(x).\]
Our goal is to show that 
\begin{equation}\label{e:d-1}
|Df(x)-I|\lec \ve 
\end{equation}
since then $|J_{f^{-1}}(x)-1|\lec\ve $ and so
\begin{align*}
a_{Q}
& =\cH^{d}\ps{\Sigma^{N+1}\cap D(x_{Q},P_{Q},C_{2}\ell(Q))}-|B(x_{Q},C_{2}\ell(Q))\cap P_{Q}|\\
& = \cH^{d}\ps{\Sigma^{N+1}\cap D(x_{Q},P_{Q},C_{2}\ell(Q))} \\
& \qquad - |F_{N}^{-1}\ps{\Sigma^{N+1}\cap D(x_{Q},P_{Q},C_{2}\ell(Q))}| \\
& \qquad \qquad + |F_{N}^{-1}\ps{\Sigma^{N+1}\cap D(x_{Q},P_{S},C_{2}\ell(Q))}|\\
&  \qquad \qquad \qquad -|B(x_{Q},C_{2}\ell(Q))\cap P_{Q}|\\
& 
\leq t_{Q} + \int_{B(x_{Q},C_{2}\ell(Q))\cap \Sigma^{N}}(J_{f^{-1}}(x)-1)d\cH^{d}(x)\\
& \lec t_{Q}+\ve \cH^{d}(B(x_{Q},C_{2}\ell(Q))\cap \Sigma^{N})\\
&\stackrel{\eqn{isagraph}}{\leq} t_{Q}+\ve  \cH^{d}(\Gamma_{P_{Q},Q}^{N}\cap D(x_{Q},P_{Q},C_{2}\ell(Q)))\\
& \lec t_{Q}+\ve \ell(Q)^{d}.
\end{align*}
Thus, it remans to show \eqn{d-1}. We will again use Figure \ref{f:fn}. Let $y\in B(x,r_{k})$, $x'=F_{N}(x)$, $y'=F_{N}(y)$, and $z=\pi_{V_{x}^{N}}(y')\in V_{x}^{N}$. Since $x-z=\pi_{V_{x}^{N}}(x'-y')$, $\angle(V_{x}^{N},P_{Q})\lec \ve$, and $A_{x}$ is $C\alpha$-Lipschitz, we have
\begin{align*}
|(f(x)& -f(y))-(x-y)|
 = |\pi_{P_{Q}}(x'-y')-(x-y)|\\
& \leq  |\pi_{P_{Q}}(x'-y')-(x-z)|+|z-y|\\
& \leq |\pi_{P_{Q}}(x'-y')-\pi_{P_{Q}}(\pi_{V_{x}^{N}}(x'-y'))|\\
& \qquad + |\pi_{P_{Q}}(x-z)-(x-z)|+|z-y|\\
& =|\pi_{P_{Q}}(\pi_{W_{x}^{N}}(x'-y'))| + |\pi_{P_{Q}^{\perp}}(x-z)|+|z-y|\\
& \stackrel{\eqn{vvperp}}{\lec} \ve |\pi_{W_{x}^{N}}(x'-y')|+\ve|x-z|+|z-y|\\
& 
\leq \ve |x'-y'|+\ve|x-z|+|z-y|\\
& \stackrel{\eqn{xz-xy}}{\lec}  \ve |F_{N}(x)-F_{N}(y)|+\ve|x-y|
 \stackrel{\eqn{FN2}}{\lec} \ve |x-y|
\end{align*}
which implies \eqn{d-1}.
\end{proof}

\begin{lemma}\label{l:asum3}
\[
\sum_{N\geq 0 }\sum\{\ell(Q)^{d}: Q\in \Stop(N),\  a_{Q}\geq \sqrt{\ve}\ell(Q)^{d}\}  \lec \cH^{d}(B(0,1)\cap \Sigma).\]
\end{lemma}


\begin{proof}
Note that by \eqn{a<t} that $a_{Q}\geq \sqrt{\ve}\ell(Q)^{d}$ implies  
\begin{align*}
\sqrt{\ve}\ell(Q)^{d} 
& \lec t_{Q}
 = \ps{\cH^{d}(F_{N}(C_{1}B_{Q}\cap \Sigma^{N}))-\cH^{d}(C_{1}B_{Q}\cap \Sigma^{N})}\\
& =\int_{C_{1}B_{Q}\cap \Sigma^{N}}(J_{F_{N}}(x)-1)d\cH^{d}(x).
\end{align*}
By \Lemma{QSQS'}, The balls $C_{1}B_{Q}$ have bounded overlap with constant depending on $n$ and are contined in $C_{1}B_{Q_{0}}$ by \eqn{cqincr}. Also, by \eqn{dxsigman} and \eqn{412b}, $\{C_{1}B_{Q}:Q\in \Stop(N)\}$ form a cover of $\Sigma_{N}$. Finally, recalling \eqn{fnid}, $F_{N}(x)=x$ and hence $J_{F_{N}}(x)=1$ for $x\not\in B(0,1)$. These facts imply
\begin{align*}
\sum_{Q\in \Stop(N)\atop a_{Q}\geq \sqrt{\ve}\ell(Q)^{d}} \sqrt{\ve}\ell(Q)^{d}
& \lec \sum_{Q\in \Stop(N)}
\int_{C_{1}B_{Q}\cap \Sigma^{N}}(J_{F_{N}}(x)-1)d\cH^{d}(x)\\
& \leq \sum_{Q\in \Stop(N)}
\int_{C_{1}B_{Q}\cap \Sigma^{N}}(J_{F_{N}}(x)-1)_{+}d\cH^{d}(x)\\
& 
\lec \int_{\bigcup_{Q\in \Stop(N)}C_{1}B_{Q}\cap \Sigma^{N}}(J_{F_{N}}(x)-1)_{+}d\cH^{d}(x)\\
& = \int_{B(0,1)\cap\Sigma^{N}}(J_{F_{N}}(x)-1)_{+}d\cH^{d}(x)\\
& \leq \int_{B(0,1)\cap \Sigma^{N}}(J_{F_{N}}(x)-1)d\cH^{d}(x) \\
& \qquad + \int_{\bigcup_{Q\in \Stop(N)} C_{1}B_{Q}\cap \Sigma^{N}}(J_{F_{N}}(x)-1)_{-}d\cH^{d}(x)  \\ 
& \leq \int_{B(0,1)\cap \Sigma^{N}}(J_{F_{N}}(x)-1)d\cH^{d}(x) \\
& \qquad + \sum_{Q\in \Stop(N)}\int_{ C_{1}B_{Q}\cap \Sigma^{N}}(J_{F_{N}}(x)-1)_{-} d\cH^{d}(x) 
\end{align*}
Note that by \eqn{FN} that $J_{F_N}\geq 1-C\ve $, so in particular, $(J_{F_{N}}-1)_{-}\lec \ve$. 
Hence, we have
\begin{align*}
\sum_{Q\in \Stop(N)\atop a_{Q}\geq \sqrt{\ve}\ell(Q)^{d}} \sqrt{\ve}\ell(Q)^{d}
& \lec  \int_{ B(0,1)\cap \Sigma^{N+1}}(J_{F_{N}}(x)-1)d\cH^{d}(x) \\
& \qquad + \sum_{Q\in \Stop(N)}\int_{ C_{1}B_{Q}\cap \Sigma^{N}}(J_{F_{N}}(x)-1)_{-}d\cH^{d}(x)  \\
& \lec \cH^{d}(B(0,1)\cap \Sigma^{N+1})- \cH^{d}(B(0,1)\cap \Sigma^{N})\\
& \qquad + \sum_{Q\in \Stop(N)}\ve\ell(Q)^{d}\\
& \leq \cH^{d}(B(0,1)\cap \Sigma^{N+1})- \cH^{d}(B(0,1)\cap \Sigma^{N})\\
& \qquad + \sum_{Q\in \Stop(N) \atop a_{Q}\geq \sqrt{\ve}\ell(Q)^{d}} \ve\ell(Q)^{d}+\sum_{Q\in \Stop(N) \atop a_{Q}< \sqrt{\ve}\ell(Q)^{d}} \ve\ell(Q)^{d}.
\end{align*}

Thus, for $\ve>0$ small enough, and summing over $N$, this implies 
\begin{align*}
\sum_{N\geq 0} \sum_{Q\in \Stop(N)\atop a_{Q}\geq \sqrt{\ve}\ell(Q)^{d}} \sqrt{\ve}\ell(Q)^{d}
& \lesssim \sum_{N\geq 0} \cH^{d}(B(0,1)\cap \Sigma^{N+1})- \cH^{d}(B(0,1)\cap \Sigma^{N})\\
& \qquad + \sum_{N\geq 0}\sum_{Q\in \Stop(N) \atop a_{Q}< \sqrt{\ve}\ell(Q)^{d}} \ve\ell(Q)^{d}
\end{align*}
The second sum is controlled by \eqn{asum2}, so we just need to control the first. To this end, observe that for $L\in \bN$,
\begin{multline*}
\sum_{N=0}^{L-1} \ps{\cH^{d}(F_{N}(B(0,1)\cap \Sigma^{N}))-\cH^{d}(B(0,1)\cap \Sigma^{N})}\\
\lec_{d}\cH^{d}(B(0,1)\cap \Sigma^{L})-\cH^{d}(B(0,1)
\end{multline*}
It now suffices to show
\begin{equation}
\limsup_{L\rightarrow\infty}\cH^{d}(B(0,1)\cap \Sigma^{L})\lec\cH^{d}(B(0,1)\cap \Sigma).
\end{equation}
Let $x\in \Sigma^{L}\backslash \Sigma\cap B(0,1)$. By \Lemma{sigmagraph}, $B(x,r_{k_{N}(x)})\cap \Sigma^{L}$ is contained in a $C\ve$-Lipschitz graph, and thus $\cH^{d}(B(x,r_{k_{N}(x)})\cap \Sigma^{L}) \lec r_{k_{N}(x)}^{d}$. By the Vitali covering theorem, there is a collection of disjoint balls $B_{j}:=B(x_{j},r_{k_{N}(x_{j})}/5)$ so that $5B_{j}$ cover $\Sigma^{N}\backslash \Sigma$. Moreover, as $\dist(x_{j},\Sigma)\lec \ve r_{k_{N}(x_{j})}$, for $\ve$ small enough there is $x_{j}'\in \Sigma$ so that $B(x_{j}',r_{k_{N}(x_{j})}/10)\subseteq B_{j}$. Since $r_{k_{N}(x)}\leq 1$ for all $x\in \Sigma^{N}$, and $x_{j}\in B(0,1)$,
\[B_{j}\subseteq B\ps{0,1+\frac{1}{5}}\subseteq B(0,2)\]
and so
\begin{align}\label{e:sl-s}
\cH^{d}(B(0,1)\cap & \Sigma^{L}\backslash \Sigma)
 \leq \sum \cH^{d}(5B_{j}\cap \Sigma^{L})
\lec \sum r_{k(x_{j})}^{d} \notag \\
& \lec \sum_{j} \cH^{d}(B(x_{j}',r_{k(x_{j})}/10)\cap \Sigma) \notag \\
&  \leq \cH^{d}\ps{\bigcup B_{j} \cap\Sigma}
\leq \cH^{d}(B(0,2)\cap \Sigma)
\end{align}

{
Also,
\begin{align}\label{e:s02s01}
\cH^{d}(\Sigma\cap (B(0,2)\backslash B(0,1)))
& = \cH^{d}(P_{0}\cap (B(0,2)\backslash B(0,1))) \notag \\
& \lec 1 
\stackrel{\eqn{sigmalr}}{\lec} \cH^{d}(B(0,1)\cap \Sigma)
\end{align}
Thus,
\begin{align*}
\cH^{d}(B(0,1)\cap \Sigma^{L})
& \leq \cH^{d}(B(0,1)\cap \Sigma)+ \cH^{d}(B(0,1)\cap \Sigma^{L}\backslash \Sigma)\\
& \stackrel{\eqn{sl-s}}{\lec} \cH^{d}(B(0,2)\cap \Sigma)
\stackrel{\eqn{s02s01}}{\lec}  \cH^{d}(B(0,1)\cap \Sigma).
\end{align*}

}

\end{proof}

\begin{proof}[Proof of Proposition \ref{p:asum}]
The  proposition  now follows from Lemmas \ref{l:asum2} and \ref{l:asum3}.
\end{proof}

\section{Reducing  $\beta$'s to $\Omega$'s and the completion of the proof of Theorem \ref{t:flat-case}}\label{s:beta-omega}

Let $1\leq p <p(d)$, where $p(d)$ is as in \eqref{e:pd}. Recall the definition of $\beta_{\Sigma}^{d,p}$ in Definition \ref{d:beta-h-def} (as well as the content of subsection \ref{subsec:meas-and-cont}).  We will suppress the superscript $d$ from our notation in this section. \postRef{The estimates here are modelled on those in \cite[Chapter 15]{DS}.}

Let $C_{0}$ as in Theorem \ref{t:flat-case}.

Let $S\in \cF^{N}$. For $Q\in S$, let $x_{Q}'\in \Sigma^{N}$ be closest to $x_{Q}$ and $B_{Q}'=B(x_{Q}',\ell(Q))$. Then
\begin{align*}
\sum_{Q\in S}\beta_{\Sigma}^{d,p}(C_{0}B_{Q})^{2}\ell(Q)^{d}
& \stackrel{\eqn{pushbeta}}{\lec} \sum_{Q\in S} \beta_{\Sigma^{N}}^{d,p}(2C_{0}B_{Q}')^{2}\ell(Q)^{d}\\
& \qquad \preRef{+ \sum_{Q\in S} \int_{2C_{0}B_{Q}\cap \Sigma}\frac{\dist(y,\Sigma^{N})}{\ell(Q)} d\cH^{d}_{\infty}(y)
=I_{1}+I_{2}.}\\
& \postRef{ \qquad + \sum_{Q\in S}\ps{ \frac{1}{\ell(Q)^{d}}\int_{2C_{0}B_{Q}\cap \Sigma} \ps{\frac{\dist(y,\Sigma^{N})}{\ell(Q)}}^{p} d\cH^{d}_{\infty}(y)}^{\frac{2}{p}}\ell(Q)^{d}}\\
& \postRef{=I_{1}+I_{2}.}
\end{align*}

We begin by bounding $I_2$.
Note that if $y\in R$ for some $R\in \Stop(N)$, then
\begin{equation}\label{e:d<R}
\dist(y,\Sigma^{N}) \stackrel{\eqref{e:dxsigman}}{\lec} \ve d_{\Up(N)}(y)\leq \ve d_{\Up(N)}(x_{R})+2\ve \ell(R)\lec \ve \ell(R)
\end{equation}
\preRef{this and the fact that the cubes in $\Stop(N)$ are disjoint  and Corollary \ref{c:up-positive-then-stop} imply} \postRef{Moreover, if $y\not\in R$ for any $R\in \Stop(N)$, then $d_{\Up(N)}(y)=0$ by Corollary \ref{c:up-positive-then-stop}, and \eqref{e:dxsigman}, we know $\dist(y,\Sigma^{N})=0$. Combined with the fact that the cubes in $\Stop(N)$ are disjoint,}
\preRef{\begin{equation}\label{e:I21}
I_{2} 
 \lec \sum_{Q\in S}\sum_{R\in \Stop(N) \atop R\cap 2C_{0}B_{Q}\neq\emptyset} \ve \frac{\ell(R)^{d+1}}{\ell(Q)} = \sum_{R\in \Stop(N) \atop R\cap 2C_{0}B_{Q(S)}\neq\emptyset} \sum_{Q\in S: 2C_{0}B_{Q}\cap R\neq\emptyset}  \ve \frac{\ell(R)^{d+1}}{\ell(Q)}  
\end{equation}}\postRef{and by Jensen's inequality (i.e. Lemma \ref{e:jensens}), we know $\beta_{E}^{d,p}\lec \beta_{E}^{d,2}$ for $p< 2$, and so we can assume without loss of generality that $p\geq 2$. Then 
\begin{equation}\label{e:I21}
I_{2} 
\lec  \ve \sum_{Q\in S} \ell(Q)^{d}\ps{\sum_{R\in \Stop(N) \atop R\cap 2C_{0}B_{Q}\neq\emptyset}  \frac{\ell(R)^{d+p}}{\ell(Q)^{d+p}}}^{\frac{2}{p}} 
\leq \ve \sum_{Q\in S} \sum_{R\in \Stop(N) \atop R\cap 2C_{0}B_{Q}\neq\emptyset}  \frac{\ell(R)^{d\frac{2}{p}+2}}{\ell(Q)^{d(\frac{2}{p}-1)+2}} 
\end{equation}
}
Next, for a fixed $R\in \Stop(N)$, if $R\cap 2C_{0}B_{Q}\neq\emptyset$ for some $Q\in S$, then $\dist(x_{Q},R) \leq 2C_{0}\ell(Q)$ and
\begin{equation}\label{e:R<4CQ}
\ell(R)\sim d_{\Up(N)}(R)\leq \ell(Q)+\dist(R,Q)
\leq \ell(Q)+2C_{0}\ell(Q)\leq 3C_{0}\ell(Q)
\end{equation}
hence, for each $k$, 
\begin{equation}
\#\{Q\in \cD_{k}\cap S: 2C_{0}B_{Q}\cap R\neq\emptyset\}\lec_{C_{0},n} 1
\label{e:2C0QR}
\end{equation}
and this is nonzero only if $\ell(Q)\gec_{C_{0}} \ell(R)$. Thus, by summing a geometric series,
\preRef{
\begin{equation}\label{e:I22}
 \sum_{R\in \Stop(N) \atop R\cap 2C_{0}B_{Q(S)}\neq\emptyset} \sum_{Q\in S: 2C_{0}B_{Q}\cap R\neq\emptyset}  \ve \frac{\ell(R)^{d+1}}{\ell(Q)}  \\
\lec \sum_{R\in \Stop(N) \atop R\cap 2C_{0}B_{Q(S)}\neq\emptyset}\ve \ell(R)^{d}
.\end{equation}
}
\postRef{and using the fact that $p<\frac{2d}{d-2}$ implies $d(2/p-1)+2>0$ (and that this holds for all $p\geq 1$ if $d=1,2$),
\begin{equation}\label{e:I22}
I_{2}
\lec 
\ve \sum_{R\in \Stop(N)\atop R\cap 2C_{0}B_{Q(S)}\neq\emptyset} \sum_{Q\in S \atop R\cap 2C_{0}B_{Q}\neq\emptyset}  \frac{\ell(R)^{d\frac{2}{p}+2}}{\ell(Q)^{d(\frac{2}{p}-1)+2}} 
\lec \ve \sum_{R\in \Stop(N)\atop R\cap 2C_{0}B_{Q(S)}\neq\emptyset}  \ell(R)^{d}
.\end{equation}

}

Now, suppose $R\in \Stop(N)$ is such that $R\cap 2C_{0}B_{Q(S)}\neq\emptyset$. Let $x_{R}'\in \Sigma^{N}$ be closest to $x_{R}$, hence 
\[
|x_{R}'-x_{R}|\stackrel{\eqref{e:dxsigman}}{\lec} \ve d_{\Up(N)}(x_{R})\sim \ve \ell(R)\stackrel{\eqn{R<4CQ}}{\lec} C_{0}\ve \ell(Q(S))\]
and so for $\ve$ small enough (in relation to $C_{0}$) and if $C_{2}>2$,
\[
B(x_{R}',\ell(R))\subseteq D(x_{Q(S)},P_{S},C_{2}\ell(Q(S)))\]
this, \Lemma{graph}, and the fact that the balls $\{C_{1}B_{R}:R\in \Stop(N)\}$ have bounded overlap by \Lemma{QSQS'} imply
\begin{align*}
 \sum_{R\in \Stop(N) \atop R\cap 2C_{0}B_{Q(S)}\neq\emptyset}\ve \ell(R)^{d}
&  \lec   \sum_{R\in \Stop(N) \atop R\cap 2C_{0}B_{Q(S)}\neq\emptyset} 
\ve\cH^{d}(B(x_{R}', \ell(R))\cap \Gamma_{S})\\
&  \lec \ve \cH^{d}(\Gamma_{S}\cap D(x_{Q},P_{S},C_{2}\ell(Q(S))))
 \lec \ve \ell(Q(S))^{d}.
 \end{align*}
This combined with \eqn{I21} and \eqn{I22} show that
\begin{equation}\label{e:I2end}
I_{2}\lec \ve \ell(Q(S))^{d}.
\end{equation}
We now turn to bounding     $I_{1}$.
Choose a family of dyadic cubes for $P_{S}$ and let $I_{Q}\ni \pi_{P_S}(x_{Q}) $ be a minimal  dyadic cubes for which  $\ell(I_{Q})\geq 8C_{0}\ell(Q)$. 
Note that
\[
|\pi_{P_S}(x_{Q})-\pi_{P_S}(x_{Q}')|
\leq |x_{Q}-x_{Q}'|\stackrel{\eqn{dxsigman}}{\lec}\ve d_{\Up(N)}(x_Q)\leq \ve \ell(Q)\]
and so for $\ve$ small enough,
\[
\pi_{P_S}(2C_{0}B_{Q}')\subseteq 3I_{Q}.\]

If $4C_{0}<C_{2}$ and $\ve>0$ is small, then $2C_{0} B_{Q}'\subseteq C_{2}B_{Q(S)}$, and so
\begin{align*}
\sum_{Q\in S} \beta_{\Sigma^{N}}^{d,p}(2C_{0}B_{Q}')^{2}\ell(Q)^{d}
& \lec_{d}
\sum_{Q\in S} \beta_{\Sigma^{N}}^{d,p}(2C_{0}B_{Q}')^{2}\ell(Q)^{d}\\
& \stackrel{\eqn{b-to-om}}{\lec_{d}}
\sum_{Q\in S} \Omega_{A_{S}}^{2}(3I_{Q})^{2}\ell(I_{Q})^{d}\\
& =\sum_{I\subseteq P_{S}}  \Omega_{A_{S}}^{2}(3I)^{2}\ell(I)^{d} \# \{Q\in S: I_{Q}=I\}.
\end{align*}

\begin{remark}\label{r:c2}
We now fix $C_{2}=1+4C_{0}$.
\end{remark}

In the second inequality, we used the fact that $A_{S}$ is $(1+C\alpha)$-Lipschitz. We will now show $ \# \{Q\in S: I_{Q}=I\}$ is bounded for each dyadic cube $I$. Suppose $Q,R\in S$ are such that $I_{Q}=I_{R}=I$. 
This implies $\ell(Q)\sim_n\ell(R)$.
Then by \Lemma{graph}, since $\pi_{S}$ is $(1+C\alpha)$-bi-Lipschitz on $D(x_{Q(S)},P_{S},C_{2}\ell(Q))$,
\begin{align*}
|x_{Q}-x_{R}|
& \leq |x_{Q}'-x_{R}'|+\dist(x_Q,\Sigma^{N})+\dist(x_R,\Sigma^{N})\\
& \lec |\pi_{S}(x_{Q}')-\pi_{S}(x_{R}')| +\dist(x_Q,\Sigma^{N})+\dist(x_R,\Sigma^{N})  \\
&  \leq |\pi_{S}(x_{Q})-\pi_{S}(x_{R})| + 2\dist(x_Q,\Sigma^{N})+2\dist(x_R,\Sigma^{N}) \\
& \stackrel{\eqn{dxsigman}}{\lec} \ell(I_{Q})+\ve d_{\Up(N)}(x_Q)+\ve d_{\Up(N)}(x_R)\\
& \lec \ell(I_{Q})+\ve \ell(Q)+\ve \ell(R)
\lec \ell(I).\end{align*}
Moreover, since $c_{0}B_{Q}\cap c_{0}B_{R}=\emptyset$ by \Theorem{Christ}, $|x_{Q}-x_{R}|\gec \ell(Q)\sim \ell(I)$. Thus, $|x_{Q}-x_{R}|\sim \ell(I)$ for any two $Q,R\in S$ for which $I_{Q}=I_{R}=I$, and this implies $ \# \{Q\in S: I_{Q}=I\}\lec_{n} 1$. Thus, by \Theorem{dorronsoro} and \Lemma{graph}

\begin{align*}
\sum_{Q\in S} \beta_{\Sigma^{N}}^{d,p}(2C_{0}B_{Q}')^{2}\ell(Q)^{d} 
& \lec  \sum_{I\subseteq P_{S}} \Omega_{A_{S}}^{2}(3I)^{2}\ell(I)^{d} 
\stackrel{\eqn{dorronsorocubes}}{\lec} ||\grad A_{S}||_{L^{2}(P_{S})}^{2}\\
& \lec \alpha^{2} \ell(Q(S))^{d}.
\end{align*}

Thus, if $R\in \cD_{0}$ and $R\cap B(0,1)\neq\emptyset$,
\begin{align*}
\sum_{Q\subseteq R} \beta_{\Sigma}^{d,p}(C_{0}B_{Q})^{2}\ell(Q)^{d}
& \lec \sum_{N\geq 0}\sum_{S\in \cF^{N}\atop Q(S)\subseteq R} \sum_{Q\in S} \beta_{\Sigma}^{d,p}(C_{0}B_{Q})^{2}\ell(Q)^{d}\\
& \lec \sum_{N\geq 0}\sum_{S\in \cF^{N}\atop Q(S)\subseteq R} (\ve+\alpha^{2}) \ell(Q(S))^{d}\\
& = (\ve+\alpha^{2}) \ps{ \ell(R)^{d}+\sum_{N\geq 0}\sum_{Q\in \Stop(N)} \ell(Q)^{d}}\\
& \stackrel{\eqn{asum}}{\lec} 1+\cH^{d}( \Sigma\cap B(0,1))\lec \cH^{d}( \Sigma\cap B(0,1)).
\end{align*}

If $R\in \cD_{0}$ and $R\cap B(0,1)=\emptyset$, then it is not hard to show that, since $\Sigma\backslash B(0,1)$ is a plane, 
\begin{align*}
\sum_{Q\subseteq R} \beta_{\Sigma}^{d,p}(C_{0}B_{Q})^{2}\ell(Q)^{d} 
& \leq \sum\{\ell(Q)^{d}: Q\subseteq R ,\;\; C_{0}B_{Q}\cap B(0,1)\neq\emptyset\}  \\
& \lec_{C_{0}} 1 \lec \cH^{d}( \Sigma\cap B(0,1)).
\end{align*}
Here we used the fact that, since $\Sigma\backslash B(0,1)$ is a plane, $ \beta_{E}^{d,p}(C_{0}B_{Q})\neq 0$ only when $C_{0}B_{Q}\cap B(0,1)\neq\emptyset$. For the same reason, there are at most boundedly many $R\in \cD_{0}$ for which $R\cap B(0,1)=\emptyset$ and $\sum_{Q\subseteq R} \beta_{E}^{d,p}(C_{0}B_{Q})^{2}\ell(Q)^{d}\neq 0$ (with constant depending on $C_{0}$) . Combining this with the sum in $Q_{0}$ above, we get that 

\begin{align*}
\sum_{k\geq 0} \sum_{x\in \cX_{k}} \beta_{\Sigma}^{d,p}(B(x,C_{0}\rho^{-k}))^{2}\rho^{-kd} 
& \lec_{d}\sum_{R\in \cD_{0}} \sum_{Q\subseteq R}  \beta_{\Sigma}^{d,p}(C_{0}B_{Q})^{2}\ell(Q)^{d}\\
&  \lec_{C_{0}} \cH^{d}( \Sigma\cap B(0,1)).
\end{align*}

This completes the proof of Theorem \ref{t:flat-case}.

%
%

\section{Proof of Theorem \ref{t:thmii}}
\label{s:TheoremII}

We first collect some notation and lemmas, which will be quite similar to that found in previous sections.

Let $E$ satisfy the conditions of Theorem \ref{t:thmii}. {Let $\cD$ be the cubes from \Theorem{Christ} for $E\cap B(0,1)$ with $\rho^{k}$-nets such that $\cX_{0}=\{0\}$. In this way, $\cD_{0}=\{Q_{0}\}$ where 
$Q_{0}=B(0,1)\cap E$ 
and all cubes in $\cD$ are contained in and partition $E\cap B(0,1)$}. Let

\[
\cG=\{Q\in \cD: \vartheta_{E}(AB_{Q})<\ve\}.\]

For a cube $Q\in \cG$, we will construct an approximating surface for $E$ relative to $Q$ called $\Sigma_{Q}$. 

Define a stopping-time region $S(Q)$ by inductively adding cubes $R$ to $S(Q)$ if either $R=Q$  or 
\begin{enumerate}
\item $R^{(1)}\in S(Q)$
\item $\vartheta_{E}(AB_{R'})<\ve$ for each sibling $R'$ of $R$, including $R$ itself.
\end{enumerate}

For $x\in \bR^{n}$, define
\[
d_{Q}(x)=d_{S(Q)}(x) \;\; \mbox{ and }\;\; d_{Q}(R)=\inf_{x\in R}d_{Q}(x).\]
%

Let $C_{1}>10$. We will adjust its value as we go along but we will assume that $10<C_{1}\ll A$.

Again, for $k\geq 0$ an integer, let $s(k)$ be such that $5\rho^{s(k)}\leq r_{k}< 5\rho^{s(k)-1}$. We can assume that $\rho$ is $5$ times a power of $10$ and so there is $k_{Q}$ so that $5\rho^{s(k_{Q})}=r_{k_{Q}}$. Set
\[
S(Q)_{k}=\cD_{s(k)}\cap S(Q)\]
and let $\cX^{Q}_{k}=\{x_{j,k}\}_{j\in J_{k}^{Q}}$ be a maximal $r_{k}$-separated set of points for the set
\[\cC_{k}^{Q}=\{x_{Q}:Q\in S(Q)_{k}\}.\]
For $j\in J_{k}^{Q}$, let $Q_{j,k}\in S(Q)_{k}$ be such that $x_{Q_{j,k}}=x_{j,k}$. In this way, $\cX^{Q}_{k_{Q}}=\{x_{Q}\}$.  For each $R\in \cG$, there is $P_{R}'$ so that
\[
\vartheta_{E}(AB_{R},P_{R}')<\ve.\]
Let $P_{R}$ be the translate of $P_{R}'$ so that $x_{R}\in P_{R}$, then it is not hard to show that, for $\ve>0$ small enough,
\begin{equation}\label{e:9thetaabr}
\vartheta_{E}(AB_{R},P_{R})\lec \ve.
\end{equation}
Let $P_{j,k}=P_{AB_{Q_{j,k}}}$. 

The following has the same proof as \Lemma{DTready}.

\begin{lemma}
For each $Q$ and $A>10^{5}$, $\{x_{j,k}\}_{j\in J_{k}^{Q}}$ satisfies the conditions of \Theorem{DT}.
\end{lemma}

Let $P_{Q}$ be defined as before and let $\sigma_{k}^{Q}$, $\Sigma_{k}^{Q}$, $\Sigma^{Q}$, and $\sigma^{Q}$ be the functions and surfaces obtained from \Theorem{DT} with the nets $\cX^{Q}_{k}$. 

The following has a similar proof as \Lemma{up-zero-intersection}.
\begin{lemma}\label{l:9up-zero-intersection}
If $d_{Q}(x)=0$ then $x\in \Sigma^Q\cap \Sigma$
 \end{lemma}

For $x\in \Sigma^{Q}$, we define $k_{Q}(x)$ to be the maximal integer $k$ such that $x\in V_{k-1}^{11}$.

\begin{lemma}\label{l:9lemma-5-4}
For $x\in \Sigma^{Q}\cap C_{1}B_{Q}$, let $k=k_{Q}(x)$ . Then we have 
\begin{equation}\label{e:9dx<dT}
\dist(x,E)\lec   \ve r_{k} \sim \ve d_{Q}(x)
\end{equation}
and there is a twice-differentiable $C\ve$-Lipschitz graph $\Gamma_{x}^{Q}$ over $P_{x,r_{k}}^{N}$ so that
\begin{equation}
B(x,r_{k})\cap \Sigma^{Q}=B(x,r_{k})\cap \Sigma_{k}^{Q}=B(x,r_{k})\cap \Gamma_{x}^{Q}
\label{e:9bsnbsnk}
\end{equation}
\end{lemma}

\begin{proof}
We first prove the left-hand-side inequality \eqn{9dx<dT}. There is nothing to show if $x\in E$, so assume $x\not\in E$. 

Suppose first that $x\not\in 10B_{Q}$. Let $k=k_{Q}(x)$. Note that $\Sigma^{Q}$ is $C\ve$-flat for some constant $C$, thus there is a $d$-plane $P$ for which $d_{C_{1}B_{Q}}(\Sigma^{Q},P)<C\ve$. As $\vartheta_{E}(C_{1}B_{Q})<\ve$, we then have $d_{C_{1}B_{Q}}(P,P_{Q})\lec \ve$, and since $C_{1}>10$, we have that $x\in C_{1}B_{Q}$ and thus 
\[
\dist(x,E)\lec_{C_{1}} \ve \ell(Q).\]
In this case, $r_{k}\sim_{C_{1}} \ell(Q)$, and we are done.

Now suppose $x\in 10 B_{Q}$. 
Again, set $k=k_{Q}(x)$. Let $x'\in \Sigma_{k-1}^{Q}$ be such that $x=\lim_{K\rightarrow\infty}\sigma_{k+K}^{Q}\circ\cdots\circ \sigma_{k-1}^{Q}(x')$. Then by \eqn{ytosigma},
\[|x'-x|\lec \ve r_{k-1}\lec \ve r_{k}.\]
Hence, for $\ve$ small enough, $x'\in V_{k-1}^{12}$. 
Thus, there is $j\in J_{k-1}^{Q}$ so that $x'\in 12B_{j,k-1}$. By \eqn{49r}, $\dist(x',P_{j,k-1})\lec\ve r_{k}$ and $\pi_{j,k-1}(x')\in 13B_{j,k-1}$. By our choice of $P_{j,k-1}$, $\dist(\pi_{j,k-1}(x'),E)\lec \ve r_{k-1}$. Combining these inequalities, we get that $\dist(x,E)\lec \ve r_{k}$. 
It now remains to show the right-hand-side inequality \eqn{9dx<dT}, i.e. that $ r_{k}\sim d_{Q}(x)$. First, 
\begin{align*}
d_{Q}(x) & \leq \ell(Q_{j,k-1})+\dist(x,Q_{j,k-1})\lec r_{k-1}+|x-x'|+\dist(x',Q_{j,k-1})
\\
& \leq r_{k-1}+\ve r_{k-1}+|x'-x_{j,k-1}|\lec r_{k-1}\lec r_{k}.
\end{align*}
Next, let $Q\in S(Q)$ be such that $d_{Q}(x)\sim \ell(Q)+\dist(x,Q)$. Suppose $d_{Q}(x)\sim r_{\ell}$ for some $\ell\geq k$. Since $|x-x_{Q}|\lec d_{Q}$ we know $|x-x_{Q}|\leq  Cr_{\ell}$ for some universal constant $C$, and if
 $\ell - k\gtrsim\log(C)$, $|x-x_{Q}|<r_{k}$. Then $x_{Q}\in B_{i\ell}$ for some $i\in J_{\ell}^Q$ (since $\cX^{Q}_{\ell}$ is a maximal net), so in particular, $x_{Q}\in V_{\ell}^{1}$. Thus, $x_{Q}\in V_{k}^{2}$, and hence $x\in V_{k}^{11}$, contradicting our choice of $k$. Thus, $|\ell-k|$ is bounded by a universal constant, implying $d_{Q}(x)\gec r_{k}$, we have \eqn{9dx<dT}.

To get \eqn{9bsnbsnk}, notice that since $x\not\in V_{k}^{11}$ and by the maximality of $k$, $B(x,r_{k})\subseteq (V_{\ell}^{10})^{c}$ for all $\ell\geq k$, and so $\sigma_{\ell}^{Q}$ is the identity on $B(x,r_{k})$ for all $\ell\geq k$. This and \eqn{ygraph} imply \eqn{9bsnbsnk}.
\end{proof}

%
%

%
%
%
%

\begin{lemma}\label{l:sizesigmaq}
For $Q\in \cG$, 
\begin{equation}
\label{e:sizesigmaq}
\cH^{d}( C_{1}B_{Q}\cap \Sigma^{Q})
 \lec \sum_{R\in m(S(Q))} \ell(R)^{d} + \cH^{d}(\{x\in \cnj{Q}:d_{Q}(x)=0\}).
\end{equation}
\end{lemma}

\begin{proof}
Let $\delta>0$ (to be decided later) and 
\begin{equation}\label{e:e1}
E_{1}=\{x\in \Sigma^{Q}\cap C_{1}B_{Q}: d_{Q}(x)>0 \mbox{ and }\dist(x,Q)<\delta r_{k_{Q}(x)}\}\end{equation}
By Besicovitch's covering theorem and \Lemma{9lemma-5-4}, we may find balls $\{B_{j}=B(x_{j},r_{k_{Q}(x_{j})})\}_{j\in J}$ of bounded overlap with centers in $E_{1}$ so that 
\[
E_{1}\subseteq \bigcup_{j\in J} B_{j}.\]
 Let $k_{j}=k_{Q}(x_{j})$ and $x_{j}'\in E$  be closest to $x_{j}$. Again, it's not hard to show using \eqn{9dx<dT} that
\begin{equation}
\label{e:9compd}
d_{Q}(x_{j})\sim d_{Q}(x_{j}')
\end{equation}

 For $j\in J$, let $z_{j}\in Q$ be a closest point to $x_{j}$ so that
\begin{equation}\label{e:zjxj}
\dist(x_{j},Q)=|x_{j}-z_{j}|\stackrel{\eqn{e1}}{<}\delta r_{k_{j}}.
\end{equation} 
 
  Then for any cube $R\in S(Q)$ containing $z_{j}$, since $x_{j}\in E_{1}$,
\[
\ell(R)+\delta r_{k_{j}}
\stackrel{\eqn{e1}}{>}\ell(R)+\dist(x_{j},Q)
\geq\ell(R)+\dist(x_{j},R)
\geq d_{Q}(x_{j})\stackrel{\eqn{9dx<dT}}{\sim} r_{k_{j}}\]
and so for $\delta$ small, 
\[
\ell(R)\gec r_{k_{j}},\] 
thus there is $Q_{j}\in m(S(Q))$ containing $z_{j}$ with $\ell(Q_{j})\gec r_{k_{j}}$. Moreover,
\[
r_{k_{j}} \stackrel{\eqn{9dx<dT}}{\sim} d_{Q}(x_{j})
\leq |x_{j}-z_{j}|+d_{Q}(z_{j})
\stackrel{\eqn{zjxj}}{<} \delta r_{k_{j}}+\ell(Q_{j})\]
which implies for $\delta>0$ small enough that 
\begin{equation}
\label{e:rkjsimqj}
r_{k_{j}}\sim \ell(Q_{j}). 
\end{equation}
For $R\in m(S(Q))$, consider
\[
\cC_{R}=\{j\in J: Q_{j}=R\}.\]
We claim that 
\begin{equation}\label{e:cr<1}
\# \cC_{R}\lec 1.
\end{equation}
We have already seen that $\ell(Q_{j})\sim r_{k_{j}}$, and thus $r_{k_{j}}\sim \ell(R)$ for all $j\in \cC_{R}$. Moreover,
\[
\dist(x_{j},R)
=\dist(x_{j},Q_{j})
\leq |x_{j}-z_{j}| 
\stackrel{\eqn{zjxj}}{<} \delta r_{k_{j}}\sim \delta \ell(R).\]
Since the balls $\{B_{j}\}_{j\in \cC_{R}}$ have bounded overlap, these facts imply \eqn{cr<1}, which proves the claim.

These facts and \eqn{9bsnbsnk} imply that
\begin{align}\label{e:sumrkj1}
\cH^{d}(E_{1})
& \leq \sum_{j\in J} \cH^{d}(E_{1}\cap B_{j})
\stackrel{\eqn{9bsnbsnk}}{=}\sum_{j\in J} \cH^{d}(\Sigma^{Q}\cap B_{j})
\stackrel{\eqn{9bsnbsnk}}{\lec} \sum_{j\in J}r_{k_{j}}^{d}  \notag \\
& 
 \leq \sum_{j\in J} \ell(Q_{j})^{d}
= \sum_{R\in m(S(Q))}\sum_{j\in \cC_{R}} \ell(R)^{d} 
\stackrel{\eqn{cr<1}}{\lec}  \sum_{R\in m(S(Q))}\ell(R)^{d}.
\end{align}
 
Now set 
\begin{equation}\label{e:e2}
E_{2}=\{x\in \Sigma^{Q}\cap C_{1}\ell(Q): d_{Q}(x)>0 \mbox{ and }\dist(x,Q)\geq \delta r_{k_{Q}(x)}\}.\end{equation}

Again by Besicovitch's theorem, we may find a collection of balls $\{B_{j} = B(x_{j},r_{k_{Q}(x_j)})\}_{j\in I}$ with centers in $E_{2}$ of bounded overlap.  Then there is $Q_{j}\in S(Q)$ so that $\ell(Q_{j})+\dist(x_{j},Q_{j})<2d_{Q}(x_{j})$. Let $R_{j}$ be the maximal ancestor of $Q_{j}$ so that $\ell(R_{j})+\dist(x_{j},R_{j})<2d_{Q}(x_{j})$, then 
\begin{equation}
\label{e:rjxj}
\ell(R_{j})\sim d_{Q}(x_{j})\sim r_{k_{j}} \;\; \mbox{and}\;\; \dist(x_{j},R_{j})<2d_{Q}(x_{j}).
\end{equation}
 Let 
 \[
 B_{j}'=\frac{c_{0}}{2}\rho B_{R_{j}}.\] 
 We claim that the balls $\{B_{j}'\}_{j\in I}$ have bounded overlap on $\Sigma^{Q}$. Let $x\in \Sigma^{Q}$ and 
 \[
 I(x)=\{j\in I:x\in B_{j}'\}.\] 
 Let $x'\in E$ be closest to $x$, so for any $j\in I(x)$,
\begin{align}
|x-x'|
& =\dist(x,E)
\stackrel{\eqn{9dx<dT}}{\lec}  \ve d_{Q}(x)
\leq \ve(\dist(x,R_{j})+\ell(R_{j})) \notag \\
& <\ve(r_{B_{j}'}+\ell(R_{j}))
\lec   \ve \ell(R_{j}).
\label{e:x-x'Rj}
\end{align}
Also, if $x_{j}'\in E$ is closest to $x_{j}$, then 
\begin{equation}\label{e:xj-xj'Rj}
|x_{j}-x_{j}'|
=\dist(x_{j},E)
\stackrel{\eqn{9dx<dT}}{\lec} \ve r_{k_{j}} \stackrel{\eqn{rjxj}}{\sim} \ell(R_{j})
\end{equation}
and so for $\ve>0$ small enough depending on $\delta$.
\begin{equation}\label{e:xj'far}
\dist(x_{j}',Q)
\geq \dist(x_{j},Q)-|x_{j}-x_{j}'|
\stackrel{\eqn{e2} \atop \eqn{xj-xj'Rj}}{\geq} \delta r_{k_{j}}-C\ve r_{k_{j}}\gec \delta r_{k_{j}}.
\end{equation}

Suppose $i,j\in I(x)$. Note that since $x\in B_{j}'\cap \Sigma^{Q}$, 
\[
\dist(x,E)\stackrel{\eqn{9dx<dT}}{\lec} \ve d_{Q}(x)
\lec \ve \ell(R_{j})\]
For $\ve>0$ small enough, the closest point to $x$ in $E$ must be in $R_{j}$, that is, $x'\in R_{j}$. Then as 
\begin{equation}\label{e:bj'inQ}
B_{j}'\cap E\subseteq c_{0} B_{Q}\cap E\subseteq Q
\end{equation}
and $x_{i}'\in E\backslash Q$ by \eqn{xj'far}, we have
\begin{align*}
\frac{c_{0}}{2}\ell(R_{j})
& \leq \dist(x_{i}',R_{j})
< \dist(x_{i}',B_{j}'\cap E)
\leq |x_{i}'-x|\\
& \leq |x_{i}'-x'|+|x'-x|
\stackrel{\eqn{x-x'Rj}\atop x'\in R_{j}}{\lec} \dist(x_{i}',R_{i})+\ell(R_{i})
\stackrel{\eqn{rjxj}}{\lec} \ell(R_{i}).
\end{align*}
Interchanging the roles of $i$ and $j$, we see that $\ell(R_{i})\sim \ell(R_{j})$ for all $i,j\in I(x)$. Let $r(x)=\sup_{j\in I(x)} \ell(R_{j})$. Then for each $j\in I(x)$, $\ell(R_{j})\sim r(x)$ and
\[
\dist(x,R_{j}) 
\leq \dist(B_{j}',R_{j})+2r_{B_{j}'}
\lec 0+\ell(R_{j})\leq r(x)\]
and these facts imply that $\# I(x)\lec 1$, and thus the balls $\{B_{j}'\}_{j\in I}$  have bounded overlap on $\Sigma^{Q}$ and proves the claim.

Thus, since $\Sigma^{Q}$ is lower regular by virtue of being Reifenberg flat (see \eqn{sigmalr}), this bounded overlap implies
\begin{equation}\label{e:sumrkj}
\sum_{j\in I} r_{k_{j}}^{d}
\lec \sum_{j\in I} \ell(R_{j})^{d}
\lec \sum_{j\in I} \cH^{d}(B_{j}'\cap \Sigma^{Q})
\lec \cH^{d}\ps{\bigcup_{j} B_{j}'\cap \Sigma^{Q}}
\end{equation}
Now, recall that if $x\in B_{j}'\cap \Sigma^{Q}$, then $x'\in R_{j}$, and hence in $Q$, thus 
\[
\dist(x,E)=\dist(x,Q)\lec \ve d_{Q}(x)\sim \ve r_{k_{Q}(x)}\]
and so for $\ve>0$ small enough, this implies $\dist(x,E)<\delta r_{k_{Q}(x)}$, and so either $x\in E_{1}$ or $x\in E_{0}$ where
\[
E_{0}=\{x\in C_{1}B_{Q}\cap E:d_{Q}(x)=0\}.\] 
We then have
\begin{equation}\label{e:bj'inE1}
\bigcup_{j} B_{j}'\cap \Sigma^{Q}\subseteq E_{0}\cup E_{1}
\end{equation}
Now we estimate $\cH^{d}(E_{2})$:
\begin{align*}
\cH^{d}(E_{2})
& \leq \sum \cH^{d}(E_{2}\cap B_{j})
=\sum \cH^{d}(\Sigma^{Q}\cap B_{j})
\stackrel{ \eqn{9bsnbsnk}}{\lec} \sum_{j\in I} r_{k_{j}}^{d}\\
& \stackrel{\eqn{sumrkj}}{\lec} \cH^{d}\ps{\bigcup_{j} B_{j}'\cap \Sigma^{Q}}
\stackrel{\eqn{bj'inE1}}{ \leq} \cH^{d}(E_{0}\cup E_{1})\\
& \stackrel{\eqn{sumrkj1}}{\lec} \sum_{R\in m(S(Q))} \ell(R)^{d}+\cH^{d}(E_{0})
\end{align*}

Thus, combining this with \eqn{sumrkj1}, we have
\begin{align*}
\cH^{d}(\{x\in C_{1}Q\cap \Sigma^{Q}: d_{Q}(x)>0\})
& \leq \cH^{d}(E_{0})+\cH^{d}(E_{1})+\cH^{d}(E_{2}) \\
& \lec \cH^{d}(E_{0})+\sum_{R\in m(S(Q))} \ell(R)^{d}.
\end{align*}

 Note that we actually have $E_{0}\subseteq \cnj{Q}$ by definition. This observation and the above inequality finish the lemma.

\end{proof}

\begin{lemma}\label{l:zeroboundary}
For $Q \subseteq Q_{0}$, $\cH^{d}(\d Q)=0$, where 
\[
\d Q=\{x\in  \cnj{Q}: B(x,r)\cap E\backslash Q\neq\emptyset \mbox{ for all }r>0\}.\]
\end{lemma}

\begin{proof}
{%
Let $x\in \d Q$ and $R$ be any cube containing $x$ with $\ell(R)\leq \ell(Q)$. Then $c_{0} B_{R}\cap Q=\emptyset$. In particular, $\frac{c_{0}}{2} B_{R}\cap \cnj{Q}=\emptyset$. Thus,
\[
\cH^{d}(E\cap B(x,2\ell(R))\backslash \d Q)\geq \cH^{d}\ps{\frac{c_{0}}{2} B_{R}\cap E}\gec \ell(R)^{d}.\]
Since there are such cubes $R$ of arbitrarily small size,
\[
\limsup_{r\rightarrow 0} \frac{\cH^{d}(E\cap B(x,r)\backslash \d Q)}{r^{d}}>0.\]
However, by \cite[Corollary 2.14]{Mattila}, this limit must be zero for $\cH^{d}$-a.e. $x\in \d Q$, hence $\cH^{d}(\d Q)=0$.}
\end{proof}

Let
\[
\cB= \cD\backslash \cG = \{Q\subseteq Q_{0}: \vartheta_{E}^{d}(AB_{Q})\geq \ve \}.\]
Let $Q_{j}$ be a collection of maximal cubes in $\cG$. Note that since $\largetheta{E}^{d,A,\ve}(B(0,1))<\infty$, $Q_{j}$ covers 
almost all of $Q_{0}$ (up to $\cH^d$ measure zero). 
Let $\cF_{0}=\{S(Q_{j})\}$. Assume we have defined $\cF_{N}$ and let $S\in \cF_{N}$. Let $m'(S)$ be the collection of maximal cubes in $\cG$ contained in cubes in $m(S)$. Let
\[
\cF_{N+1}=\bigcup_{S\in \cF_{N}} \{S(Q): Q\in m'(S)\} \;\; \mbox{and} \;\; \cF=\bigcup \cF_{N}.\]

For $S\in \cF$, write $\Sigma^{S}=\Sigma^{Q(S)}$.

\begin{lemma}
\begin{equation}\label{e:summinc}
\sum_{S\in \cF} \sum_{R\in m(S)\cap \cB} \ell(R)^{d}
\lec \largetheta{E}^{d,A,\ve}(0,1)
\end{equation}
and
\begin{equation}\label{e:sigmaq<q+c}
\sum_{S\in \cF} \cH^{d}(C_{1}B_{Q(S)}\cap \Sigma^{S})
\lec \cH^{d}(E\cap B(0,1))+\largetheta{E}^{d,A,\ve}(0,1).
\end{equation}
\end{lemma}
\begin{proof}
Write

\[
d(S)=\{x\in \cnj{Q (S)}: d_{Q(S)}(x)=0\}.\]

 Then
\begin{align*}
\sum_{S\in \cF} \cH^{d}(C_{1} B_{Q(S)}\cap \Sigma^{S})
& \stackrel{\eqn{sizesigmaq}}{\lec}
\sum_{S\in \cF} \sum_{R\in m(S)} \ell(R)^{d} + \sum_{S\in \cF}\cH^{d}(d(S))\\
& = I_{1}+I_{2}.
\end{align*}
To bound the first term,  note that for $S,S'\in \cF$, if $S\cap S'\neq\emptyset$, then either $Q(S)\in m(S')$ or $Q(S')\in m(S)$ by construction.   

Also observe also that if $Q\in m(S)$ for some $S\in \cF$, then one of its siblings is in $\cB$,  and so by definition of $\largetheta{E}^{d,A,\ve}(0,1)$,
\[
I_{1} 
 \lec\sum_{S\in \cF} \sum_{R\in m(S)\cap \cB} \ell(R)^{d}
\leq \sum_{Q\in \cB} \ell(Q)^{d}
 \leq \largetheta{E}^{d,A,\ve}(0,1).
\]
Note that this proves \eqn{summinc}.
Now we bound the second term. Let $Q^{\circ}=Q\backslash \d Q$ and note that if $S,S'\in \cF$ and $Q(S)\subseteq T\in m(S')$, then we have $d(S)\subseteq \cnj{Q(S)}$. Hence 
\[d(S)\cap Q(S)^{\circ}\subseteq T^{\circ}\subseteq Q(S')\backslash d(S'),\]
and thus, for $S,S'\in \cF$, by \Lemma{zeroboundary}, $\cH^{d}(d(S)\cap d(S'))=0$.


 Hence,
\[
I_{2}=\sum_{S\in \cF}\cH^{d}(d(S))\leq \cH^{d}(Q_{0})= E\cap B(0,1).\]
\end{proof}

\begin{proposition}[Theorem A Part 1] \label{p:thmapt1}
With the assumptions of Theorem \ref{t:thmii},
\begin{multline}
\sum_{k\geq 0}\sum_{B \in \cX_{k}} \beta_{E}^{d,p}(C_{0}B)^{2}r_{B}^{d}\\
\lec_{A,C_{0},n,\ve,p}
\ps{\cH^{d}(E\cap B(0,1)) + \largetheta{E}^{d,A,\ve}(0,1)}.
\label{e:beta<hd2-part1}
\end{multline}
\end{proposition}

\begin{proof}
Let $1\leq p < \frac{2d}{2-d}$ if $d>2$ and $1\leq p<\infty$ if $1\leq d\leq 2$. By Jensen's Inequality, we can assume $p\geq 2$. Note that 

\[
\sum_{Q\subseteq Q_{0}} \beta_{E}^{d,p}(C_{0}B_{Q})^{2}\ell(Q)^{d}
 \lec \largetheta{E}^{d,A,\ve}(B_{Q_{0}})
+\sum_{S\in \cF}\sum_{Q\in S} \beta_{E}^{d,p}(C_{0} B_{Q})^{2}\ell(Q)^{d}\]
and so we just need to bound the second sum. For $S\in \cF$, set
\[
\Sigma^{S}=\Sigma^{Q(S)}.\]
If $Q\in S$, we will let $x_{Q}'$ denote the closest point in $\Sigma^{S}$ to $x_{Q}$ and $B_{Q}'=B(x_{Q}',\ell(Q))$. Then
\begin{align*}
\sum_{S\in \cF} & \sum_{Q\in S} \beta_{E}^{d,p}(C_{0} B_{Q})^{2}\ell(Q)^{d}
 \stackrel{\eqn{pushbeta}}{\lec} \sum_{S\in \cF} \sum_{Q\in S}\beta_{\Sigma^{S}}^{d,p}(2C_{0}B_{Q}')^{2}\ell(Q)^{d}\\
& \qquad +\sum_{S\in \cF} \ps{\sum_{Q\in S}\frac{1}{\ell(Q)^{d}}\int_{2C_{0}B_{Q}\cap E}\ps{\frac{\dist(x,\Sigma^{S})}{\ell(Q)}}^{p}d\cH^{d}_{\infty}(x)}^{\frac{2}{p}}\ell(Q)^{d}\\
& = I_{1}+I_{2}.
\end{align*}

We first bound the second term. Let $Q\in S$. Let $k$ be so that $\ell(Q)\leq r_{k}<\rho^{-1}\ell(Q)$, and let $j\in J_{k}$ be so that $x_{Q}\in B_{jk}$. By \eqn{49r}, since $x_{jk}\in P_{jk}$, $\dist(x_{jk},\Sigma_{k}^{S})\lec \ve r_{k}$. If $x_{jk}'\in \Sigma_{k}^{S}$ is closest to $x_{jk}$, then this means $x_{jk}'\in B_{jk}$. By \eqn{sky-y}, $\dist(x_{jk}',\Sigma^{S})\lec \ve r_{k}$, and so $\dist(x_{jk},\Sigma^{S})\lec \ve r_{k}$. Thus, $\dist(x_{Q},\Sigma^{S})\lec r_{k}\sim \ell(Q)$, and so if $x\in E\cap 2C_{0}B_{Q}$, then $\dist(x,\Sigma^{S})\lec C_{0}\ell(Q)$.\\

 {
Let $B_{j}$ be a Besicovitch subcover of the collection 
$$\{B(x,2d(x,\Sigma^S)):x\in 2C_{0}B_{Q(S)}\cap E\}$$ and write $B_{j}=B(x_{j},s_{j})$.
Hence,  if $x_{j}'\in \Sigma^{S}$ is closest to $x_{j}$, then 
\[
s_{j}^{d}\lec  \cH^d(B(x_j', s_j/2)\cap \Sigma^S)\leq \cH^{d}(B_{j}\cap \Sigma^S)\]

If $Q\in S$ is such that $2C_{0}B_{Q}\cap B_{j}\neq\emptyset$, then $s_{j}=d_{S}(x_{j})\lec \ell(Q)$, and so, recalling that $p<2d/(2-d)$ (so $d-2(d/p+1)<0$),
\begin{multline*}
\sum_{Q\in S}
\ps{\ell(Q)^{-d}\int_{2C_{0} B_{Q}\cap E}\frac{ \dist(x,\Sigma^{S})}{\ell(Q)}d\cH^{d}_{\infty}(x)}^{\frac{2}{p}}\ell(Q)^{d}\\
  \stackrel{\eqref{e:intineq}}{\lec}
  \sum_{Q\in S}
\ps{\sum_{j}\ell(Q)^{-d} \int_{2C_{0} B_{Q}\cap E\cap B_{j}}\ps{\frac{ \dist(x,\Sigma^{S})}{\ell(Q)}}^{p}d\cH^{d}_{\infty}(x)}^{\frac{2}{p}}\ell(Q)^{d}
\\
\lec   \sum_{Q\in S}
\ps{\sum_{j} \ell(Q)^{-d}\int_{2C_{0} B_{Q}\cap E\cap B_{j}}\ps{\frac{s_{j}}{\ell(Q)}}^{p}d\cH^{d}_{\infty}(x)}^{\frac{2}{p}}\ell(Q)^{d}\\
\lec   \sum_{Q\in S}
\ps{\sum_{2C_{0} B_{Q}\cap E\cap B_{j} \neq\emptyset} \frac{s_{j}^{d+p}}{\ell(Q)^{d+p}}}^{\frac{2}{p}}\ell(Q)^{d}\\
\leq \sum_{Q\in S}
\sum_{2C_{0} B_{Q}\cap E\cap B_{j} \neq\emptyset} \frac{s_{j}^{2(d/p+1)}}{\ell(Q)^{2(d/p+1)}}\ell(Q)^{d}\\
\sum_{j} s_{j}^{2(d/p+1)}\sum_{Q\in S\atop 2C_{0} B_{Q}\cap E\cap B_{j} \neq\emptyset} \ell(Q)^{d-2(d/p+1)}
\lec \sum_{j} s_{j}^{d} \\
\sum_{j} \cH^{d}(B_{j}\cap \Sigma^{S})
\leq \cH^{d}(C_{1}B_{Q(S)}\cap \Sigma^{S}).
\end{multline*}
Thus,
\[
I_{2}
   \lec \sum_{S\in \cF}\cH^{d}(C_{1}B_{Q(S)}\cap \Sigma^{S})
 \stackrel{\eqn{sigmaq<q+c}}{\lec}
 \cH^{d}(E\cap B(0,1))+\largetheta{E}^{d,A,\ve}(0,1).
\]
}

To bound $I_{1}$, fix $S\in \cF$ and suppose $Q(S)\in \cD_{N}$. For $n\geq N$, let $\cX_{n}$ be a collection of maximally separated $\rho^{n}$ nets for $\Sigma^{Q}$. Then for $n\geq N$ and $Q\in S\cap \cD_{n}$, there is $x^{Q}\in X_{n}$ so that $x_{Q}'\in B^{Q}:=B(x^{Q},\rho^{n})$, and so $2C_{0}B_{Q}'\subseteq 20C_{0}B^{Q}$. Moreover, since the centers of cubes in $\cD_{n}$ are maximally $\rho^{n}$-separated, we know that for any $x\in \cX_{n}$ that $\#\{Q\in S\cap \cD_{n}:x^{Q}=x\}\lec 1$. Thus, if $20 C_{0}\ll C_{1}$, Theorem \ref{t:flat-case} implies
\begin{align*}
\sum_{Q\in S}\beta_{\Sigma^{S}}^{d,p}(2C_{0}B_{Q}')^{2}\ell(Q)^{d}
& \lec  \sum_{n\geq N}\sum_{x\in \cX_{n}} \beta_{\Sigma^{S}}^{d,p}(B(x,20C_{0}\rho^{n})^{2}\rho^{nd} \\
& \lec \cH^{d}(C_{1}B_{Q(S)}\cap \Sigma^{S}).
\end{align*}
Thus,
\[
I_{1}
\lec \sum_{S\in \cF} \cH^{d}(C_{1}B_{Q(S)}\cap \Sigma^{S})
\stackrel{\eqn{sigmaq<q+c}}{\lec}
 \cH^{d}(Q_{0})+\largetheta{E}^{d,A,\ve}(B_{Q_{0}}).\]
 
 Combining the estimates for $I_{1}$ and $I_{2}$ gives \eqn{beta<hd2-part1}.

\end{proof}

\begin{proposition}[Theorem A Part 2] 
\label{p:diamEC}
Let $E\subseteq \bR^{d}$ be closed. Then
\begin{equation}
\label{e:diamEC}
1 \lec_{d} \cH^{d}(E\cap B(0,1))+\largetheta{E}^{d,A,\ve}(0,1).
\end{equation}
\end{proposition}

\begin{proof}
If $Q_{0}\not\in \cG$, then $\largetheta{E}^{d,A,\ve}(0,1)\geq 1$ and there is nothing to show, so we may assume $Q_{0}\in \cG$. 

Suppose that $\largetheta{E}^{d,A,\ve}(0,1)<\delta \ell(Q_{0})^{d}$, where $\delta>0$ will be decided shortly. Let $x\in (\Sigma^{Q_{0}}\backslash E)\cap \frac{c_{0}}{2} B_{Q_{0}}$. Let $x'\in E$ be closest to $x$, so \eqn{9dx<dT} implies 
\begin{equation}\label{e:x-x'<dqx}
|x-x'|\lec \ve d_{Q_{0}}(x).
\end{equation}
For $\ve>0$ small enough, this implies $x'\in c_{0}B_{Q_{0}}\cap E\subseteq Q_{0}$. Hence, there is $R\in m(S(Q_{0}))$ that contains $x'$ and $\ell(R)\geq d_{Q_{0}}(x)$ by definition of $d_{Q_{0}}$. Since $R\in m(S(Q_{0}))$, it has a sibling $R'$ for which $\vartheta_{E}(AB_{R'})\geq \ve$. Since $\ell(R)=\ell(R')$ and $|x_{R}-x_{R'}|\leq \frac{2}{\rho}\ell(R)$, we have that
\begin{align*} |x_{R'}-x|
& \leq |x_{R'}-x_{R}|+|x_{R}-x'|+|x'-x|
<\frac{2}{\rho}\ell(R)+\ell(R')+\ve d_{Q_{0}}(x)\\
& <\ps{\frac{2}{\rho}+1+\ve}\ell(R')
\leq \frac{4}{\rho} \ell(R').\end{align*}
Thus,
\[
 (\Sigma^{Q_{0}}\backslash E)\cap \frac{c_{0}}{2} B_{Q_{0}}
 \subseteq \bigcup_{R\in m(S(Q_{0}))} \frac{4}{\rho} B_{R'}\]
 and thus
\begin{multline}\label{e:sigmaq-E}
 \cH^{d}_{\infty}\ps{ (\Sigma^{Q_{0}}\backslash E)\cap \frac{c_{0}}{2} B_{Q_{0}}}\\
 \lec \sum_{R\in m(S(Q_{0}))}\ell(R')^{d}
 \lec \largetheta{E}^{d,A,\ve}(B_{Q_{0}})<\delta \ell(Q_{0})^{d}.
 \end{multline}

 {

 Since $\cX_{0}=\{0\}$, $\Sigma_{0}^{Q_{0}}=P_{Q_{0}}$.  }
 By \eqn{sky-y}, $\dist(y,\Sigma^{Q_{0}})\lec \ve $ for all $y\in P_{Q_{0}}$, and since $0=x_{Q_{0}}\in P_{Q_{0}}$, this means $\dist(0,\Sigma^{Q_{0}})\lec \ve <\frac{\ve}{\rho}\ell(Q_{0})\lec \ve $. Thus, for $\ve>0$ small enough, if $x_{Q_{0}}'\in \Sigma^{Q_{0}}$ is closest to $0$,
 \[
 \cH^{d}_{\infty}(\Sigma^{Q_{0}}\cap c_{0}B_{Q_{0}})
 \geq  \cH^{d}_{\infty}\ps{\Sigma^{Q_{0}}\cap B\ps{x_{Q_{0}}',\frac{c_{0}}{2} \ell(Q_{0})}}
\stackrel{\eqn{sigmalr}}{ \gec } \ell(Q_{0})^{d} \gec 1.\]
so for $\ve>0$ small, this and \eqn{sigmaq-E} imply 
\[
\cH^{d}(E\cap B(0,1))=\cH^{d}(Q_{0})
\geq \cH^{d}_{\infty}(\Sigma^{Q_{0}}\cap E \cap c_{0}B_{Q_{0}})\gec 1.\]

\end{proof}

Combining Propositions \ref{p:thmapt1} and \ref{p:diamEC} gives Theorem \ref{t:thmii}.

\section{Proof of Theorem \ref{t:thmiii} - part 1}\label{s:pf-of-thmiii}

In this Section we prove half of  Theorem \ref{t:thmiii} , which is summarized in the following proposition.

\begin{proposition}\label{p:thmiii1case}
With the assumptions of Theorem \ref{t:thmiii}, we have 
\begin{equation}\label{e:thmiii-upper-bound1}
 \cH^{d}(E\cap B(0,1)) \leq C(A,d,n,c)\ps{1+\sum_{k\geq 0}\sum_{B \in \cX_{k}\atop x_{B}\in B(0,1)} \beta_{E}^{d,1}(AB)^{2}r_{B}^{d}}.
 \end{equation}
 Furthermore, if the right hand side of \eqref{e:thmiii-upper-bound1} is finite then we have that $E$ is $d$-rectifiable.
\end{proposition}

Let $\epsilon$ be sufficiently small for the applicaiton of Theorem \ref{t:DT}.
Let $\cD$ be the cubes from \Theorem{Christ} for $E\cap B(0,1)$ so that $\cX_{0}=\{0\}$. Again, in this way, all cubes are contained in $B(0,1)$ and $Q_{0}=B(0,1)\cap E$.

We wish to show that for $M$ sufficiently large (we may fix $M=10^5$ to avoid ambiguity), 
\begin{equation}\label{e:10-upper-bound}
\cH^d(Q_{0})\leq \cH^d\ps{\bigcup_i \Sigma_{S^i}} \leq  C_{\epsilon, n}
\ps{1+		\sum_{Q\subseteq Q_{0}} \beta_E^{d,1}(MB_Q)^{2} \ell(Q)^d}, 
\end{equation}
which $A$ large enough, gives \eqref{e:thmiii-upper-bound}.

We recall a result from \cite{DT12}.

\begin{theorem}[David, Toro {\cite[Theorem 2.5]{DT12}}]\label{DTtheorem2.5}
With the notation and assumptions of \Theorem{DT}, assume additionally that for some $M<\infty$ that
\[
\sum_{k\geq 0} \ve_{k}'(f_{k}(z))^2\leq M_1 \mbox{ for } z\in \Sigma_{0}\]
where 
\begin{multline*}
\ve_{k}'(z)=\sup\{d_{x_{i,l},100r_{l}}(P_{jk},P_{il}): j\in J_{k}, |l-k|\leq 2, i\in J_{l}, \\
x\in 10 B_{jk}\cap 11 B_{il}\}.
\end{multline*}
Then $f=\lim f_{N}=\lim_{N\rightarrow \infty} \sigma_{N}\circ\cdots \circ \sigma_{0}:\Sigma_{0}\rightarrow \Sigma$ is $C(M_1,n)$-bi-Lipschitz.
\end{theorem}

For each $Q\subseteq Q_{0}$, we define stopping-times $S^{Q}$ as follows: we starting by adding $Q$ to $S^{Q}$ and inductively on each descendant $R$ from largest to small, we add $R$ to $S^{Q}$ if 
\begin{enumerate}
\item $R^{(1)}\in S^{Q}$ or $R=Q$,
\item every sibling $R'$ satisfies
\[
\sum_{R'\subset T\subseteq Q} \beta_{E}^{d,1}(MB_{T})^{2}<\epsilon^{2}.\] 
\end{enumerate}
Now, using the definition of $S^{Q}$, we break up 
\[
\cD(Q_{0}):= \{Q\in \cD:Q\subseteq Q_{0}\}\]
  into a collection of stopping-times $\cS$ with 
$  \bigcup \cS=\cD(Q_{0})$.  Start with $\cS=\emptyset$. First place $S^{Q_0}$ in $\cS$. Now, if $S$ has been added to $\cS$, and $Q$ is a child of a cube in $m(S)$,  add $S^Q$ to $\cS$. Repeat this indefinitely.
This gives us 
$\cS$ with $ \bigcup \cS=\cD$. Note that if $\beta_{E}^{d,1}(MB_{Q})\geq \ve$, then $S^{Q}=\{Q\}$. We enumerate the regions $S\in \cS$ which are {\em not} singletons by $\{S^i\}$.

Our plan is that each $S^i$ will correspond to a surface $\Sigma_{S^i}$, whcih will be obtained from Theorem \ref{t:DT} and the will have the bilipschitz estimates of \cite[Theorem 2.5]{DT12} (Theorem \ref{DTtheorem2.5} above).  This will use the lemmas  from Section \ref{s:betalemmas}.

Let us fix such an $S^i$.

\begin{lemma}
There is a surface $\Sigma_{S^i}$ such that  
\begin{equation}\label{l:11-distRSigma}
\dist(R,\Sigma_{S^i})\lesssim_{n} \ell(R) 
\end{equation}
for each $R\in S^i$, and 
$$\cH^d(\Sigma_{S^i})\lesssim_{\epsilon,  n} \ell(Q(S^i))^d$$
\end{lemma}
\begin{proof}
We will use Theorem \ref{DTtheorem2.5}  stated above, which is building upon Theorem \ref{t:DT}.
Fix $M=\Lambda=10^5$.
For $k\geq 0$ an integer, let $s(k)$ be such that $5\rho^{s(k)}\leq r_{k}< 5\rho^{s(k)-1}$. 
Let $L_Q$ be a plane so that 
$\beta_{E}^{d,1}(MB_{Q},L_{Q})=\beta_{E}^{d,1}(MB_{Q})$.
Let $x'_Q\in L_{Q}\cap MB_{Q}$ be the closest point to $x_Q$.
Associate $L_Q$ to $x'_Q$ and the scale $r_k$.

Consider a maximal  $r_k$ net for
$\cC_k=\{x'_Q:Q\in \cD_{s(k)}\cap S^i\}$. 
For these maximal nets,  
Lemma \ref{l:PQPR} guarantees that the planes $L_Q$ satisfy the assumptions of  Theorem \ref{t:DT}.
with $\epsilon_k(x)\lesssim \beta_{E}^{d,1}(10^5B_Q)$, when $x\in Q \in \cC_k$.

We now wish to apply 
Theorem \ref{DTtheorem2.5}. To this end, 
let $x=f(z)\in \Sigma_{S^i}$. By \eqn{sky-y}, $x_{N}:=f_{N}(z)$ is a Cauchy sequence in $N$ and 
\[|x-x_{N}|\lec \sum_{k\geq N}\ve_{k}(x_{k}) r_{k}
\lec\ve  \sum_{k\geq N} r_{k}\lec \ve r_{N}.
\]
Thus, for $\ve>0$ small enough, $x_{N}\in B(x,2r_{k})$.  
We then have 
$\ve'_k(x)\lesssim \beta_{E}^{d,1}(10^5B_Q)$, when $x\in Q \in \cC_k$.

Theorem \ref{DTtheorem2.5} now assures that $\Sigma_{S^i}$ is a bi-Lipchitz surface with constant depending on $n,\epsilon$.
Finally, note that \eqref{l:11-distRSigma} 
follows from  
\eqref{e:ytosigma} coupled with \eqref{e:skCloseInBall} and $\epsilon_k\lesssim_n\epsilon$.
\end{proof}

Let ${min\cS}$ be the 
minimal cubes of all $S\in \cS$. In particular, for any $R\in {min\cS}$, there is $S(R)\in \cS$ such that $R$ is a minimal cube in $S(R)$.
Furthermore, any $S\in \cS$ has $Q(S)$ either equal to $Q_0$, or a cube in ${min\cS}$.

\begin{lemma}\label{l:sumS_Jlarge}
\begin{equation}\label{e:summins}
\sum_{R\in {min\cS}} \ell(R)^d\lesssim_{\epsilon, n} 
 \sum_{Q\in \cD} \beta_{E}^{d,1}(MB_Q)^{2}\ell(Q)^d
 \end{equation}
\end{lemma}
\begin{proof}
Each $R\in {min\cS}$ is a minimal cube in an $S\in\cS$, which we call $S(R)$. 
These are disjoint for different $R$. Let $R'$ be a child of $R$ for which
$$\sum_{Q\in S(R), Q\supset R'} \beta_E^{d,1}(MB_Q)^{2}\geq \epsilon^{2}.$$
If $\beta_{E}^{d,1}(MB_{R'})^{2} <\ve^{2}/2$, then this implies 
$$\sum_{Q\in S(R), Q\supset R} \beta_E^{d,1} (MB_Q)^{2}\geq \epsilon^{2}/2.$$
Otherwise, if $\beta_{E}^{d,1}(MB_{R'})^{2}\geq \ve^{2}/2$, by monotonicity this implies 
\[
\sum_{Q\in S(R)\atop Q\supseteq R}  \beta_E^{d,1}(MB_Q)^{2}
\geq \beta_{E}^{d,1}(MB_{R})^{2}\gec\beta_{E}^{d,1}(MB_{R'})^{2} \geq \ve^{2}/2.\]
In any case, we have 
\[
\sum_{Q\in S(R)\atop Q\supseteq R}  \beta_E^{d,1}(MB_Q)^{2} \gec \ve^{2} \mbox{ for all }R\in m(S).\]
Thus
\begin{eqnarray*}
\sum_{R\in {min\cS}}\ell(R)^d  &\leq&
\frac1{\epsilon^{2}}
	\sum_{R\in {min\cS}} \ell(R)^d \sum_{Q\in S(R), Q\supset R}\beta_E^{d,1}(MB_Q)^{2}\\
&\leq&
\frac1{\epsilon^{2}}
	\sum_{S\in \cS}\sum_{Q\in S} \beta_E^{d,1}(MB_Q)^{2}\sum_{R\in {min\cS}\atop Q\in S(R), Q\supset R} \ell(R)^d	 
\end{eqnarray*}
Now observe that for a fixed cube $Q\in \cD$, the cubes $R\in {min\cS}$ such that $Q\in S(R)$
are disjoint, and $\cH^{d}(c_{0} B_{R}\cap \Sigma_{S(R)})\gec \ell(R)^{d}$ for $\ve>0$ small enough, since for $\ve>0$ small, $c_{0}B_{R}$ will intersect a large portion of $\Sigma_{S(R)}$. (If $S$ is a singleton, then $\Sigma_{S(R)}$ was not defined, but then $S=\{Q\}=\{R\}$). Thus,
\begin{align*}
\sum_{R\in {min\cS}\atop Q\in S(R), Q\supset R} \ell(R)^d
& \lec_{d} \sum_{R min \cS \atop Q\in S(R),Q\supset R} \cH^{d}(c_{0} B_{R}\cap \Sigma_{S(R)})\\
&
 \leq \cH^{d}(B_{Q}\cap \Sigma_{S(R)}) \lesssim_{\epsilon, n} \ell(Q)^d
 \end{align*}
where the constant comes form the constant in \cite[Theorem 2.5]{DT12}.
Putting this all together, we have
\begin{eqnarray*}
\sum_{R\in {min\cS}}\ell(R)^d  &\leq&
\frac1{\epsilon^{2}}
		\sum_{S\in \cS}\sum_{Q\in S} \beta_E^{d,1}(MB_Q)^{2}\sum_{R\in {min\cS}\atop Q\in S(R), Q\supset R} \ell(R)^d\\
&\lesssim_{\epsilon, d, n}&
		\sum_{Q\in S} \beta_E^{d,1}(MB_Q)^{2} \ell(Q)^d	 
\end{eqnarray*}
as desired.
\end{proof}

Let $G$ be the points in $E$ which are either stopped a finite number of times, or never stopped, i.e.
$$G:=\{x\in E\cap B(0,1): \sum\beta_{E}^{d,1}(MB_Q)^2\chi_Q(x)<\infty\}.$$
{ Then, by \eqref{l:11-distRSigma}, $G\subset  \cup_i \Sigma_{S^i}$, and we have from the above lemmas the following corollary.}
\begin{corollary}\label{c:Si-total-area}
$$\cH^d(G)\leq \cH^d\ps{\bigcup_i \Sigma_{S^i}}\leq  C_{\epsilon, n}
\ps{1+		\sum_{Q\in \cD} \beta_E^{d,1}(MB_Q)^{2} \ell(Q)^d}.$$	 
\end{corollary}

Let $E'= E\cap B(0,1)\setminus G$. Then
$x\in E'$ implies  $x=\lim x_{i_j}$ with $x_{i_j}\in S^{i_j}$, and $Q(S^{i_{j+1}})$.
More is true: 
\begin{lemma}
Let $x\in E'$.
For  $r>0$, there is a cube $Q\in {min\cS}$ with $\ell(Q)<r$ such that 
$x\in Q$.   
\end{lemma}
\begin{proof}
This lemma is simply a restatement of the fact that $E'$ is the set of points that were stopped infinitely many times, and the scale must decrease to 0 by the properties of $\{Q\in \cD: Q\ni x\}$.
\end{proof}

\begin{corollary}\label{c:EprimeSize}
If the right hand side of \eqref{e:10-upper-bound} is finite, then
 $\cH^d(E')=0$.
 \end{corollary}
\begin{proof}
Lemma \ref{l:sumS_Jlarge} together with the previous lemma first give that
$$\cH^d(E')\lesssim_{\epsilon,  n} 
 \sum_{Q\in \cD} \beta_E^{d,1}(MB_Q)^{2} \ell(Q)^d.$$
 However, since each point $x\in E'$ is covered by an infinite number of cubes  in ${min\cS}$
we have that Lemma \ref{l:sumS_Jlarge} gives $\cH^d(E')=0$.
\end{proof}

We now note that $E\cap B(0,1)\subset E' \cup \bigcup_i \Sigma_{S^i}$.
\begin{corollary}\label{c:rect}
If the right hand side of \eqref{e:10-upper-bound} is finite then $E$ is $d$-rectifiable.
\end{corollary}

Proposition \ref{p:thmiii1case}
 follows from Corollary \ref{c:Si-total-area}, Corollary \ref{c:EprimeSize} and Corollary \ref{c:rect}. 

\section{Proof of Theorem \ref{t:thmiii} - part 2}\label{s:pf-of-thmiii-part2}
We continue with the same notation of Section \ref{s:pf-of-thmiii}. 
In particular,
$E\subset \bR^n$,
and $\cD$ be the cubes from \Theorem{Christ} for $E\cap B(0,1)$ so that $\cX_{0}=\{0\}$. 
(again, in this way, all cubes are contained in $B(0,1)$ and $Q_{0}=B(0,1)$.)
\begin{proposition}\label{p:theta-large-bdd}
Suppose  that for all $x\in E$ and $0<r\leq1$ we have 
$\cH^d_\infty(B(x,r)\cap E)>t r^d$. Then
\begin{align*}
\sum\{\ell(Q)^d: Q\in \cD,\ \vartheta_{E}(MB_Q)\geq \epsilon\}
\lesssim_{\epsilon,n,t} 1+
   \sum_\cD \beta_E^{d,1}(2MB_Q)^2\ell(Q)^d
\end{align*}
\end{proposition}
The purpose of this Section is to prove the  Proposition \ref{p:theta-large-bdd}, which together with Proposition \ref{p:thmiii1case},  the properties of the cubes in Theorem \ref{t:Christ}, and Lemma \ref{l:theta-monotone-1} gives Theorem \ref{t:thmiii}.

\begin{lemma}\label{l:theta-monotone-1}
If $B_1\subset B_2$ and $r_i$ is the radius of $B_i$, then 
$\vartheta_E(B_1)r_1\leq \vartheta_E(B_2)r_2$.
\end{lemma}
The previous lemma is not hard to show using definitions.


The following lemma is a a simple corollary of  Lemma \ref{l:betainfbeta}.
\begin{lemma}\label{l:content-vs-infty}
Suppose $E$ is such that for all $x\in E$ and $r\in(0,R)$ we have 
$\cH^d_\infty(B(x,r)\cap E)>t r^d$.
Then there is a $c_t>0$ such that 
if $\beta_{\infty,E}(2B)>\delta$ then
$\beta_E^{d,1}(B)^{2} \geq c_t \delta^{\frac{1}{d+1}}$ 
\end{lemma}
Lemma \ref{l:content-vs-infty} immediately gives the following.
\begin{corollary}\label{c:large-beta}
Under the same assumptions of Lemma \ref{l:content-vs-infty},
there is a $c_t>0$ such that 
\begin{multline*}
\sum\{\ell(Q)^d: Q\in \cD,\ \beta_{\infty,E}(MB_Q)\geq \delta\}\\
\leq 
\sum\{\ell(Q)^d: Q\in \cD,\ \beta_{E}^{d,1}(2MB_Q)\geq c_t \delta^{\frac{1}{d+1}}\}\\
\leq
\frac1{(c_t\delta^{\frac{1}{d+1}})^2}
   \sum_{Q\in \cD} \beta_E^{d,1}(2MB_Q)^2\ell(Q)^d
\end{multline*}
\end{corollary}
 
Recall that for $S\in \cS$, $\Sigma_S$ is a bilipschitz topological $d$-plane with bi-Lipschitz constant depending only on $n,\epsilon$.

The following lemma is a special case of \cite[Theorem 2.4, page.32]{DS93}, which holds more generally for uniformly rectifiable sets, but we will apply it to the case of bi-Lipschitz surfaces (namely, the $\Sigma_{S}$). We note that  it is possible to make use of the fact $\Sigma_S$ is a bi-Lipschitz image of a plane to prove it directly, but omit the proof here. In \cite{DS93}, the notaton $b\beta$ ({\it bilateral $\beta$}) is used for $\vartheta$.
\begin{lemma}\label{l:large-theta-sigma}
$$\sum\{\ell(Q)^d: Q\in S,\ \vartheta_{\Sigma_S}(MB_Q)\geq \delta\}\leq C_\delta \cH^d(\Sigma_S)\lesssim_{\delta, \epsilon, n}\ell(Q(S))^d$$ 
\end{lemma}

For a ball $B$ centered on $E$ and $d$-plane $L$ define
$$\eta_{E}(B,L):=\frac{1}{r_{B}}\sup_{x\in L\cap B}\dist(x,E)$$
and
$\eta_{E}(B):=\inf_{L}\eta_{E}(B,L)$ 
where $L$ ranges over all affine $d$-planes. Thus $\eta_E(\cdot)\leq \vartheta_E(\cdot)$.

\begin{lemma}\label{l:beta-and-eta-is-theta}
There is a constant $C_\vartheta>0$ (independent of dimension) such that the following holds.
Let $B$ be a ball. Suppose $\beta_{E,\infty}(B)<\delta$, and  $\eta_{E}(B)<\delta$.
Then $\vartheta_{E}(B)<C_\vartheta\delta$.
\end{lemma}
\begin{proof}
Without loss of generality suppose $B$ is the unit ball.

Let $L_\eta$ a  $d$-plane giving $\eta_{E}(B)$ and
Let $L_\beta$ a  $d$-plane giving $\beta_{\infty,E}(B)$.
Let $\cN_\delta(F)$ denote the $\delta$ neighborhood of a set $F$.
Then
\begin{equation}\label{e:appricot-1}
\cN_\delta(E\cap B)\supset  L_\eta\cap \frac12 B.
\end{equation}
Let $\pi_{L_\beta} $ be the projection  to $L_\beta$.  We have that $|\pi_{L_\beta}(x)-x|\leq \delta r_B$ for $x\in E\cap B$, and so 
\begin{equation}\label{e:appricot-2}
\cN_{\delta}(E\cap B) \subset \cN_{2\delta}(L_\beta\cap \frac 12 B).
\end{equation}
Combining  \eqref{e:appricot-1} and\eqref{e:appricot-2}, one can show (using Lemma \ref{l:ATlemma}, for example) that
$d_B(L_\eta, L_\beta)\lesssim \delta$. 
%
Hence, for $x\in E\cap B$, $\pi_{L_{\beta}}(x)\in B\cap L_{\beta}$
\[
\dist(x,L_{\eta})
\leq |x-\pi_{L_{\beta}}(x)|+\dist(\pi_{L_{\beta}}(x),L_{\eta})
\lec \delta r_{B}.\]
Thus, $\beta_{E,\infty}(B,L_{\eta})\lec \delta$, and since we already have $\eta_{E}(B,L_{\eta})<\delta$, the lemma follows.

\end{proof}

\begin{proposition}\label{p:eta-E-large-eta-sigma-small} 
Let $S\in \cS$ be given. 
Set
$$S_{\delta}=\{Q\in S, \ \eta_{E}(MB_Q)>3\delta,\ \eta_{\Sigma_S}(MB_Q)<\delta\}.$$
Then 
$$\sum_{Q\in S_{\delta}}\ell(Q)^d\lesssim_{\delta, \epsilon,n} \ell(Q(S))^d.$$
\end{proposition}

\begin{proof}
Assume without loss of generality that $S$ is not a singleton, and thus we have defined $\Sigma_S$.
Let $Lip=Lip(n,\epsilon)$ be the bi-Lipschitz constant of the implicit map from $\bR^d$ to $\Sigma_S$.   
We first note that 
$Q\in S_{\delta}$ satisfies
\begin{equation}\label{e:banana-1}
\frac1{M\ell(Q)}\sup_{x\in \Sigma_S\cap MB_Q}\dist(x, E)>\delta\,. 
\end{equation}
Indeed suppose instead that all $x\in \Sigma_S\cap MB_Q$ have 
$\dist(x, E)\leq\delta M \ell(Q)$.
Then take $L_Q$ such that $\delta>\eta_{\Sigma_S}(MB_Q)=\eta_{\Sigma_S}(MB_Q, L_Q)$ 
and $p\in L_Q\cap MB_Q$. Now, by our contrapositive assumption, 
\begin{align*}
\dist(p,E)\leq \dist(p,\Sigma_S \cap MB_Q) + \delta M\ell(Q) \leq\\
 (\eta_{\Sigma_S}(MB_Q) +\delta)M\ell(Q)\leq 2\delta M\ell(Q).
 \end{align*}
This is a contradiction to 
$ \eta_{E}(MB_Q)>3\delta$.


{
Now that we have \eqref{e:banana-1} for all  $Q\in S_{\delta}$, there is $y_{Q}\in MB_Q$  so that $\dist(y_Q,E)\geq \delta M\ell(Q)$. Let 
\[
b_Q=B(y_{Q},\delta M\ell(Q)/2).
\]
We claim that for $Q\in S_{\delta}$, the balls $b_{Q}$ have bounded overlap, that is,
$$\#\{Q\in S_{\delta}: b_Q\ni x\}\leq C(\delta, d, Lip).$$
Indeed, suppose $x\in b_{Q}\cap b_{R}$ for some $Q,R\in S_{\delta}$. Then 
\[
2 M\ell(Q)\geq |y_{Q}-x|\geq \dist(x,E)\geq \dist(b_R,E)\geq \delta M\ell(R)/2\]
and so $\ell(Q)\gec \ell(R)$, and reversing the rolls of $Q$ and $R$ gives $\ell(Q)\sim  \ell(R)$. Thus, all cubes $Q$ with $x\in b_Q$ have comparable sidelengths, so in particular, if $Q$ is another such cube and $Q\in \cD_{k}$ and $R\in \cD_{\ell}$, then $|k-\ell|\lec 1$. Since each such $Q$ is distance at most $M\ell(Q)$ from $x$ and the centers of each $\cD_{k}$ are $\rho^{k}$-separated, these facts imply there are boundedly many such cubes (with constant depending on $d$ and $M$). This proves the claim. 
}

Thus, with implicit constants depending on $\delta, d,Lip$
\begin{align*}
\sum_{Q\in S_{\delta}}\ell(Q)^d\lesssim
\sum_{Q\in S_{\delta}}\cH^d(b_Q)\lesssim \cH^d(\Sigma_S)\lesssim \ell(Q(S))^d 
\end{align*}
giving  the proposition.
\end{proof}

\begin{proof}[Proof of Proposition \ref{p:theta-large-bdd}.]
We use Lemma \ref{l:beta-and-eta-is-theta}  to write
\begin{align*}
\{Q\in \cD:& \vartheta_{E}(MB_Q)\geq \epsilon\} \subset \\
&\{Q\in \cD: \eta_{E}(MB_Q)\geq \epsilon/C_\vartheta\}
\cup
\{Q\in \cD: \beta_{\infty,E}(MB_Q)\geq \epsilon/C_\vartheta\}\\
&=\cD_\eta \cup \cD_\beta,
\end{align*}
where $\cD_\eta$ and $\cD_\beta$ are the two collections of cubes on the penultimate line.
We use Corollary  \ref{c:large-beta} with $\delta=\epsilon/C_\vartheta$ for the sum over $\cD_\beta$ i.e.
$$\sum_{Q\in \cD_\beta} \ell(Q)\lesssim_{t,\epsilon}    \sum_\cD \beta_E(2MB_Q)^2\ell(Q)^d.$$

We write $\cD_\eta=\cup_{S\in \cS} S\cap\cD_\eta$
For each $S\in \cS$,  Proposition \ref{p:eta-E-large-eta-sigma-small} and Lemma \ref{l:large-theta-sigma}, both  with 
$\delta=\epsilon/(3C_\vartheta)$, 
give control over the sum for $S\cap\cD_\eta$ i.e.
$$
\sum_{Q\in S\cap\cD_\eta}\ell(Q)^d\lesssim_{\epsilon,n} \ell(Q(S))^d.
$$
Summing over al $S\in \cS$ is then controlled by Lemma \ref{l:sumS_Jlarge}.
This completes the proof of Proposition \ref{p:theta-large-bdd}.
\end{proof}

\section{Appendix}\label{s:appendix}

\begin{proof}[Proof of Lemma \ref{l:intineq}]
 Let 
 \[
 S_{j}=\{x:f_{j}(x)=\max\{f_{i}(x)\}>0\}.\] 
 Then $\bigcup \supp f_{i}=\bigcup S_{j}$. If $x\in S_{j}$, then there are at most $C$ many indices $i$ for which $f_{i}>0$, and so $\sum f_{i}(x)\leq Cf_{j}(x)$. Therefore,

\begin{align*}
\int \ps{ \sum f_{i}}^{p} d\cH^{d}_{\infty}
& =\int \cH^{d}_{\infty}\ps{\ck{\sum f_{i}>\lambda}}\lambda^{p-1}d\lambda \\
& \leq \sum_{j} \int \cH^{d}_{\infty}\ps{S_{j}\cap \ck{\sum f_{i}>\lambda}}\lambda^{p-1}d\lambda \\
& \leq \sum_{j}\int \cH^{d}_{\infty} \ps{ \{Cf_{j} >\lambda\}} \lambda^{p-1}d\lambda\\
& = C^{p} \sum_{j} \int fd\cH^{d}_{\infty}.
\end{align*}
\end{proof}

\begin{proof}[Proof of Lemma \ref{l:contentlim}]
Extend $f$ to be continuous on all of $\bR^{n}$. Since $E$ is compact and $f$ is continuous, the set
\[
E^{t}:=\{x\in E: f(x)\geq t\}\]
is also compact for each $t>0$. It is not hard to show (using compactness) that if

\[
E_{j}^{t}:=\{x\in E_{j}: f(x)\geq t\}\]
then, since $\bigcap E_{j}^{t}= E^{t}$, we have
\[
\lim_{j\rightarrow\infty} \cH_{\infty}^{d}(E_{j}^{t})= \cH^{d}_{\infty}(E^{t}) \;\; \mbox{for } t>0.
\]
Hence, by the monotone convergence theorem,
\begin{equation}\label{e:etlim}
\int_{0}^{\infty} \cH^{d}(E^{t})dt
=\lim_{j\rightarrow\infty}\int_{0}^{\infty}  \cH^{d}(E_{j}^{t})dt.
\end{equation}

Now we observe that, for any function $g$ and $F$ any set,
\begin{align}\label{e:atleastt1}
\int_{F} g d\cH^{d}_{\infty} 
&  =\int_{F} \cH^{d}_{\infty} (\{x\in F: g(x)>t\})dt \notag \\
& \leq \int_{F} \cH^{d}_{\infty} (\{x\in F: g(x)\geq t\})dt
\end{align}

and 

\begin{align}
\int_{0}^{\infty}  & \cH^{d}_{\infty} (\{x\in F: g(x)\geq t\})dt \notag  \\
&  \leq \inf_{\alpha\in (0,1)} \int_{0}^{\infty}\cH^{d}_{\infty} (\{x\in F: g(x)> \alpha t\})dt \notag \\
& =\inf_{\alpha\in (0,1)} \alpha^{-1}  \int_{0}^{\infty} \cH^{d}_{\infty} (\{x\in F: g(x)>  t\})dt
=\int g \cH^{d}_{\infty}
\label{e:atleastt2}
\end{align}

Combining \eqn{atleastt1} and \eqn{atleastt2} gives

\begin{equation}\label{e:atleastt3}
\int_{F} g d\cH^{d}_{\infty}  = \int_{0}^{\infty} \cH^{d}_{\infty} (\{x\in F: g(x)\geq t\})dt
\end{equation}

Thus, applying this to $g=f$ and $F$ equal to either $E$ or $E_{j}$,
\begin{multline*}
\int_{E} f d\cH^{d}_{\infty} 
 \stackrel{\eqn{atleastt3}}{=} 
\int_{0}^{\infty} \cH^{d}(E^{t})dt
\stackrel{\eqn{etetj}}{=} \lim_{j\rightarrow\infty}\int_{0}^{\infty}  \cH^{d}(E_{j}^{t})dt \\
 \stackrel{\eqn{atleastt3}}{\sim} \lim_{j\rightarrow\infty}\int_{E_{j}}f d\cH^{d}_{\infty}
\end{multline*}

\end{proof}

\begin{proof}[Proof of Lemma \ref{l:jensens}]
First assume that $E$ is open. Without loss of generality, we may assume 
$\cH^{d}_\infty(E)=1$.
Note that $f\one_{E}$ is still lower semicontinuous since $E$ is open.  By the corollary on page 118 of \cite{Ad88}, for $f\geq 0$ lower semicontonuous,

\begin{equation}\label{e:adams}
\int fd\cH^{d}_{\infty} \sim_{n} \sup\ck{\int fd\mu: \mu\in L^{1,d}(\bR^{n}),\;\; ||\mu||=1}.
\end{equation}
Where $L^{1,d}(\bR^{n})$ is the Morrey space of Radon measures with the norm
\[
||\mu||=\sup_{x\in \bR^{n} \atop r>0} |\mu|(B(x,r))r^{-d}.\]
Note that if $A_{i}$ is a cover of $E$, then each $A_{i}$ is contained in a ball of radius $\diam A_{i}$, and so
\[
\mu(E)\leq \sum \mu(A_{i})\leq \sum ||\mu||(\diam A_{i})^{d}\]
and infimizing over all such covers gives 
\[
\mu(E)\leq ||\mu||\cH^{d}_{\infty}(E)=||\mu||=1.\]
Thus, if $\frac{1}{p}+\frac{1}{q}=1$,
\begin{align*}
\int fd\mu 
 & \leq \ps{\int f^{p}d\mu}^{\frac{1}{p}}\ps{\int_{E} d\mu}^{\frac{1}{q}}
\leq \ps{||\mu||\int f^{p}d\cH^{d}_{\infty}}^{\frac{1}{p}} \mu(E)^{\frac{1}{q}}\\
& = \ps{\int f^{p}d\cH^{d}_{\infty}}^{\frac{1}{p}}.
\end{align*}

Supremizing over all $\mu$ and using \eqn{adams} once more gives \eqn{jensens}.

Now assume $E$ be a compact set and $f$ a continuous function on $E$. Again, we may assume $\cH^{d}_{\infty}(E)=1$. Extend $f$ to a continuous function on all of $\bR^{n}$ and let

\[
E_{j}=\{x\in \bR^{n}:\dist(x,E)<j^{-1}\}.\]
Since we know \eqn{jensens} for open sets, we may apply it to the sets $E_{j}$. Since $f$ and $f^{p}$ are  continuous, and since the $E_{j}$ are open, contain $E$, and converge to $E$ in the Hausdorff metric, we may use \Lemma{contentlim} and get

\begin{align*}
\int_{E} fd\cH^{d}_{\infty} 
& \stackrel{\eqn{etetj}}{=}  \lim_{j\rightarrow\infty} \int_{E_{j}} f d\cH^{d}_{\infty}
 \stackrel{\eqn{jensens}}{\lec} \lim_{j\rightarrow\infty} \ps{ \int_{E_{j}}f^{p} d\cH^{d}_{\infty}}^{\frac{1}{p}}\\
&  \stackrel{\eqn{etetj}}{=} \ps{ \int_{E}f^{p} d\cH^{d}_{\infty}}^{\frac{1}{p}}.
\end{align*}

 \end{proof}

 \bigskip

%
%
%
%
 
\bigskip

\begin{proof}[Proof of Lemma \ref{l:tech-integrating-sum}]
Let $\alpha'=1-\alpha$. We will first prove that for all $\lambda>0$
\begin{multline}\label{e:doubling}
\cH^{d}_{\infty} \ps{\ck{x\in F_{1}:\sum_{j\in \cX}\one_{B_{j}'}f(z_{j})>\lambda}}\\
\lec \cH^{d}_{\infty} \ps{\ck{x\in F_{1}:\sum_{j\in \cX}\one_{\alpha'B_{j}'}f(z_{j})>\lambda}}.
\end{multline}
Let 
\[
A=\ck{x\in F_{2}:\sum_{j\in \cX}\one_{\alpha'B_{j}'}f(z_{j})>\lambda}.\]
Let $\cI$ be a collection of balls covering $A$ so that
\begin{equation}\label{e:hda}
\cH^{d}_{\infty}(A)\sim_{d} \sum_{B\in \cI} (2r_{B})^{d}.
\end{equation}
Let 
\[
\cX^{\lambda}=\{j\in \cX: f(z_{j})>\lambda\}.\]
Note that as $|z_{j}-z_{j}'|<\alpha r_{B_{j}}$, 
\[
\alpha' B_{j}'= B(z_{j}',\alpha' r_{B_{j}})\subseteq  B(z_{j},(\alpha' +\alpha)r_{B_{j}})=B_{j}.\]
Thus, since the balls $B_{j}$ are disjoint for $j\in \cX$, so are the balls $B_{j}'$, and hence we have 
\[A=F_{2}\cap \bigcup_{j\in \cX^{\lambda}} \alpha' B_{j}'.\]

We will define a new collection of balls as the limit of a sequence of collections $\cI(j)$ which we define inductively as follows. Assume $\cX^{\lambda}=\bN$ and set $\cI(0)=\emptyset$. Now assume that for some $j>0$, $\cI(j-1)$ has already been defined. Let 
\[
\cI_{j}=\ck{B\in \cI: B\cap \alpha'B_{j}'\neq\emptyset}.\]
\begin{enumerate}
\item If there is $B\in \cI_{j}$ for which $r_{B}\geq  \frac{\alpha}{2} r_{B_{j}}$, we let 
\[
\cI(j)=\cI(j-1)\cup \ck{\frac{4}{\alpha}B}\]
and note that $\frac{4}{\alpha} B\supseteq B_{j}$ since $B\cap B_{j}\supseteq B\cap \alpha'B_{j}'\neq\emptyset$ and $r_{B_{J}}\leq \frac{2}{\alpha} r_{B}$ by assumption.
\item If $r_{B}<\frac{\alpha}{2} r_{B_{j}}$ for all $B\in \cI_{j}$, we let 
\[\cI(j)=\cI(j-1)\cup \{B_{j}\}.\]
Note that in this case, $B\subseteq B_{j}$ for all $B\in \cI_{j}$.
\end{enumerate}
We let $\cI'=\bigcup \cI(j)$. In this way, every $B_{j}$ is contained in a ball from $\cI'$, that is,
\[
\bigcup_{j\in \cX^{\lambda}}B_{j} \subseteq \bigcup_{B\in \cI'}B.\]
For $i=1,2$, let $\cI_{i}'$ be those balls in $\cI_{i}$ added in case $i$ and let $\cX^{\lambda,2}$ be those $j$ for which case 2 happened. Since the $B_{j}$ are disjoint, and $B_{j}\supseteq \alpha'B_{j}'\cap F_{2}$,
\begin{multline}\label{e:sumrb1}
\sum_{B\in \cI} r_{B}^{d}\geq \sum_{j\in \cX^{\lambda,2} }\sum_{B\in \cI(j)} r_{B}^{d}
\gec \sum_{j\in \cX^{\lambda,2} }  \cH^{d}_{\infty} \ps{\alpha'B_{j}'\cap F_{2}}
{\gec}_{\alpha}  \sum_{j\in \cX^{\lambda,2} }  r_{B_{j}}^{d}\\
\gec \sum_{B\in \cI_{2}'} r_{B}^{d}.
\end{multline}
Also,
\begin{equation}\label{e:sumrb2}
\sum_{B\in \cI}  r_{B}^{d} \geq \ps{\frac{\alpha}{4}}^{d} \sum_{B\in \cI_{1}'} r_{B}^{d}\end{equation}
and hence
\begin{equation}\label{e:sumrb3}
\cH^{d}_{\infty} (A) \stackrel{\eqn{hda}}{\gec} \sum_{B\in \cI} r_{B}^{d} \stackrel{\eqn{sumrb1}\atop \eqn{sumrb2}}{\gec} \sum_{B\in \cI'} r_{B}^{d}
\geq \cH^{d}_{\infty}\ps{\bigcup_{j\in \cX^{\lambda}} B_{j}}.
\end{equation}

Let 
\[
A'=\ck{x\in F_{1}:\sum_{j\in \cX}\one_{B_{j}}f(z_{j})>\lambda}.\]
If $x \in A'$, since the $B_{j}$ are disjoint for $j\in \cX$, there is $x\in B_{j}$ for some $j\in \cX^{\lambda}$, and so $x\in \bigcup_{j\in \cX^{\lambda}}B_{j} $, hence
\[
\cH^{d}_{\infty}(A')
\leq  \cH^{d}_{\infty}\ps{\bigcup_{j\in \cX^{\lambda}} B_{j}}
\stackrel{\eqn{sumrb3}}{\lec} \cH^{d}_{\infty} (A)\]
and this finishes the proof of \eqn{doubling}. Hence \eqn{1<4} follows by integrating \eqn{doubling}.

\end{proof}

\def\cprime{$'$}

\end{document}